\documentclass[12pt,oneside,a4papper]{osajnl}
\usepackage{amsthm}
\usepackage{dsfont}
\usepackage{comment}
\journal{jocn} 

\setboolean{shortarticle}{false}
\title{\centering Nash-Stackelberg controllability for coupled systems of degenerate equations in non-cylindrical domains}

\author[2,3,*]{\centering \textcolor{black}{Alfredo S. Gamboa}}
\author[1]{\textcolor{black}{Juan B. Limaco}}
\author[1]{\textcolor{black}{Luis P. Yapu}}

\affil[1]{Universidade Federal Fluminense, Instituto de Matemática e Estatística, Niterói, Brazil}
\affil[2]{Universidad Privada Boliviana, Departamento de Ciencias Exactas, Cochabamba, Bolivia}
\affil[3]{Universidade Estadual do Rio de Janeiro, Escola Politécnica, Nova Friburgo, Brazil}

\affil[*]{\centering Corresponding author: alfredo.soliz@iprj.uerj.br}

\dates{\centering Compiled \today}

\doi{}

\begin{abstract}
In this paper we investigate the Hierarchical null controllability of a coupled degenerate semilinear parabolic equation in domains which are moving in time.
We show the local null controllability of the semilinear system using Liusternik's inverse function theorem.
Nevertheless, the main difficulty is to adapt a Carleman estimate for the controllability of the linearized otimality system, using a Carleman inequality for degenerate non-autonomous equation obtanied by the authors previously. 
\end{abstract}

\definecolor{mygreen}{RGB}{44,162,67}
\definecolor{mylilas}{RGB}{186,85,211}
\setboolean{displaycopyright}{false}

\def\R{\mathbb{R}}

\newcommand{\dpar}[2]{\frac{\partial #1}{\partial #2}}

\newcommand{\rea}{\mathbb{R}}

\newcommand{\dpd}[2]{\frac{\partial^2 #1}{\partial #2^2}}

\newcommand{\cara}{\mathbb{1}}

\newtheorem{defi}{Definition} %[chapter]
\newtheorem{teo}{Theorem} %[chapter]
\newtheorem{propo}{Proposition} %[chapter]
\newtheorem{lema}{Lemma} %[chapter]
\newtheorem{coro}{Corollary} %[chapter]

\setboolean{displaycopyright}{false}

\begin{document}
	
	\maketitle

    \section{Introduction}
    \label{sec:intro}
	
	Let us consider the open interval $\Omega = (0,1) \subset \mathbb{R}$, with boundary $\Gamma=\partial\Omega=\{0,1\}$. Take $T > 0$ and consider the cylinder $Q = \Omega\times (0, T)$, with lateral boundary $\Sigma=\Gamma\times (0, T)$. Let us consider a family of functions $\{\tau_t\}_{0\leq t\leq T}$ , where for each $t$, $\tau_t$ is a deformation of $\Omega$ into an open bounded interval $\Omega_t$ defined by
	$$\Omega_t=\{ x'\in \rea \ | \ x'=\tau_t(x), \ \ \text{for} \ \ x\in\Omega \}$$
    
	For $t = 0$, $\tau_0$ is the identity map and $\Omega_0$ is identified with $\Omega$. The smooth boundary of  $\Omega_t$ is denoted by $\Gamma_t$. The non-cylindrical domain $\widehat{Q}$ and its lateral boundary $\widehat{\Sigma}$ are defined by 
	$$\widehat{Q}=\bigcup_{0\leq t\leq T}\{\Omega_t\times\{t \}\}, \ \ \ \ \ \ \ \ \ \ \ \ \ \widehat{\Sigma}=\bigcup_{0\leq t\leq T}\{ \Gamma_t\times\{t\}  \},$$
	respectively.
	
	We assume the following regularity on the functions $\tau_t$ for $0\leq t\leq T$:
	\begin{description}
		\item[(R1)] $\tau_t$ is a  $C^2$ diffeomorphism from $\Omega$ to $\Omega_t$,
		\item[(R2)] The map $t \mapsto \tau_t$ is in $C^1([0,T];C^0(\overline{\Omega}))$.
	\end{description}
    
	Thus, we have a diffeomorphism $\tau : Q\rightarrow \widehat{Q}$ defined by
	$$(x,t)\in Q \rightarrow (x',t)\in \widehat{Q}, \quad \text{where} \quad x'=\tau_t(x).$$
	
	We are interested in the controllability of a degenerate semilinear parabolic system defined in the noncylindrical domain $\widehat Q$.
    The strategy of control is hierarchical since the system has one leader control and two followers controls pursuing independent goals. Let us consider the system:
    %\begin{comment}
	%\begin{equation}\label{eq1a}
	%	\begin{cases}
	%		u_{1t}-({a}(x')u_{1x'})_{x'}+c_{11}u_1+c_{12}u_2=\widehat{h}\cara_{_{\widehat{\mathcal{O}}}}+\widehat{v}^1\cara_{_{\widehat{\mathcal{O}}_1}}+\widehat{v}^2\cara_{_{\widehat{\mathcal{O}}_2}}, & \\
	%		u_{2t}-({a}(x')u_{2x'})_{x'}+c_{21}u_1+c_{22}u_2=0 & \ \ \ \text{en} \ \ \ \widehat{Q}\\
	%		u_1(0,t)=u_1(\ell(t),t)=u_2(0,t)=u_2(\ell(t),t)=0, & \ \ \ \text{en} \ \ \ \widehat{\Sigma}\\
	%		u_1(0)=u_0^1(x'), \ u_2(0)=u_2^0(x') & \ \ \ \text{en} \ \ \ \Omega_t
	%	\end{cases}
	%\end{equation}
    %\end{comment}
    \begin{equation}\label{sistemanocilin4}
	\begin{cases}
		u_{1t}-({a}(x')u_{1x'})_{x'}+F_1(u_1,u_2)=\widehat{h}\cara_{_{\widehat{\mathcal{O}}}}+\widehat{v}^1\cara_{_{\widehat{\mathcal{O}}_1}}+\widehat{v}^2\cara_{_{\widehat{\mathcal{O}}_2}}, & \ \ \ \text{in} \ \ \ \widehat{Q},\\
		u_{2t}-({a}(x')u_{2x'})_{x'}+F_2(u_1,u_2)=0, & \ \ \ \text{in} \ \ \ \widehat{Q},\\
		u_1(0,t)=u_1(\ell(t),t)=u_2(0,t)=u_2(\ell(t),t)=0, & \ \ \ \text{on} \ \ \ \widehat{\Sigma},\\
		u_1(0)=u_1^0(x'), \ u_2(0)=u_2^0(x'), & \ \ \ \text{in} \ \ \ \Omega_t,
	\end{cases}
    \end{equation}  
	where $u_1^0$ and $u_2^0$ are the initial data, $\widehat{h}$ is the control of the leader, $\widehat{v}^i$ are controls of the followers and $\cara_{_{A}}$ denotes the characteristic function of the set $A$.
    The set $\widehat{\mathcal{O}}\subset \Omega_t$ denotes the domain of the leader control $\widehat{h}=\widehat{h}(x',t)$, whereas
	$\widehat{\mathcal{O}}_1, \widehat{\mathcal{O}}_2\subset\Omega_t$ denote the domains of the followers $\widehat{v}^1=\widehat{v}^1(x',t)$ and $ \widehat{v}^2=\widehat{v}^2(x',t)$.
    %On the other hand, $\widehat{\mathcal{O}}_{1,d}$, $\widehat{\mathcal{O}}_{2,d}\subset \Omega_t$ are open sets denoting the observation domains of the followers.

    The function $a(\cdot)$ represents the (degenerate) diffusion coefficient and we assume that it vanishes at $x=0$ and satisfies the following conditions:
    \begin{subequations}\label{condiciones_a4}
    	\begin{eqnarray}
    		&&a\in C([0,1])\cap C^1((0,1]), \quad a>0 \ \text{in} \ (0,1] \quad \text{and} \quad a(0)=0,\label{condiciones_a_a4}\\
    		&&\exists K \in (0,1] : a\geq 0, \ \ x a'(x)\leq Ka(x), \ \ \forall x\in[0,1],\label{condiciones_a_b4}\\
    		&& a(x\cdot y)=a(x)a(y), \quad \forall x,y\in [0,1]. \label{condiciones_a_c4}
    	\end{eqnarray}
    \end{subequations}
    
    In other words, the function $a$ behaves like $x^\alpha$, with $\alpha\in(0, 1)$. This condition is called \emph{weakly degenerate} in \cite{Alabau06}.
    On the other hand, for $i=1,2$, the coupling functions $F_i$ verify
    %\color{red}
    \begin{subequations}\label{condiciones_a4f}
    	\begin{eqnarray}
    		&&F_i(0,0)=0, \label{condiciones_a_a4f}\\
    		&&F_i\in C^2(\mathbb{R\times \R}),\label{condiciones_a_b4f}\\
    		&& \exists M>0 : \sum_{j=1}^2|D_j F_i(r,s)| + \sum_{j,k=1}^2 |D^2_{j,k}F''_i(r,s)|\leq M, \ \ \ \forall r,s\in \mathbb{R}.  \label{condiciones_a_c4f}
    	\end{eqnarray}
    \end{subequations}
    %\color{black}
    
    \noindent\textbf{Cost functionals in the noncylindrical domain $\widehat{Q}$}. [Associated to the solution $u_i$, $i=1,2$, of \eqref{sistemanocilin4}].  

    We denote by $\widehat{\mathcal{O}}_{1,d}$ and  $\widehat{\mathcal{O}}_{2,d}$ open sets of $\Omega_t$ defining the observation domains of the followers. For each follower, we consider the (secondary) functional:
    \begin{multline}\label{funcionalcilindro4}
        \widehat{J}_i(\widehat{h},\widehat{v}^1,\widehat{v}^2)=\frac{\alpha_i}{2}\iint_{\widehat{\mathcal{O}}_{i,d}\times(0,T)}\left(|u_1-u^i_{1,d}|^2+|u_2-u^i_{2,d}|^2\right)\ \left| Jac(\tau_t)\right|^{-1}dx' dt\\
	   +\frac{\mu_i}{2}\iint_{\widehat{\mathcal{O}}_{i}\times(0,T)}\rho_*^2|\widehat{v}^i|^2\ \left| Jac(\tau_t)\right|^{-1} dx' dt, \ \ \ \ \ \ \ i=1,2.
    \end{multline}
where $\alpha_i, \mu_i >0$ are constants and $(u^i_{1,d},u^i_{2,d})\in L^2((0,T);\widehat{\mathcal{O}}_{1,d})\times L^2((0,T);\widehat{\mathcal{O}}_{2,d})$ are given functions and $\rho_*=\rho_*(t)\in C^\infty ([0,T])$ is a weight function that blows up at $t=0$ and $t=T$.

%	Consideremos los funcionales asociados a los seguidores y líder %\limits 
%	\begin{eqnarray}\label{ec2a}
		%\widehat{J}_i(\widehat{h},\widehat{v}^1,\widehat{v}^2)&=&\frac{\alpha_i}{2}\iint_{\widehat{\mathcal{O}}_{i,d}\times(0,T)}\left(|u_1-u^1_{i,d}|^2+|u_2-u^i_{2,d}|^2\right)\ \left| J(\tau_t)\right|^{-1}dx' dt+\frac{\mu_i}{2}\iint_{\widehat{\mathcal{O}}_{i}\times(0,T)}\rho_*^2|\widehat{v}^i|^2\ \left| J(\tau_t)\right|^{-1} dx' dt\nonumber\\
	%\end{eqnarray}
	
	%onde $\alpha_i>0$, $\mu_i>0$ são contantes e $u_d^i=(u^i_{1d},u^i_{2d})\in L^2((0,T),\widehat{\mathcal{O}}_{1d	})\times L^2((0,T),\widehat{\mathcal{O}}_{2d})$ são funções dadas y $\rho_*=\rho_*(t)\in C^\infty ([0,T])$ son funciones peso que explotan en $t=0$ y $t=T$ 

%------------------------------------------------------

We can describe the Stackelberg-Nash control process in the following way:
\begin{enumerate}
	\item The followers $\widehat{v}^i$ assume that the leader $\widehat{h}$ has made a choice and intend to be in Nash equilibrium for the costs $\widehat{J}_i$. That is, once $\widehat{h}$ has been fixed, we look for controls $ \widehat{v}^i\in  L^2(\widehat{\mathcal{O}}_i\times  (0,T))$ that satisfy
	\begin{equation}\label{minimo4}
		\widehat{J}_1(\widehat{h};\widehat{v}^1,\widehat{v}^2):=\min_{\widetilde{\widehat{v}^1}}\widehat{J}_1(\widehat{h};\widetilde{\widehat{v}^1},\widehat{v}^2), \ \ \ \ \ \ 	\widehat{J}_2(\widehat{h};\widehat{v}^1,\widehat{v}^2):=\min_{\widetilde{\widehat{v}^2}}\widehat{J}_2(\widehat{h};\widehat{v}^2,\widetilde{\widehat{v}^2}).
	\end{equation}
	\begin{defi}\label{defiequilibrio4}
		Any pair $(\widehat{v}^1,\widehat{v}^2)$ satisfying (\ref{minimo4}) is called a Nash equilibrium for $\widehat{J}_1$ and $\widehat{J}_2$.
	\end{defi}
	Note that, if the functionals $\widehat{J}_i$, $i=1,2$,  are convex, then $(\widehat{v}^1,\widehat{v}^2)$ is a Nash equilibrium if and only if
	\begin{equation}\label{derivada14}
		\widehat{J}'_1(\widehat{h};\widehat{v}^1,\widehat{v}^2)(\widetilde{\widehat{v}^1},0)=0, \ \ \ \ \ \forall \widetilde{\widehat{v}^1}\in L^2(\widehat{\mathcal{O}}_1\times (0,T)),
	\end{equation}
	and 
	\begin{equation}\label{derivada24}
		\widehat{J}'_2(\widehat{h};\widehat{v}^1,\widehat{v}^2)(0,\widetilde{\widehat{v}^2})=0, \ \ \ \ \ \forall \widetilde{\widehat{v}^2}\in L^2(\widehat{\mathcal{O}}_2\times (0,T)).
	\end{equation}
	\begin{defi}\label{quasi4}
		Any pair $(\widehat{v}^1,\widehat{v}^2)$ satisfying (\ref{derivada14}) and (\ref{derivada24}) is called a Nash quasi-equilibrium for $\widehat{J}_1$ and $\widehat{J}_2$.
	\end{defi}
	\item Once the Nash equilibrium has been identified and fixed for each $\widehat{h}$, we look for a control $\widetilde{\widehat{h}}\in L^2(\widehat{\mathcal{O}}\times (0,T))$ giving of null controllability, i.e., such that
	\begin{equation}\label{nulo}
		u_1(T)=u_2(T)=0, \ \ \ \ \ \text{in} \ \ \ \Omega_T.
	\end{equation}
\end{enumerate}

Our first main result shows the existence of Nash quasi-equilibrium $(\hat v_1,\hat v_2)$ for $\hat J_1$ and $\hat J_2$, such that we obtain local null controllability of  \eqref{sistemanocilin4}. Since the equation is degenerate, the initial condition $u_0$ is taken in the weighted space $H_a^1(\hat \Omega)$,  defined in \cite{Alabau06} in the case of a cylindrical domain.

\begin{teo}
	\label{thm:local_null_controllability4}
	Let us assume the hypothesis considered in the system \eqref{sistemanocilin4} and that,
	$$
	\hat{\mathcal{O}}_{i,d} \cap \hat{\mathcal{O}} \neq \emptyset, \ \ \ \ \ \ \text{for}\quad i=1,2.
	$$
	and
	$$
	\hat{\mathcal{O}}_{d} := \hat{\mathcal{O}}_{1,d} = \hat{\mathcal{O}}_{2,d}.
	$$
	Then, for any $T>0$ there exist $\varepsilon>0$ and a positive function $\rho_2(t)$ blowing up at $t=T$ such that, if $\rho_2 u^i_{d} \in L^2(\hat{\mathcal{O}}_{d} \times (0,T))$ and $u^0_1, u^0_2 \in H_a^1(\Omega)$ verify
	$$
	\|(u^0_1, u^0_2)\|_{H_a^1(\hat \Omega_0)} \leq \varepsilon, 
	$$
	then there exists a control $\hat h \in L^2(\hat{\mathcal{O}}\times(0,T))$ and an associated Nash quasi-equilibrium $(\hat v^1,\hat v^2)$ such that the solution of \eqref{sistemanocilin4} is null-controllable at time $T$.
\end{teo}

    Our next result shows conditions under which the Nash quasi-equilibrium $(\hat v_1,\hat v_2)$ obtained in Theorem \ref{thm:local_null_controllability4} (see also Proposition \ref{propocaracteriza} in Section \ref{sec:characterization}) is a Nash equilibrium. More precisely, if  $\mu_i$ , $i = 1, 2$ are sufficiently large, then the functionals $J_i , i = 1, 2$ given by  \eqref{eqprin24} are convex.
	\begin{teo}\label{teoconvex_intro}
        Under the hypotheses of Theorem \ref{thm:local_null_controllability4}. If if $(\hat{v}^1,\hat{v}^2)$ in a Nash quasi-equilibrium for $\widehat J_i, i=1,2$, and $\mu_i, i=1,2$ are sufficiently large,  then there exists a constant $C>0$ independent of $\mu_i$, $i=1,2$ such
		\begin{equation*}%\label{ecsegundaderivada}
			D_i^2 \widehat J_i (\widehat h; \hat{v}^1,\hat{v}^2)\cdot (\widetilde{v}^i,\widetilde{v}^i)\geq C\|\widetilde{v}^i\|^2_{L^2(\mathcal{O}_i\times (0,T)} \ \ \ \ \ \ \ \ \
			\forall \widetilde{v}^i \in L^2( \mathcal{\widehat O}_i\times (0,T)), \ \ \ i=1,2.
		\end{equation*}
	\end{teo}

    For hierarchical control in coupled parabolic systems, Kéré, Marcan and Mophou \cite{francessis27} proposed a bi-objective control strategy with finite constraints in one of the states. These results were based on a Carleman observability inequality adapted to these constraints. In the works \cite{francessis22} and \cite{frances23}, V. Hernández-Santamaría, de Teresa and Poznyak  estudied a Stackelberg–Nash strategy for cascade systems of parabolic equations, as well as for cascade systems where the leader is modeled as a vector function acting on both equations. More recently, Límaco and Huaman \cite{frances33} applied the Stackelberg–Nash strategy to a coupled quasi-linear parabolic system, with controls acting inside the domain.

    In all the previous studies, hierarchical strategies were applied to non-degenerate systems. To the best of our knowledge, no works in the literature addresses the Stackelberg–Nash strategy in the context of a coupled degenerate parabolic system. As far as we know, the only studies examining a hierarchical strategy applied to a degenerate equation are those of Araruna, Araujo and Fernandez-Cara \cite{ararunano} and Djomegne, Kenne, Dorville and Zongo \cite{sistemacuate}, where the Stackelberg–Nash strategy for semi-linear degenerate parabolic equations in fixed domains is analyzed.

    The present work introduces a novelty by extending the results on the Stackelberg–Nash control to the coupled degenerate parabolic equation in moving domains, representing an advance in the application of hierarchical strategies to degenerate systems. It is based on a Carleman inequality for degenerate non-autonomous parabolic equation obtained by the authors in \cite{GYL-Carleman-2025} and also \cite{GYL-equation-2025} where the authors study the hierarchical control of one degenerate equation in moving domains.

%---------------------------------------------------------

    The paper is organized as follows. In Section \ref{sec:characterization} we characterize the follower controls in Nash quasi-equilibrium and prove Theorem \ref{teoconvex_intro}. In Section \ref{sec:carleman_estimates} we prove a Carleman estimate for the linearized optimality system using a Carleman inequality for non-autonomous equations in \cite{GYL-Carleman-2025}. In Section \ref{control for nonlinear system} we apply Liusternik's inverse function Theorem to show local null controllability of our semilinear system \eqref{sistemanocilin4}. 
    %Finally in Section ... we present some final remarks and related problems.
    Finally, in Appendices \ref{appendix_a} and \ref{appendix_b} we prove some technical results related to the Carleman inequality for our linearized optimality system.

\section{Characterization of Nash Equilibrium}
\label{sec:characterization}

First we use a change of variables given by a diffeomorphism to pass from the non-cylindrical domain $\widehat{Q}$ to a cylindrical domain $Q$, at the cost of new time-dependent coefficients of the transformed equation in $Q$.  

\subsection{Passing to the cylindrical domain}

The methodology in this paper consists in changing the non-cylindrical state equation (\ref{sistemanocilin4}) to a cylindrical one by the diffeomorphism $\tau : Q \rightarrow \widehat{Q}$. Its inverse is given by
$$\tau_t^{-1}: \widehat{Q}\rightarrow Q \ \ \text{such that } \ (x',t)\in \widehat{Q} \rightarrow (x,t)\in {Q}, \ \ \ \ \text{with} \ \ \ x=\tau_t^{-1}(x')=\psi(x',t)$$
The domains are transformed in the following way
$$\mathcal{O} \rightarrow \widehat{\mathcal{O}}, \ \quad  \mathcal{O}_1 \rightarrow \widehat{\mathcal{O}}_1, \ \quad 
\mathcal{O}_2 \rightarrow \widehat{\mathcal{O}}_2, \ \quad \mathcal{O}_{i,d} \rightarrow \widehat{\mathcal{O}}_{i,d}, \ \ i=1,2.$$

Then, we have the functions in the new variable $x$,
$$\widehat{h}(x',t)=h(\tau_t(x),t), \ \ \ \widehat{v}^i(x',t)=v^i(\tau_t(x),t), i=1,2, \quad
u_i(x',t)=y_i(x,t)=y_i\left(\tau_t^{-1}(x'),t\right),$$
$$\mathds{1}_{\widehat{\mathcal{O}}}(x',t)=\mathds{1}_{{\mathcal{O}}}(\tau_t(x),t), \ \ \ \ \mathds{1}_{\widehat{\mathcal{O}}_i}(x',t)=\mathds{1}_{{\mathcal{O}}_i}(\tau_t(x),t), \ \ \ i=1,2.$$	
Finally, applying the diffeomorphism to the system (\ref{sistemanocilin4}), then we obtain the following system:
{\small	
	\begin{equation}\label{ec1p4}
		\hspace*{-0.5cm}	\begin{cases}
			y_{1t}-\left( \dpar{\psi}{x'}(\tau_t(x),t) \right)^2{a}(\tau_t(x))y_{1xx}+\left(\dpar{\psi}{t}(\tau_t(x),t)-{a}'(\tau_t(x))\dpar{\psi}{x'}-{a}(\tau_t(x))\dpd{\psi}{x'}(\tau_t(x),t) \right)y_{1x}\\
			\hspace*{7.5cm}+F_1(y_1,y_2) ={h}\cara_{_{{\mathcal{O}}}}+{v}^1\cara_{_{{\mathcal{O}}_1}}+{v}^2\cara_{_{{\mathcal{O}}_2}}, & \text{in} \ \ {Q},\\
			y_{2t}-\left( \dpar{\psi}{x'}(\tau_t(x),t) \right)^2{a}(\tau_t(x))y_{2xx}+\left(\dpar{\psi}{t}(\tau_t(x),t)-{a}'(\tau_t(x))\dpar{\psi}{x'}-{a}(\tau_t(x))\dpd{\psi}{x'}(\tau_t(x),t) \right)y_{2x}\\
			\hspace*{10.8cm}+F_2(y_1,y_2) =0, & \text{in} \ \ {Q},\\
			y_1(0,t)=y_1(1,t)=y_2(0,t)=y_2(1,t)=0, & \text{in} \ \ {(0,T)},\\
			y_1(0)=u_1^0(\tau_0(x)), y_2(0)=u_2^0(\tau_0(x)), &  \text{in} \ \ \Omega.
		\end{cases}
	\end{equation}
	}
	%\begin{equation}\label{ec1p}
	%	\hspace*{-0.5cm}	\begin{cases}
		%		y_t-\left( \dpar{\psi}{x'}(\tau_t(x),t) \right)^2{a}(\tau_t(x))y_{xx}+\left(\dpar{\psi}{t}(\tau_t(x),t)-{a}'(\tau_t(x))\dpar{\psi}{x'}-{a}(\tau_t(x))\dpd{\psi}{x'}(\tau_t(x),t) \right)y_x\\
		%		\hspace*{6cm}+F\left(y,\dpar{\psi}{x'} \sqrt{a(\tau_t(x))}y_x\right)={h}\cara_{_{{\mathcal{O}}}}+{v}^1\cara_{_{{\mathcal{O}}_1}}+{v}^2\cara_{_{{\mathcal{O}}_2}}, & \text{in} \ \ {Q}\\
		%		y=0, & \text{in} \ \ {\Sigma}\\
		%		y(0)=v_0(\tau_0(x))=y_0, &  \text{in} \ \ \Omega
		%	\end{cases}
	%\end{equation}

Since the problem is treated in one dimension, we can explicitly define the diffeomorphism as follows. Let $\Omega_t=\{x'\in\mathbb{R} \ | \ 0<x' <\ell(t)\}$,   $x=\tau_t^{-1}(x')=\psi(x',t)=\frac{x'}{\ell(t)}$ and\\ $\Omega=\{x \in \mathbb{R} \ | \ 0< x<1 \}$, where the function $\ell(t)$ is positive, i.e. $\ell(t)>0$ for all $t$ in $(0,T)$. The function $\ell(t)$ must also satisfy that $\frac{\ell'(t)}{\ell(t)}\leq C$ with $C>0$.
	
Then, we have that $$\dpar{\psi}{x'}(\tau_t(x),t)=\frac{1}{\ell(t)}, \ \ \ \ \ \dpd{\psi}{x'}(\tau_t(x),t)=0, \ \ \ \ \ \ \dpar{\psi}{t}(\tau_t(x),t)=-x'\frac{\ell'(t)}{\ell(t)^2}=-x\ell(t)\frac{\ell'(t)}{\ell(t)^2}=-\frac{\ell'(t)}{\ell(t)}x$$
and
	\begin{equation}\label{eqprin54}
		\begin{cases}
			y_{1t}-\frac{1}{\ell(t)^2}\left({a}(\ell(t)x)y_{1x}\right)_x-\frac{\ell'(t)}{\ell(t)}xy_{1x}+F_1(y_1,y_2)={h}\cara_{_{{\mathcal{O}}}}+{v}^1\cara_{_{{\mathcal{O}}_1}}+{v}^2\cara_{_{{\mathcal{O}}_2}}, & \ \ \ \text{in} \ \ \ {Q},\\
			y_{2t}-\frac{1}{\ell(t)^2}\left({a}(\ell(t)x)y_{2x}\right)_x-\frac{\ell'(t)}{\ell(t)}xy_{2x}+F_2(y_1,y_2)=0, & \ \ \ \text{in} \ \ \ {Q},\\
			y_1(0,t)=y_2(1,t)=y_2(0,t)=y_2(1,t)=0, & \ \ \ \text{on} \ \ \ (0,T),\\
			y_1(0)=y_1^0(x), \ y_2(0)=y_2^0(x), & \ \ \ \text{in} \ \ \ \Omega.
		\end{cases}
	\end{equation}
	Let us define
    \begin{equation}\label{eq:def_b_B}
        b(t)=\frac{a(\ell(t))}{\ell(t)^2}, \qquad B(x,t)=\frac{\ell'(t)x}{\ell(t)\sqrt{a}}.
    \end{equation}
    Then, 
	\begin{equation}\label{eqprin544}
	\begin{cases}
		y_{1t}-b(t)\left({a}(x)y_{1x}\right)_x-B(x,t)\sqrt{a}y_{1x}+F_1(y_1,y_2)={h}\cara_{_{{\mathcal{O}}}}+{v}^1\cara_{_{{\mathcal{O}}_1}}+{v}^2\cara_{_{{\mathcal{O}}_2}}, & \ \ \ \text{in} \ \ \ {Q},\\
		y_{2t}-b(t)\left({a}(x)y_{2x}\right)_x-B(x,t)\sqrt{a}y_{2x}+F_2(y_1,y_2)=0, & \ \ \ \text{in} \ \ \ {Q},\\
		y_1(0,t)=y_2(1,t)=y_2(0,t)=y_2(1,t)=0, & \ \ \ \text{on} \ \ \ (0,T),\\
		y_1(0)=y_1^0(x), \ y_2(0)=y_2^0(x), & \ \ \ \text{in} \ \ \ \Omega.
	\end{cases}
\end{equation}	
The functionals for $J_i$, $i=1,2$, take the form
\begin{multline}\label{eqprin24}
	{J}_i({h},{v}^1,{v}^2)=\frac{\alpha_i}{2}\iint_{{\mathcal{O}}_{i,d}\times(0,T)}\left(|y_1-y^i_{1,d}|^2+|y_2-y^i_{2,d}|^2\right) dx dt
	+\frac{\mu_i}{2}\iint_{{\mathcal{O}}_{i}\times(0,T)}\rho_*^2|{v}^i|^2\  dx dt. %, \ \ \ i=1,2.
\end{multline}
    	
	We will suppose the following
	\begin{equation}\label{ec9}
		\mathcal{O}_{1,d}=\mathcal{O}_{2,d}, %=\mathcal{O}_{d},
	\end{equation}
    and we denote both sets by $\mathcal{O}_{d}$.
%	Assim, denotaremos esses conjuntos por $\mathcal{O}_{d}$, ver abaixo, na Seção 5, alguns comentários sobre a necessidade da hipótese.\\
	 %a controlabilidade exata das trajetórias é equivalente à propriedade de controlabilidade nula. Vale o seguinte resultado:

%\subsection{\bf Caracterização do Equilíbrio de Nash}

\subsection{Characterization of the Nash-quase equilibrium and the optimality system}
\label{subsec:characterization}

First, we give the following characterization of Nash quasi-equilibrium pair (recall Definition \ref{defiequilibrio4}) $(v^1,v^2)$ of \eqref{sistemanocilin4} for the functionals  $J_i$, $i = 1, 2$ given by \eqref{eqprin24}.
	
\begin{propo}\label{propocaracteriza}
	Let $h\in L^2(\mathcal{O}\times(0,T))$ and assume that $\mu_i, i = 1, 2$ are sufficiently large. Also let   $(v^1,v^2) \in L^2((0,T);L^2(\mathcal{O}_1))\times L^2((0,T);L^2(\mathcal{O}_2))$ be the Nash
	quasi-equilibrium pair for $(J_1, J_2)$. Then, there exists $p^i=(p_1^i, p_2^i)\in \mathbb{H}$  such that the Nash quasi-equilibrium pair   $(v^1,v^2)$ is characterized by
	\begin{equation}\label{caracteriza}
		v^i=-\frac{1}{\mu_i}\rho_*^{-2}p_1^i\cara_{_{\mathcal{O}_i}},
	\end{equation}
	where $y = (y_1, y_2)$ and $p^i = (p^i_1, p^i_2)$ are solutions of the following optimality system:
	\begin{equation}\label{optimal1}
		\begin{cases}
			y_{1t}-b(t)\left({a}(x)y_{1x}\right)_x-B\sqrt{a}y_{1x}+F_1(y_1,y_2)\\
			\hspace*{6cm}={h}\cara_{_{{\mathcal{O}}}}-\frac{1}{\mu_1}\rho_*^{-2}p_1^1\cara_{_{{\mathcal{O}}_1}}-\frac{1}{\mu_2}\rho_*^{-2}p_1^2\cara_{_{{\mathcal{O}}_2}}, & \ \ \ \text{in} \ \ \ {Q},\\
			y_{2t}-b(t)\left({a}(x)y_{2x}\right)_x-B\sqrt{a}y_{2x}+F_2(y_1,y_2)=0 & \ \ \ \text{in} \ \ \ {Q},\\
			y_1(0,t)=y_2(1,t)=y_2(0,t)=y_2(1,t)=0, & \ \ \ \text{on} \ \ \ (0,T),\\
			y_1(0)=y_1^0(x), \ y_2(0)=y_2^0(x) & \ \ \ \text{in} \ \ \ \Omega.
		\end{cases}
	\end{equation}	
	and, for $i=1,2$,
	\begin{equation}\label{optimal2}
        \begin{cases}
        	-p^i_{1t}-b(t)\left(a(x)p^i_{1x}\right)_x+\left(B\sqrt{a}p^i_1\right)_x + D_1 F_1(y_1,y_2)p^i_1 + D_1 F_2(y_1,y_2) p^i_2 \\
            \qquad= \alpha_i\left( y_1-y_{1,d}^i  \right)\cara_{_{{\mathcal{O}}_{i,d}}}, &\ \ \ \text{in} \ \ \ {Q},\\
        	-p^i_{2t}-b(t)\left(a(x)p^i_{2x}\right)_x+\left(B\sqrt{a}p^i_2\right)_x + D_2 F_1(y_1,y_2)p^i_1 + D_2 F_2(y_1,y_2)p^i_2 \\
            \qquad = \alpha_i\left( y_2-y_{2,d}^i  \right)\cara_{_{{\mathcal{O}}_{i,d}}}, &\ \ \ \text{in} \ \ \ {Q},\\
        	p^i_1(0,t)=p^i_1(1,t)=p^i_2(0,t)=p^i_2(1,t)=0, & \ \ \ \text{on} \ \ \ (0,T),\\
        	p^i_1(T)=0, \ p^i_2(T)=0, & \ \ \ \text{in} \ \ \ \Omega.
        \end{cases}%, \ \ \ i=1,2.
	\end{equation}	
\end{propo}	
\begin{proof}
    We first compute $J'_1(h,v^1,v^2)(\widetilde{v}^1,0)=\left.\frac{d}{d\lambda}J_1(h,v^1+\lambda\widetilde{v}^1,v^2)\right|_{\lambda=0}$. We have
	\begin{multline*}
        J_1(h,v^1+\lambda\widetilde{v}^1,v^2)=\frac{\alpha_1}{2}\iint_{\mathcal{O}_{1,d}\times(0,T)}\left( |y_1^\lambda-y^1_{1,d}|^2+|y_2^\lambda-y^1_{2,d}|^2 \right)\ dx dt\\
        +\frac{\mu_1}{2}\iint_{\mathcal{O}_{1}\times(0,T)}\rho_*^2|v^1+\lambda\widetilde{v}^1|^2\ dx dt,
	\end{multline*}
	where $y^\lambda$ is the solution of
		\begin{equation}\label{eqprin244}
			\hspace*{-0.2cm}	\begin{cases}
				y^\lambda_{1t}-b(t)\left(a(x)y^\lambda_{1x}\right)_x-B\sqrt{a}y^\lambda_{1x}+F_1(y^\lambda_1,y^\lambda_2)={h}\cara_{_{{\mathcal{O}}}}+(v^1+\lambda\widetilde{v}^1)\cara_{_{{\mathcal{O}}_1}}+{v}^2\cara_{_{{\mathcal{O}}_2}}, & \ \ \ \text{in} \ \ \ {Q},\\
				y^\lambda_{2t}-b(t)\left(a(x)y^\lambda_{2x}\right)_x-B\sqrt{a}y^\lambda_{2x}+F_2(y^\lambda_1,y^\lambda_2)=0,& \ \ \ \text{in} \ \ \ {Q},\\
				y^\lambda_1(0,t)=y^\lambda_2(1,t)=y^\lambda_2(0,t)=y^\lambda_2(1,t)=0, & \ \ \ \text{on} \ \ \ (0,T),\\
				y^\lambda_1(0)=y_0^1, \ y^\lambda_2(0)=y_2^0, & \ \ \ \text{in} \ \ \ \Omega.
			\end{cases}
		\end{equation}
		From the continuity with respect to the initial data, we have that $y_1^\lambda\rightarrow y_1$, $y_2^\lambda\rightarrow y_2$, as $\lambda\rightarrow 0$. We denote $w^1_1=\lim\limits_{\lambda\rightarrow 0}\frac{y_1^\lambda-y}{\lambda}$  and $w^1_2=\lim\limits_{\lambda\rightarrow 0}\frac{y_2^\lambda-y}{\lambda}$. 
		Subtracting (\ref{eqprin244}) from (\ref{optimal1}), dividing by $\lambda$, and making $\lambda\rightarrow 0$, we get
		\begin{equation}\label{eqprin344}
			\begin{cases}
				w^1_{1t}-b(t)\left(a(x)w^1_{1x}\right)_x-B\sqrt{a}w^1_{1x} + D_1 F_1(y_1,y_2)w^1_1 + D_2 F_1(y_1,y_2)w^1_2 =\widetilde{v}^1\cara_{_{{\mathcal{O}}_1}}, &\ \ \ \text{in} \ \ \ {Q}, \\
				w^1_{2t}-b(t)\left(a(x)w^1_{2x}\right)_x-B\sqrt{a}w^1_{2x}+D_1 F_2 (y_1,y_2)w^1_1 + D_2 F_2(y_1,y_2)w^1_2=0, &\ \ \ \text{in} \ \ \ {Q},\\
				w^1_1(0,t)=w^1_1(1,t)=w^1_2(0,t)=w^1_2(1,t)=0, & \ \ \ \text{on} \ \ \ (0,T),\\
				w^1_1(0)=0, \ w^1_2(0)=0, & \ \ \ \text{in} \ \ \ \Omega.
			\end{cases}
		\end{equation}
		Then, the adjoint system is given by
		\begin{equation}\label{eqprin444}
			\begin{cases}
				-p^1_{1t}-b(t)\left(a(x)p^1_{1x}\right)_x+\left(B\sqrt{a}p^1_1\right)_x + D_1 F_1(y_1,y_2)p^1_1 + D_1 F_2(y_1,y_2)p^1_2\\
                \qquad =\alpha_1\left( y_1-y_{1,d}^1  \right)\cara_{_{{\mathcal{O}}_{1,d}}}, &\ \ \ \text{in} \ \ \ {Q}, \\
				-p^1_{2t}-b(t)\left(a(x)p^1_{2x}\right)_x+\left(B\sqrt{a}p^1_2\right)_x + D_2 F_1(y_1,y_2)p^1_1 + D_2 F_2(y_1,y_2)p^1_2\\
                \qquad =\alpha_1\left( y_2-y_{2,d}^1  \right)\cara_{_{{\mathcal{O}}_{2,d}}}, &\ \ \ \text{in} \ \ \ {Q},\\
				p^1_1(0,t)=p^1_1(1,t)=p^1_2(0,t)=p^1_2(1,t)=0, & \ \ \ \text{on} \ \ \ {(0,T)},\\
				p^1_1(T)=0, \ p^1_2(T)=0, & \ \ \ \text{in} \ \ \ \Omega.
			\end{cases}
		\end{equation}
				
		If we multiply (\ref{eqprin444})$_1$ by $w_1^1$ and (\ref{eqprin444})$_2$ by $w_2^1$, where $(w_1^1,w_2^1)$ is solution of (\ref{eqprin344}), we have
		\begin{eqnarray*}
			&&\int_{Q}\left(-p^1_{1t}-b(t)\left(a(x)p^1_{1x}\right)_x+\left(B\sqrt{a}p^1_1\right)_x + D_1 F_1(y_1,y_2)p^1_1 + D_1 F_2(y_1,y_2)p^1_2\right)w^1_1\\
            &&\qquad=\alpha_1\int_{Q}\left( y_1-y_{1,d}^1  \right)\cara_{_{{\mathcal{O}}_{1,d}}}w_1^1,\\
			&&\int_{Q}\left(-p^1_{2t}-b(t)\left(a(x)p^1_{2x}\right)_x+\left(B\sqrt{a}p^1_2\right)_x+ D_2 F_1(y_1,y_2)p^1_1 + D_2 F_2(y_1,y_2)p^1_2\right)w^1_2\\
            &&\qquad =\alpha_1\int_{Q}\left( y_2-y_{2,d}^1  \right)\cara_{_{{\mathcal{O}}_{2,d}}}w^1_2,
		\end{eqnarray*}
		and, integrating by parts on $Q$, we obtain
		\begin{eqnarray*}
			&&\int_{Q}\left(w^1_{1t}-b(t)\left(a(x)w^1_{1x}\right)_x-B\sqrt{a}w^1_{1x} + D_1 F_1(y_1,y_2)w^1_1 + D_2 F_1(y_1,y_2)w^1_2\right)p^1_1\\
            &&\qquad = \alpha_1\int_{Q}\left( y_1-y_{1,d}^1  \right)\cara_{_{{\mathcal{O}}_{1,d}}}w_1^1,\\
			&&\int_{Q}\left(w^1_{2t}-b(t)\left(a(x)w^1_{2x}\right)_x-B\sqrt{a}w^1_{2x} + D_1 F_2 (y_1,y_2)w^1_1 + D_2 F_2(y_1,y_2)w^1_2 \right)p^1_2\\
            &&\qquad = \alpha_1 \int_{Q}\left( y_2-y_{2,d}^1  \right)\cara_{_{{\mathcal{O}}_{2,d}}}w^1_2.
		\end{eqnarray*}
		
		Substituting the value of the right-hand side of the two equations (\ref{eqprin344}) and summing both equations we get
        %from (\ref{ecderivadaJ4}), we get
		$$\int_{Q}\widetilde{v}^1\cara_{_{{\mathcal{O}}_1}}p^1_1=\alpha_1\int_{Q}\left( y_1-y_{1,d}^1  \right)\cara_{_{{\mathcal{O}}_{1,d}}}w_1^1+\alpha_1\int_{Q}\left( y_2-y_{2,d}^1  \right)\cara_{_{{\mathcal{O}}_{2,d}}}w^1_2.$$
        
		Using the fact that ${\mathcal{O}}_{1,d}={\mathcal{O}}_{2,d}$, one has
        $$\int_{Q}\widetilde{v}^1\cara_{_{{\mathcal{O}}_1}}p^1_1 = \alpha_1\int_{{\mathcal{O}}_{1,d}}\left[\left( y_1-y_{1,d}^1  \right)w_1^1+\left( y_2-y_{2,d}^1  \right)w^1_2\right].$$

        From the derivative of the functional $J_1$, one has
		\begin{equation*}%\label{ecderivadaJ4}
			\alpha_1\iint_{\mathcal{O}_{1,d}\times(0,T)}\left( (y_1-y^1_{1,d})w^1_1+(y_2-y^1_{2,d})w_2^1 \right)\ dx dt+\mu_1\iint_{\mathcal{O}_{1}\times(0,T)}\rho_*^2v^1\widetilde{v}^1\ dx dt=0.
		\end{equation*}
        
		%Using the derivative of the functional given in (\ref{ecderivadaJ4}), we have
        Thus,
		$$\iint_{\mathcal{O}_{1}\times(0,T)}\widetilde{v}^1p^1_1=-\mu_1\iint_{\mathcal{O}_{1}\times(0,T)}\rho_*^2v^1\widetilde{v}^1\ dx dt, \ \ \ \ \ \ \ \forall \widetilde{v}^1\in L^2(\mathcal{O}_i \times (0,T)), $$
		from which
		$$  	v^1=-\frac{1}{\mu_1}\rho_*^{-2}p^1_1\cara_{_{\mathcal{O}_1}}.$$
		
		Analogously, computing $J'_2(h,v^1,v^2)$, we get $v^2=-\frac{1}{\mu_2}\rho_*^{-2}p^2_1\cara_{_{\mathcal{O}_2}}$. Thus, substituting the followers controls, we get the optimality system \eqref{optimal1} and \eqref{optimal2}.
				
	\end{proof}

Now, we prove that if  $\mu_i$ , $i = 1, 2$ are sufficiently large, then the functionals $J_i , i = 1, 2$ given by  \eqref{eqprin24} are indeed convex, which demonstrates the equivalence between Nash quasi-equilibrium and Nash equilibrium. 
More precisely, we have the following result:
\begin{teo}\label{teoconvex}
	Assume that \eqref{condiciones_a4}, \eqref{condiciones_a4f} hold and that $y_1^0, y_2^0\in H^1_a(\Omega)$ and $y^i_{1,d} \in L^\infty((0,T)\times \mathcal{O}_{1,d})$, $y^i_{2,d} \in L^\infty((0,T)\times \mathcal{O}_{2,d})$. Let us suppose that $h\in L^2((0,T)\times \mathcal{O})$ and $\mu_i, i=1,2$ are sufficiently large. Then, if $(v^1,v^2)$ is a Nash quasi-equilibrium for $J_i, i=1,2$, there exists a constant $C>0$ independent of $\mu_i$, $i=1,2$, such that
		\begin{equation}\label{ecsegundaderivada}
			D_i^2 J_i (h; v^1,v^2)\cdot (\widetilde{v}^i,\widetilde{v}^i)\geq C\|\widetilde{v}^i\|^2_{L^2(\mathcal{O}_i\times (0,T))} \ \ \ \ \ \ \ \ \
			\forall \widetilde{v}^i \in L^2(\mathcal{O}_i\times (0,T)), \ \ \ i=1,2.
		\end{equation}
	\end{teo}
	
\begin{proof}
	The proof is based on the approach used in \cite[Proposition 1.4,]{FC2015} and \cite{sistemafrances}.
    Since $(v^1,v^2)$ is a Nash quasi-equilibrium of $J_i$, $i=1,2$, then
%\begin{eqnarray*}
%	J_1(f,v^1,v^2)&=&\frac{\alpha_1}{2}\iint_{\mathcal{O}_{1,d}\times(0,T)}|z-z_{1,d}|^2\ dx dt+\frac{\mu_1}{2}\iint_{\mathcal{O}_{1}\times(0,T)}|v^1|^2\ dx dt\nonumber\\
%	& & \\
%	J_2(f,v^1,v^2)&=&\frac{\alpha_2}{2}\iint_{\mathcal{O}_{2,d}\times(0,T)}|z-z_{2,d}|^2\ dx dt+\frac{\mu_2}{2}\iint_{\mathcal{O}_{2}\times(0,T)}|v^2|^2\ dx dt\nonumber
%\end{eqnarray*}
%o que significa
%	\begin{equation*}
	%	\begin{array}{ccc}
		%		J_1(f;v^1,v^2)& = &\min\limits_{\hat{v}^1} J_1(f;\hat{v}^1,v^2)  \\
		%		J_2(f;v^1,v^2)& = & \min\limits_{\hat{v}^2} J_2(f;v^1,\hat{v}^2)
		%	\end{array}	
	%\end{equation*}
	% se o somente se \'e um quase equilíbrio se
	%  
	$$J_1'(f,v^1,v^2)(\hat{v}_1,0)=0, \ \ \ \ \ \forall \hat{v}^1\in L^2(\mathcal{O}_1 \times (0,T)), \qquad J_2'(f,v^1,v^2)(0,\hat{v}^2)=0, \ \ \ \ \ \forall \hat{v}^2\in L^2(\mathcal{O}_2 \times (0,T)).$$
    
	  That is, for $i=1,2$,	
	%J_i'(f,v^1,v^2)(\hat{v}_1,0) =
    $$\alpha_i\int_{\mathcal{O}_{i,d}\times (0,T)}\left( (y_1-y_{1,d})w^i_1 + (y_2-y_{2,d})w^i_2 \right) dxdt+\mu_i\int_{\mathcal{O}_{i}\times (0,T)}v^i\widetilde{v}^idxdt=0,
	\ \ \ \ \ \forall \widetilde{v}^i\in L^2(\mathcal{O}_i \times (0,T)),$$
	where $w^i$ satisfies the systems \eqref{eqprin344}.
    
%	\begin{equation}\label{eqprin443}
%		\begin{cases}
%			w^i_{1t}-b(t)\left(a(x)w^i_{1x}\right)_x-B \sqrt{a}w^i_{1x} + D_1 F_1\left(y_1,y_2 \right)w^i_1 + D_2 F_1\left(y_1,y_2 \right)w^i_2=\widetilde{v}^i\cara_{_{{\mathcal{O}}_i}}, & \ \ \ \text{in} \ \ \ {Q}\\
%			w^i=0, & \ \ \ \text{in} \ \ \ {\Sigma}\\
%			w^i(0)=0, & \ \ \ \text{in} \ \ \ \Omega
%		\end{cases}
%	\end{equation}

	Fix $i=1$. Given $\lambda\in \rea$ and $\widetilde{v}^1, \overline{v}^1$ in $L^2(\mathcal{O}_1 \times (0,T))$, define 
	$$P(\lambda) \rightarrow \langle D_1 J_1(h,v^1+\lambda\widetilde{v}^1,v^2 ),\overline{v}^1 \rangle$$
	such that  
	\begin{multline}\label{D2lambda}
        \langle D_1 J_1(h,v^1+\lambda\widetilde{v}^1,v^2 ),\overline{v}^1 \rangle=\alpha_1\int_{0}^{T}\int_{\mathcal{O}_{1,d}}\left( (y_1^\lambda-y^1_{1,d})q^{1,\lambda}_1+(y_2^\lambda-y^1_{2,d})q^{1,\lambda}_2 \right)\ dx dt\\
        +\mu_1\int_{0}^{T}\int_{\mathcal{O}_{1}}\rho_*^2(v^1+\lambda\widetilde{v}^1)\overline{v}^1\ dx dt,
	\end{multline}
%	$$\langle D_1 J_1(h,v^1+\lambda\widetilde{v}^1,v^2 ),\overline{v}^1 \rangle=\alpha_1\int_{0}^{T}\int_{\mathcal{O}_{1,d}}\left( (y_1-y^1_{1,d})w^1_1+(y_2-y^1_{2,d})w_2^1 \right)\ dx dt+\mu_1\int_{0}^{T}\int_{\mathcal{O}_{1}}\rho_*^2(v^1+\lambda\widetilde{v}^1)\widetilde{v}^1\ dx dt$$
%	$$\langle D_1 J_1(h,v^1+\lambda\widetilde{v}^1,v^2 ),\overline{v}^1 \rangle=\alpha_1\int_{\mathcal{O}_{1,d}\times (0,T)}(y^\lambda-y_{1d})q^\lambda dxdt+\mu_1\int_{\mathcal{O}_{1}\times (0,T)}(v^1+\lambda\widetilde{v}^1)\overline{v}^1dxdt$$
		where $(y_1^\lambda, y_2^\lambda)$ is the solution of
\begin{equation}\label{ylambda}
	\hspace*{-0.2cm}	\begin{cases}
		y^\lambda_{1t}-b(t)\left(a(x)y^\lambda_{1x}\right)_x-B\sqrt{a}y^\lambda_{1x}+F_1(y^\lambda_1,y^\lambda_2)={h}\cara_{_{{\mathcal{O}}}}+(v^1+\lambda\widetilde{v}^1)\cara_{_{{\mathcal{O}}_1}}+{v}^2\cara_{_{{\mathcal{O}}_2}}, & \ \ \ \text{in} \ \ \ {Q},\\
		y^\lambda_{2t}-b(t)\left(a(x)y^\lambda_{2x}\right)_x-B\sqrt{a}y^\lambda_{2x}+F_2(y^\lambda_1,y^\lambda_2)=0,& \ \ \ \text{in} \ \ \ {Q},\\
		y^\lambda_1(0,t)=y^\lambda_2(1,t)=y^\lambda_2(0,t)=y^\lambda_2(1,t)=0, & \ \ \ \text{on} \ \ \ (0,T),\\
		y^\lambda_1(0)=y_0^1, \ y^\lambda_2(0)=y_2^0, & \ \ \ \text{in} \ \ \ \Omega.
	\end{cases}
\end{equation}
	and $q^{1,\lambda}_1, q^{1,\lambda}_2$ are the derivative of the state $y^\lambda_1$ and $y^\lambda_2$ with respect to $v^1$ in the direction $\widetilde{v}^1$, i.e, the solution of   
		\begin{equation}\label{qlambda}
	\begin{cases}
		q^{1,\lambda}_{1t}-b(t)\left(a(x)q^{1,\lambda}_{1x}\right)_x-B\sqrt{a}q^{1,\lambda}_{1x}+D_1 F_1(y_1^\lambda,y_2^\lambda)q^{1,\lambda}_1 + D_2 F_1(y_1^\lambda,y_2^\lambda)q^{1,\lambda}_2=\widetilde{v}^1\cara_{_{{\mathcal{O}}_1}}, & \ \ \text{in} \ \ {Q}, \\
		q^{1,\lambda}_{2t}-b(t)\left(a(x)q^{1,\lambda}_{2x}\right)_x-B\sqrt{a}q^{1,\lambda}_{2x}+D_1 F_2(y_1^\lambda,y_2^\lambda) q^{1,\lambda}_1 + D_2 F_2(y_1^\lambda,y_2^\lambda)q^{1,\lambda}_2=0, &\ \ \text{in} \ \ {Q},\\
		q^{1,\lambda}_1(0,t)=q^{1,\lambda}_1(1,t)=q^{1,\lambda}_2(0,t)=q^{1,\lambda}_2(1,t)=0, & \ \ \text{on} \ \ (0,T),\\
		q^{1,\lambda}_1(0)=0, \ q^{1,\lambda}_2(0)=0, & \ \ \text{in} \ \ \Omega.
	\end{cases}
\end{equation}
The adjoint system of \eqref{qlambda} is given by
\begin{equation}\label{qlambdaadjunto}
	\begin{cases}
		-p^{1,\lambda}_{1t}-b(t)\left(a(x)p^{1,\lambda}_{1x}\right)_x+\left(B\sqrt{a}p^{1,\lambda}_1\right)_x + D_1 F_1(y_1^\lambda,y_2^\lambda)p^{1,\lambda}_1+D_1 F_2(y_1^\lambda,y_2^\lambda)p^{1,\lambda}_2\\
        \qquad =\alpha_1\left( y^\lambda_1-y_{1,d}^1  \right)\cara_{_{{\mathcal{O}}_{1,d}}}, &\ \text{in} \ {Q}, \\
		-p^{1,\lambda}_{2t}-b(t)\left(a(x)p^{1,\lambda}_{2x}\right)_x+\left(B\sqrt{a}p^{1,\lambda}_2\right)_x + D_2 F_1(y_1^\lambda,y_2^\lambda)p^{1,\lambda}_1 + D_2 F_2(y_1^\lambda,y_2^\lambda)p^{1,\lambda}_2\\
        \qquad =\alpha_1\left( y^\lambda_2-y_{2,d}^1  \right)\cara_{_{{\mathcal{O}}_{2,d}}}, &\ \text{in} \ {Q},\\
		p^{1,\lambda}_1(0,t)=p^{1,\lambda}_1(1,t)=p^{1,\lambda}_2(0,t)=p^{1,\lambda}_2(1,t)=0, & \ \text{on} \  {(0,T)},\\
		p^{1,\lambda}_1(T)=0, \ p^{1,\lambda}_2(T)=0, & \ \text{in} \ \Omega,
	\end{cases}
\end{equation}
	where we denote $y_1=\left. y_1^\lambda\right|_{\lambda=0}$, $y_2=\left. y_2^\lambda\right|_{\lambda=0}$ and $q^1_1=\left. q^{1,\lambda}_1\right|_{\lambda=0}$, $q^1_2=\left. q^{1,\lambda}_2\right|_{\lambda=0}$. Using the systems above, we consider
	\begin{multline}\label{ecD2defi}
		P(\lambda)-P(0)=\lambda \mu_1 \int_{0}^{T}\int_{\mathcal{O}_{1}}\rho_*^2\overline{v}^1\widetilde{v}^1dxdt
		+\alpha_1\int_{0}^{T}\int_{\mathcal{O}_{1,d}}\left( (y_1^\lambda-y^1_{1,d})q^{1,\lambda}_1+(y_2^\lambda-y^1_{2,d})q^{1,\lambda}_2 \right)\ dx dt \\
		-\alpha_1\int_{0}^{T}\int_{\mathcal{O}_{1,d}}\left( (y_1-y^1_{1,d})q^{1}_1+(y_2-y^1_{2,d})q^{1}_2 \right)\ dx dt.
	\end{multline}
	
    Multiplying the first and the second equation of \eqref{qlambda} by $p^{1,\lambda}_{1}$ and $p^{1,\lambda}_{2}$ , respectively, we obtain
	\begin{equation*}
	    \begin{split}
	        &\int_{Q}\left(q^{1,\lambda}_{1t}-b(t)\left(a(x)q^{1,\lambda}_{1x}\right)_x-B\sqrt{a}q^{1,\lambda}_{1x}+D_1 F_1(y_1^\lambda,y_2^\lambda)q^{1,\lambda}_1 + D_2 F_1(y_1^\lambda,y_2^\lambda)q^{1,\lambda}_2 \right)p^{1,\lambda}_{1} dxdt \\
            &\qquad=\int_{Q} \widetilde{v}^1\cara_{_{{\mathcal{O}}_1}}p^{1,\lambda}_{1} dxdt,
	    \end{split}
	\end{equation*}
	$$
	\int_{Q}\left(q^{1,\lambda}_{2t}-b(t)\left(a(x)q^{1,\lambda}_{2x}\right)_x-B\sqrt{a}q^{1,\lambda}_{2x}+D_1 F_2(y_1^\lambda,y_2^\lambda) q^{1,\lambda}_1 + D_2 F_2(y_1^\lambda,y_2^\lambda)q^{1,\lambda}_2\right)p^{1,\lambda}_{2} dxdt=0.
	$$
    
	Integrating by parts over $Q$ and summing the two equations, one has
	\begin{multline*}
        \int_{Q}\left(-p^{1,\lambda}_{1t}-b(t)\left(a(x)p^{1,\lambda}_{1x}\right)_x+\left(B\sqrt{a}p^{1,\lambda}_1\right)_x+ D_1 F_1(y_1^\lambda,y_2^\lambda)p^{1,\lambda}_1+D_1 F_2(y_1^\lambda,y_2^\lambda)p^{1,\lambda}_2\right)q^{1,\lambda}_{1} dxdt\\
        +\int_{Q}\left(-p^{1,\lambda}_{2t}-b(t)\left(a(x)p^{1,\lambda}_{2x}\right)_x+\left(B\sqrt{a}p^{1,\lambda}_2\right)_x+ D_2 F_1(y_1^\lambda,y_2^\lambda)p^{1,\lambda}_1 + D_2 F_2(y_1^\lambda,y_2^\lambda)p^{1,\lambda}_2\right)q^{1,\lambda}_{2} dxdt\\
        =\int_{Q} \widetilde{v}^1\cara_{_{{\mathcal{O}}_1}}p^{1,\lambda}_{1} dxdt.
	\end{multline*}
	
	Using \eqref{qlambdaadjunto}, we get
	\begin{equation}\label{inter1}
		\alpha_1\int_{0}^{T}\int_{\mathcal{O}_{1,d}}\left( (y_1^\lambda-y^1_{1,d})q^{1,\lambda}_1+(y_2^\lambda-y^1_{2,d})q^{1,\lambda}_2 \right)\ dx dt=\int_{Q} \widetilde{v}^1\cara_{_{{\mathcal{O}}_1}}p^{1,\lambda}_{1} \ dxdt.
	\end{equation}
	
    Similarly, when $\lambda=0$ we have
		\begin{equation}\label{inter2}
			\alpha_1\int_{0}^{T}\int_{\mathcal{O}_{1,d}}\left( (y_1-y^1_{1,d})q^{1}_1+(y_2-y^1_{2,d})q^{1}_2 \right)\ dx dt=\int_{Q} \widetilde{v}^1\cara_{_{{\mathcal{O}}_1}}p^{1}_{1} \ dxdt,
	\end{equation}
    and using \eqref{ecD2defi}, \eqref{inter1} and \eqref{inter2} we obtain
 	%\begin{equation}\label{ecD2defisimpli}
 	  %    P(\lambda)-P(0)=\lambda \mu_1 \int_{0}^{T}\int_{\mathcal{O}_{1}}\rho_*^2\overline{v}^1\widetilde{v}^1dxdt
 	  %    +\int_{0}^{T}\int_{\mathcal{O}_{1}}\left( p^{1,\lambda}_1-p^{1}_1 \right)\widetilde{v}^1 \ dx dt.
    %\end{equation}
 %
	%Then,
	\begin{equation}\label{dividelambda}
		\frac{P(\lambda)-P(0)}{\lambda}=\int_{0}^{T}\int_{\mathcal{O}_{1}}\left(\frac{ p^{1,\lambda}_1-p^{1}_1}{\lambda}\right)\widetilde{v}^1 dxdt+\mu_1 \int_{0}^{T}\int_{\mathcal{O}_{1}} \rho_*^2 \overline{v}^1\widetilde{v}^1 \ dxdt.
	\end{equation}
	
	Notice that, using the systems \eqref{ylambda} and \eqref{qlambdaadjunto}, one has
	$$	\begin{cases}
		(y_1^\lambda-y_1)_{t}-b(t)\left(a(x)(y_1^\lambda-y_1)_{x}\right)_x-B\sqrt{a}(y_1^\lambda-y_1)_{x}+F_1(y^\lambda_1,y^\lambda_2)-F_1(y_1,y_2)=\lambda\widetilde{v}^1\cara_{_{{\mathcal{O}}_1}}, &\\
		(y_2^\lambda-y_2)_{t}-b(t)\left(a(x)(y_2^\lambda-y_2)_{x}\right)_x-B\sqrt{a}(y_2^\lambda-y_2)_{x}
        %+c_{21}(y_1^\lambda-y_1)
        +F_2(y^\lambda_1,y^\lambda_2)-F_2(y_1,y_2)=0,&
	\end{cases}$$
	and
		$$		\begin{cases}
			-(p^{1,\lambda}_1-p^{1}_1)_{t}-b(t)\left(a(x)(p^{1,\lambda}_1-p^{1}_1)_{x}\right)_x+\left(B\sqrt{a}(p^{1,\lambda}_1-p^{1}_1)\right)_x&\\
			\hspace*{2cm}+[D_1 F_1(y_1^\lambda,y_2^\lambda)-D_1 F_1(y_1,y_2)]p^{1,\lambda}_1 + D_1 F_1(y_1,y_2)(p^{1,\lambda}_1-p^{1}_1)\\
            \hspace*{2cm}+[D_1 F_2(y_1^\lambda,y_2^\lambda)-D_1 F_2(y_1,y_2)]p^{1,\lambda}_2 + D_1 F_2(y_1,y_2)(p^{1,\lambda}_2-p^{1}_2)
            %+c_{21}(p^{1,\lambda}_2-p^{1}_2)
            =\alpha_1\left( y^\lambda_1-y_1  \right)\cara_{_{{\mathcal{O}}_{1d}}}, & \\
			-(p^{1,\lambda}_2-p^{1}_2)_{t}-b(t)\left(a(x)(p^{1,\lambda}_2-p^{1}_2)_{x}\right)_x+\left(B\sqrt{a}(p^{1,\lambda}_2-p^{1}_2)\right)_x\\
            \hspace*{2cm}+[D_2 F_1(y_1^\lambda,y_2^\lambda)-D_2 F_1(y_1,y_2)]p^{1,\lambda}_1 + D_2 F_1(y_1,y_2)(p^{1,\lambda}_1-p^{1}_1)\\
            \hspace*{2cm}+[D_2 F_2(y_1^\lambda,y_2^\lambda)-D_2 F_2(y_1,y_2)]p^{1,\lambda}_2 + D_2 F_2(y_1,y_2)(p^{1,\lambda}_2-p^{1}_2)
			%\hspace*{2cm}+[F_2'(y_2^\lambda)-F_2'(y_2)]p^{1,\lambda}_2+F_2'(y_2)(p^{1,\lambda}_2-p^{1}_2)
            =\alpha_1\left( y^\lambda_2-y_2  \right)\cara_{_{{\mathcal{O}}_{2d}}}. &
		\end{cases}$$
		
		%Consequently, using that $F_i\in C^2(\mathbb{R})$, we obtain that the following limits:

	Let us denote
	\begin{equation*}
		\eta^1_i=\lim_{\lambda\to 0} \frac{1}{\lambda}(p_i^{1,\lambda} - p_i^1) \ \ \ \  \ \ \text{and} \ \ \ \ \ \  \theta_i=\lim_{\lambda\to 0} \frac{1}{\lambda}(y^\lambda_i - y_i), \ \ \ \ \ \text{for} \ \ i=1,2.
	\end{equation*}

%------------------------------------------------
    
	Using that $F_i\in C^2(\mathbb{R})$, dividing the previous two systems  by $\lambda$ and making $\lambda \to 0$ we deduce the following systems:
	\begin{equation}\label{sistemaconvexo1}
	\begin{cases}
		\theta_{1t}-b(t)\left(a(x)\theta_{1x}\right)_x-B\sqrt{a}\theta_{1x}+D_1 F_1(y_1,y_2)\theta_1 + D_2 F_1(y_1,y_2)\theta_2 = \widetilde{v}^1\cara_{_{{\mathcal{O}}_1}}, & \ \ \ \text{in} \ \ \ {Q},\\
		\theta_{2t}-b(t)\left(a(x)\theta_{2x}\right)_x-B\sqrt{a}\theta_{2x}+D_1 F_2(y_1,y_2)\theta_1 + %+c_{21}\theta_1+
        D_2 F_2(y_1,y_2)\theta_2=0,& \ \ \ \text{in} \ \ \ {Q},\\
		\theta_1(0,t)=\theta_1(1,t)=\theta_2(0,t)=\theta_2(1,t)=0, & \ \ \ \text{on} \ \ \ (0,T),\\
		\theta_1(0)=0, \ \theta_2(0)=0, & \ \ \ \text{in} \ \ \ \Omega,
	\end{cases}
    \end{equation}
	and
    \begin{equation}\label{sistemaconvexo2_0}
		\begin{cases}
			-\eta^1_{1t}-b(t)\left(a(x)\eta^1_{1x}\right)_x+\left(B\sqrt{a}\eta^1_1\right)_x
            +[D_{11}^2 F_1(y_1,y_2)\theta_1+D_{12}^2 F_1(y_1,y_2)\theta_2]p_1^1 \\
            \qquad + D_1 F_1(y_1,y_2)\eta_1^1 +[D_{11}^2 F_2(y_1,y_2)\theta_1+D_{12}^2 F_2(y_1,y_2)\theta_2]p_2^1 + D_1 F_2(y_1,y_2)\eta_2^1 =\alpha_1\theta_1\cara_{_{{\mathcal{O}}_{1d}}}, &\ \text{in} \ {Q},\\
            %+F_1'(y_1)\eta^1_1+F_1''(y_1)\theta_1p^{1}_1+c_{21}\eta^1_2=\alpha_1\theta_1\cara_{_{{\mathcal{O}}_{1d}}}, &\ \text{in} \ {Q},\\
			-\eta^1_{2t}-b(t)\left(a(x)\eta^1_{2x}\right)_x+\left(B\sqrt{a}\eta^1_2\right)_x
            +[D_{21}^2 F_1(y_1,y_2)\theta_1+D_{22}^2 F_1(y_1,y_2)\theta_2]p_1^1 \\
            \qquad + D_2 F_1(y_1,y_2)\eta_1^1 + [D_{21}^2 F_2(y_1,y_2)\theta_1+D_{22}^2 F_2(y_1,y_2)\theta_2]p_2^1 + D_2 F_2(y_1,y_2)\eta_2^1 =\alpha_1\theta_2\cara_{_{{\mathcal{O}}_{1d}}}, &\ \text{in} \ {Q},\\
            %+F_2'(y_2)\eta^1_2+F_2''(y_2)\theta_2p^{1}_2=\alpha_1\theta_2\cara_{_{{\mathcal{O}}_{1d}}}, &\ \text{in} \ {Q},\\
			\eta^{1}_1(0,t)=\eta^{1}_1(1,t)=\eta^{1}_2(0,t)=\eta^{1}_2(1,t)=0, & \ \text{on} \  {(0,T)},\\
			\eta^{1}_1(T)=0, \ \eta^{1}_2(T)=0, & \ \text{in} \ \Omega.
		\end{cases}
	\end{equation}
    Or, in a simpler form
    \begin{equation}\label{sistemaconvexo2}
		\begin{cases}
			-\eta^1_{1t}-b(t)\left(a(x)\eta^1_{1x}\right)_x+\left(B\sqrt{a}\eta^1_1\right)_x
            + D_1 F_1(y_1,y_2)\eta_1^1 + D_1 F_2(y_1,y_2)\eta_2^1 +\mathcal{N}_1^1(\eta,\theta) \\
            \qquad =\alpha_1\theta_1\cara_{_{{\mathcal{O}}_{1d}}}, &\ \text{in} \ {Q},\\
            %+F_1'(y_1)\eta^1_1+F_1''(y_1)\theta_1p^{1}_1+c_{21}\eta^1_2=\alpha_1\theta_1\cara_{_{{\mathcal{O}}_{1d}}}, &\ \text{in} \ {Q},\\
			-\eta^1_{2t}-b(t)\left(a(x)\eta^1_{2x}\right)_x+\left(B\sqrt{a}\eta^1_2\right)_x
            + D_2 F_1(y_1,y_2)\eta_1^1 + D_2 F_2(y_1,y_2)\eta_2^1 + \mathcal{N}_2^1(\eta,\theta) \\
            \qquad =\alpha_1\theta_2\cara_{_{{\mathcal{O}}_{1d}}}, &\ \text{in} \ {Q},\\
            %+F_2'(y_2)\eta^1_2+F_2''(y_2)\theta_2p^{1}_2=\alpha_1\theta_2\cara_{_{{\mathcal{O}}_{1d}}}, &\ \text{in} \ {Q},\\
			\eta^{1}_1(0,t)=\eta^{1}_1(1,t)=\eta^{1}_2(0,t)=\eta^{1}_2(1,t)=0, & \ \text{on} \  {(0,T)},\\
			\eta^{1}_1(T)=0, \ \eta^{1}_2(T)=0, & \ \text{in} \ \Omega,
            
		\end{cases}
	\end{equation}
    where
    \begin{equation*}
        \begin{split}
			\mathcal{N}_1^1(\eta,\theta) &=[D_{11}^2 F_1(y_1,y_2)\theta_1+D_{12}^2 F_1(y_1,y_2)\theta_2]p_1^1
			+[D_{11}^2 F_2(y_1,y_2)\theta_1+D_{12}^2 F_2(y_1,y_2)\theta_2]p_2^1, \\
            \mathcal{N}_2^1(\eta,\theta) &= [D_{21}^2 F_1(y_1,y_2)\theta_1+D_{22}^2 F_1(y_1,y_2)\theta_2]p_1^1
            +[D_{21}^2 F_2(y_1,y_2)\theta_1+D_{22}^2 F_2(y_1,y_2)\theta_2]p_2^1.
        \end{split}
	\end{equation*}

	Then, making $\lambda \to 0$ in \eqref{dividelambda}, we obtain the second derivative, 
	\begin{equation}
		\label{D2J}
		\langle D_1^2 J_1(h,v^1,v^2), (\bar{v}^{1}, \tilde{v}^{1}) \rangle = \int_{0}^{T}\int_{\mathcal{O}_1} \eta_1^1 \tilde{v}^{1} \ dxdt + \mu_1 \int_{0}^{T}\int_{\mathcal{O}_1}\rho_*^2 \bar{v}^{1}\tilde{v}^{1} \ dxdt, \ \forall\bar{v}^{1},\Tilde{v}^{1} \in \mathcal{V}_{1}. 
	\end{equation}
    
	In particular, taking $\Bar{v}^{1}=\Tilde{v}^{1}$, one has
	\begin{equation}
		\label{D2Jigual}
		\langle D_1^2 J_1(h,v^1,v^2), (\bar{v}^{1}, \bar{v}^{1}) \rangle = \int_{0}^{T}\int_{\mathcal{O}_1} \eta_1^1 \bar{v}^{1} \ dxdt + \mu_1 \int_{0}^{T}\int_{\mathcal{O}_1} \rho_*^2|\bar{v}^{1}|^{2} \ dxdt. 
	\end{equation}
	
	Let us prove that there exists a constant $C> 0$ such that
	\begin{equation}
		\label{eq:integral_eta_w1_bounded3}
%		\left|\int_{O_1 \times (0,T)} \eta_1^1 \bar{v}^{1} \ dxdt\right|\leq C \left( 1 +  \|h\|_{\mathcal{V}}+\|y_{0}\| +\sum_{i=1}^{2}\|y_{i,d}\|^{2}_{L^{2}((0,T),L^2(O_{i,d}))}  \right) \|\bar{v}^{1}\|^{2}_{L^2(\mathcal{O}_1\times (0,T))}
		\left|\int_{O_1 \times (0,T)} \eta_1^1 \bar{v}^{1} \ dxdt\right|\leq C  \|\bar{v}^{1}\|^{2}_{L^2(\mathcal{O}_1\times (0,T))}
	\end{equation}
	In fact, using \eqref{sistemaconvexo1}, integration by parts, and \eqref{sistemaconvexo2}, we get
	{ \begin{eqnarray*}
			\int_{O_1 \times (0,T)} \eta^1_1\bar{v}^{1} \ dxdt &=& \int_{Q} \eta_1^1\left(\theta_{1t}-b(t)\left(a(x)\theta_{1x}\right)_x-B\sqrt{a}\theta_{1x}+D_1 F_1(y_1,y_2)\theta_1 + D_2 F_1(y_1,y_2)\theta_2 \right)\\
            %&&+F'_1(y_1)\theta_1\right)\\
            &&\qquad + \eta_2^1\left( \theta_{2t}-b(t)\left(a(x)\theta_{2x}\right)_x-B\sqrt{a}\theta_{2x}+D_1 F_2(y_1,y_2)\theta_1 + %+c_{21}\theta_1+
            D_2 F_2(y_1,y_2)\theta_2 \right)\\
			&=&\int_{Q} \theta_1\left(-\eta^1_{1t}-b(t)\left(a(x)\eta^1_{1x}\right)_x+\left(B\sqrt{a}\eta^1_1\right)_x
            + D_1 F_1(y_1,y_2)\eta_1^1 + D_1 F_2(y_1,y_2)\eta_2^1 \right)\\
            &&\qquad \theta_2\left( 			-\eta^1_{2t}-b(t)\left(a(x)\eta^1_{2x}\right)_x+\left(B\sqrt{a}\eta^1_2\right)_x
            + D_2 F_1(y_1,y_2)\eta_1^1 + D_2 F_2(y_1,y_2)\eta_2^1 \right)\\
			%&=& \int_{Q} \left( \alpha_1\theta_1^2\cara_{_{{\mathcal{O}}_{1d}}} - F_1''(y_1)\theta_1^2p^{1}_1-c_{21}\eta^1_2\theta_1 \right) dxdt
            &=& \int_{Q} \theta_1\left( \alpha_1\theta_1\cara_{_{{\mathcal{O}}_{1d}}} - \mathcal{N}_1^1(\eta,\theta) \right) + \theta_2\left( \alpha_1\theta_2\cara_{_{{\mathcal{O}}_{1d}}} - \mathcal{N}_2^1(\eta,\theta) \right).
	\end{eqnarray*}}

    We observe that the linear system \eqref{sistemaconvexo1} is uncoupled. By standard energy estimates (see e.g. Appendix in \cite{GYL-Carleman-2025}) we have
    \begin{equation}\label{eq:energy_theta}
	    \sum_{i=1}^2 \|\sqrt{a} \theta_{i,x}\|^2_{L^\infty(0,T;L^2(\Omega))} + \sum_{i=1}^2 	\|\theta_i\|^2_{L^\infty(0,T;L^2(\Omega))} \leq C \|\widetilde{v}^1\|^2_{L^2(\mathcal{O}_1\times (0,T))}.
    \end{equation}

    On the other hand, \eqref{qlambdaadjunto} with $\lambda=0$ is the same equation as \eqref{optimal2}. Using energy estimates for the optimality system \eqref{optimal1}-\eqref{optimal2} we have
    \begin{equation}\label{eq:energy_p_ik}
        \begin{split}
	       &\sum_{k=1}^2 \|\sqrt{a} p_{k,x}^1\|^2_{L^\infty(0,T;L^2(\Omega))} + \sum_{k=1}^2 	\|p_k^1\|^2_{L^\infty(0,T;L^2(\Omega))} \\
           &\qquad \leq C \left( \|h\|^2_{L^2(\mathcal{O}\times (0,T))} + \sum_{k=1}^2 \|y_k^0\|_{H^1_a(\Omega)} + \sum_{k=1}^2 \|y_{k,d}^1\|_{L^2(\mathcal{O}_{i,d}\times (0,T))}^2 \right).
        \end{split}
    \end{equation}
    The proof of the estimate (\ref{eq:energy_p_ik}) is the same as the one in the Appendix of \cite{GYL-equation-2025}.

    Therefore, 
    \color{red}

    \color{black}

    \begin{eqnarray*}
		\left|\int_{O_1 \times (0,T)} \eta \tilde{v}^{1} \ dxdt\right|&\leq& C \int_{Q} \left( |\theta_1|^2\cara_{_{{\mathcal{O}}_{1d}}} +  [|\theta_1|^2 + 2|\theta_1||\theta_2|+|\theta_2|^2]|p_1^1| \right. \\
        &&\left. \qquad +|\theta_2|^2\cara_{_{{\mathcal{O}}_{1d}}} + [|\theta_1|^2 + 2|\theta_1||\theta_2|+|\theta_2|^2]|p_2^1|  
        %- F_1''(y_1)\|\theta_1|^2|p^{1}_1|-c_{21}|\eta^1_2||\theta_1|
        \right) dxdt\\
        &\leq& C \int_{Q} \left( |\theta_1|^2 + |\theta_2|^2 +  2(|\theta_1|^2+|\theta_2|^2)(|p_1^1|+|p_2^1|)  
        %- F_1''(y_1)\|\theta_1|^2|p^{1}_1|-c_{21}|\eta^1_2||\theta_1|
        \right) dxdt.
    \end{eqnarray*}
    From the energy estimates \eqref{eq:energy_theta} and \eqref{eq:energy_p_ik} we have
    \begin{eqnarray*}
	    \left|\int_{O_1 \times (0,T)} \eta \tilde{v}^{1} \ dxdt\right|	&\leq & C(\|\theta_1\|_{L^2(Q)}^{2} + \|\theta_2\|_{L^2(Q)}^{2}) \\
        &&+ C\int_0^T \left( \|\theta_1(t)\|_{L^2(\Omega)}^2 + \|\theta_2(t)\|_{L^2(\Omega)}^2 \right) %\|F''_1\|_{L^\infty}\|\theta_1\|_{L^2(Q)}^{2} 
        \left(\|p_1^1(t)\|_{L^\infty(\Omega)} + \|p_2^1(t)\|_{L^\infty(\Omega)} \right)\\
	%&&+\|c_{21}\|_{L^\infty}\|\eta^1_2\|_{L^2(Q)}^{2}\|\theta_1\|_{L^2(Q)}^{2}\\
		&\leq& C\left(\|\theta_1\|_{L^\infty(0,T;L^2(\Omega))}^{2} + \|\theta_2\|_{L^\infty(0,T;L^2(\Omega))}^{2}\right) \times \\
        &&\qquad \times \left( 1 +  \int_0^T \left(\|p_1^1(t)\|_{H^1_a(\Omega)} + \|p_2^1(t)\|_{H^1_a(\Omega)} \right) \right) \\
        &\leq& 
        C\left(1+\|p_1^1\|^{2}_{L^\infty(0,T;H^1_a(\Omega))}+\|p_2^1\|^{2}_{L^\infty(0,T;H^1_a(\Omega))} %+\|y_{1,d}\|^{2}_{L^{2}(O_{1,d}\times(0,T))} +
        %+\|\sqrt{a}p_{1x}^1\|^{2}_{L^\infty(0,T;L^{2}(\Omega))}+\|\sqrt{a}p_{2x}^1\|^{2}_{L^\infty(0,T;L^{2}(\Omega))}
        \right)\|\bar{v}^{1}\|^{2}\\
		&\leq&C\left( 1 + \|h\|_{\mathcal{V}}^2 + \|y_{1,d}\|^{2}_{L^{2}(O_{1,d}\times(0,T))}+\|y_{0}\|_{H^1_a(\Omega)}\right)\|\bar{v}^{1}\|^{2}_{\mathcal{V}_{1}}\\
		&\leq&C \|\widetilde{v}^1\|^2_{L^2(\mathcal{O}_1\times (0,T))},
	\end{eqnarray*} 
	where $C$ is a positive constant independent of $\mu_1$ and $\mu_2$.\\

	We can use \eqref{D2Jigual}, to obtain the following inequality
		\begin{equation}
		\langle D_1^2 J_1(h,v^1,v^2), (\bar{v}^{1}, \bar{v}^{1}) \rangle \geq  \mu_1 \int_{0}^{T}\int_{\mathcal{O}_1} \rho_*^2|\widetilde{v}^{1}|^{2} \ dxdt -C \|\widetilde{v}^1\|^2_{L^2(\mathcal{O}_1\times (0,T))}. 
	\end{equation}
	
	Using the definition of $\rho_*(t)$, we get $\rho_*^{-2}\leq e^{2\varphi_*}\leq e^{s\varphi}\leq C$\\
	\begin{equation}
		\langle D_1^2 J_1(h,v^1,v^2), (\bar{v}^{1}, \bar{v}^{1}) \rangle \geq  (\mu_1  -C) \|\bar{v}^1\|^2_{L^2(\mathcal{O}_1\times (0,T))}. 
	\end{equation}
							
    In a similar way, we can prove that there exists a positive constant $C$ independent of $\mu_1$ and $\mu_2$ such that
	\begin{equation}
	    \langle D_2^2 J_1(h,v^1,v^2), (\bar{v}^{2}, \bar{v}^{2}) \rangle \geq  (\mu_2 -C) \|\bar{v}^2\|^2_{L^2(\mathcal{O}_2\times (0,T))}. 
    \end{equation}

Taking $\mu_i$ sufficiently large, then the functional $J_i$, $i = 1, 2$ given by \eqref{eqprin24} is convex
and therefore the pair $(v^1,v^2)$ is a Nash equilibrium.

\end{proof}

    In the next sections, we prove the Carleman estimates which will allow us to prove the global null controllability of the linearized optimality system (see \eqref{eq:linearized_system_y} and \eqref{eq:linearized_system_p} bellow). In turn, this will allow us to apply a general inverse mapping theorem to get local null controllability. 
    
	\section{Carleman Estimates}
    \label{sec:carleman_estimates}
	Now, let us present a Carleman estimate that will play an important role in the next section.\\
	\noindent Let $O'=(\alpha',\beta') \subset\subset O$ and let $\Psi : [0,1] \to \R$ be a $C^2$ function such that
	\begin{equation}\label{ecpsi3}
		\Psi(x) = 
		\begin{cases}
			\int\limits_0^x \frac{s}{a(s)} ds, \quad x \in [0,\alpha'), \\
			-\int\limits_{\beta'}^x \frac{s}{a(s)} ds, \quad x \in [\beta',1].
		\end{cases}
	\end{equation}
	For $\lambda \geq \lambda_0$ define the functions
	\begin{eqnarray}\label{ecpsi23}
		\theta(t) = \frac{1}{(t(T-t))^4}, \qquad \eta(x) &=& e^{\lambda(|\Psi|_\infty + \Psi)},  \qquad \sigma(x,t) = \theta(t) \eta(x)\nonumber\\
		\varphi(x,t) &=& \theta(t) (e^{\lambda(|\Psi|_\infty + \Psi)}-e^{3\lambda|\Psi|_\infty}).
	\end{eqnarray}
	Let us also consider the linear system   
	 First we show a Carleman estimate for the linear equation
	\begin{equation}\label{adjunto2_texto_art3}
		\begin{cases}
			v_t+b(t)\left(a(x)v_x\right)_x +d(x,t)\sqrt{a}v_x+c(x,t)v=h(x,t), & \ \ \ \text{in} \ \ \ {Q}\\
			v=0, & \ \ \ \text{in} \ \ \ {\Sigma}\\
			v(T) = v_T, \ \ \ \text{in} \ \ \ (0,1)
		\end{cases}
	\end{equation}
	where $b$ is a continuous function in $[0,T]$ and we denote
	$$
	m := \min\limits_{t \in [0,T]} b(t), \qquad M :=\max\limits_{t \in [0,T]} b(t).
	$$
	Moreover we suppose that $\frac{b'(t)}{b(t)} \leq C_b$, for some constant $C_b>0$. 
	\begin{propo}\label{prop_carleman_13}
		There exist $C>0$ and $\lambda_0,  s_0>0$ such that every solution $v$ of (\ref{adjunto2_texto_art3}) satisfies, for all
		$s\geq s_0$ and $\lambda\geq \lambda_0$,
		\begin{equation}\label{ecjv343}
			\int_{0}^{T}\int_{0}^{1}e^{2s\varphi}\left(  (s\lambda)\sigma a b^2 v_x^2+(s\lambda)^2\sigma^2b^2 v^2 \right)\leq C\left(	\int_{0}^{T}\int_{0}^{1} e^{2s\varphi} |h|^2+(\lambda s)^3\int_{0}^{T}\int_{\omega}e^{2s\varphi} \sigma^3  v^2  \right)
		\end{equation}
	\end{propo}
	\begin{proof}
    It is the main result in \cite{GYL-Carleman-2025}.
	\end{proof}

    \subsection{Linearized system}

    We compute formally the derivative a zero of the map 
	$$\mathcal{T}(y_1,y_2,p^1_1,p^2_1,p_2^1,p_2^2)=\mathcal{T}(y_1,y_2,p^j_i)=[\mathcal{T}^1_0(y_1,y_2,p^j_i),\mathcal{T}^2_0(y_1,y_2,p^j_i),\mathcal{T}^1_i(y_1,y_2,p^j_i),\mathcal{T}^2_i(y_1,y_2,p^j_i)]$$
    where, for $i=1,2$,
	\begin{eqnarray*}
		\mathcal{T}^1_0(y_1,y_2,p^j_i)&=&y_{1t}-b(t)\left({a}(x)y_{1x}\right)_x-B\sqrt{a}y_{1x}+F_1(y_1,y_2)-{h}\cara_{_{{\mathcal{O}}}}+\frac{1}{\mu_1}\rho_*^{-2}p_1^1\cara_{_{{\mathcal{O}}_1}}+\frac{1}{\mu_2}\rho_*^{-2}p_1^2\cara_{_{{\mathcal{O}}_2}},\\
		\mathcal{T}^2_0(y_1,y_2,p^j_i)&=&y_{2t}-b(t)\left({a}(x)y_{2x}\right)_x-B\sqrt{a}y_{2x}+F_2(y_1,y_2),\\
		\mathcal{T}^1_i(y_1,y_2,p^j_i)&=&-p^i_{1t}-b(t)\left(a(x)p^i_{1x}\right)_x+\left(B\sqrt{a}p^i_1\right)_x+D_1 F_1(y_1,y_2)p^i_1+D_1 F_2(y_1,y_2)p^i_2-\alpha_i y_1\cara_{_{{\mathcal{O}}_{id}}},  \\
		\mathcal{T}^2_i(y_1,y_2,p^j_i)&=&-p^i_{2t}-b(t)\left(a(x)p^i_{2x}\right)_x+\left(B\sqrt{a}p^i_2\right)_x+D_2 F_1(y_1,y_2)p^i_1 + D_2 F_2(y_1,y_2)p^i_2-\alpha_i y_2\cara_{_{{\mathcal{O}}_{id}}}.
	\end{eqnarray*}
	Thus, $D\mathcal{T}(0,0,0,0,0,0)(\overline{y}_1,\overline{y}_2,\overline{p}^1_1,\overline{p}^1_2,\overline{p}^2_1,\overline{p}^2_2)=\lim\limits_{
		\lambda\rightarrow 0}\frac{\mathcal{T}(\lambda(\overline{y}_1,\overline{y}_2,\overline{p}^1_1,\overline{p}^1_2,\overline{p}^2_1,\overline{p}^2_2))-\mathcal{T}(0,0,0,0,0,0)}{\lambda}$. Removing the
	tilde in the notations, we get the  following linearized system,	
	
		\begin{equation}%\label{optimal133_linearized1}
        \label{eq:linearized_system_y}
		\begin{cases}
			y_{1t}-b(t)\left({a}(x)y_{1x}\right)_x-B\sqrt{a}y_{1x}+D_1 F_1(0,0)y_1 + D_2 F_1(0,0)y_2\\
			\qquad={h}\cara_{_{{\mathcal{O}}}}-\frac{1}{\mu_1}\rho_*^{-2}p_1^1\cara_{_{{\mathcal{O}}_1}}-\frac{1}{\mu_2}\rho_*^{-2}p_1^2\cara_{_{{\mathcal{O}}_2}}, & \ \ \ \text{in} \ \ \ {Q},\\
			y_{2t}-b(t)\left({a}(x)y_{2x}\right)_x-B\sqrt{a}y_{2x}+D_1 F_2(0,0) y_1+D_2 F_2(0,0)y_2=0 & \ \ \ \text{in} \ \ \ {Q},\\
			y_1(0,t)=y_2(1,t)=y_2(0,t)=y_2(1,t)=0, & \ \ \ \text{on} \ \ \ (0,T),\\
			y_1(0)=y_1^0(x), \ y_2(0)=y_2^0(x) & \ \ \ \text{in} \ \ \ \Omega.
		\end{cases}
	\end{equation}	
	and, for $i=1,2$,
	\begin{equation}%\label{optimal233_linearized2}
    \label{eq:linearized_system_p}
		\begin{cases}
			%-p^i_{1t}-b(t)\left(a(x)p^i_{1x}\right)_x+\left(B\sqrt{a}p^i_1\right)_x+F_1'(0)p^i_1+c_{21}p^i_2=\alpha_i y_1 \cara_{_{{\mathcal{O}}_{id}}}, &\ \ \ \text{in} \ \ \ {Q},\\
            -p^i_{1t}-b(t)\left(a(x)p^i_{1x}\right)_x+\left(B\sqrt{a}p^i_1\right)_x+D_1 F_1(0,0)p^i_1+D_1 F_2(0,0)p^i_2=\alpha_i y_1 \cara_{_{{\mathcal{O}}_{id}}}, &\ \ \ \text{in} \ \ \ {Q},\\
			-p^i_{2t}-b(t)\left(a(x)p^i_{2x}\right)_x+\left(B\sqrt{a}p^i_2\right)_x+D_2 F_1(0,0)p^i_1 + D_2 F_2(0,0)p^i_2=\alpha_i y_2 \cara_{_{{\mathcal{O}}_{id}}}, &\ \ \ \text{in} \ \ \ {Q},\\
			p^i_1(0,t)=p^i_1(1,t)=p^i_2(0,t)=p^i_2(1,t)=0, & \ \ \ \text{on} \ \ \ (0,T),\\
			p^i_1(T)=0, \ p^i_2(T)=0, & \ \ \ \text{in} \ \ \ \Omega.
		\end{cases}
	\end{equation}

	\noindent We introduce the notation
	$$
	I(\varsigma) = \int_Q e^{2s\varphi}((s\lambda)\sigma b^2 a \varsigma_x^2 + (s\lambda)^2 \sigma^2 b^2 \varsigma^2)dxdt.
	$$

In this section we establish an observability inequality that allows us to prove the null
controllability of system 
%\eqref{optimal1}--\eqref{optimal2}. 
\eqref{eq:linearized_system_y}--\eqref{eq:linearized_system_p}. 
We start by proving an observability inequality for the adjoint systems associated
		\begin{equation}\label{optimaladj1}
		\begin{cases}
			-\phi_{1t}-b(t)\left({a}(x)\phi_{1x}\right)_x+(B\sqrt{a}\phi_1)_{x}+c_{11}\phi_1+c_{21}\phi_2=\alpha_1\psi_1^1\cara_{_{{\mathcal{O}}_{1d}}}+\alpha_2\psi_1^2\cara_{_{{\mathcal{O}}_{2d}}}+F, & \ \ \ \text{in} \ \ \ {Q}\\
			-\phi_{2t}-b(t)\left({a}(x)\phi_{2x}\right)_x+(B\sqrt{a}\phi_2)_{x}+c_{12}\phi_1+c_{22}\phi_2=\alpha_1\psi_1^2\cara_{_{{\mathcal{O}}_{1d}}}+\alpha_2\psi_2^2\cara_{_{{\mathcal{O}}_{2d}}}+\overline{F},& \ \ \ \text{in} \ \ \ {Q}\\
			\phi_1(0,t)=\phi_1(1,t)=\phi_2(0,t)=\phi_2(1,t)=0, & \ \ \ \text{on} \ \ \ (0,T)\\
			\phi_1(T)=\phi_1^T, \ \phi_2(T)=\phi_2^T & \ \ \ \text{in} \ \ \ \Omega
		\end{cases}
	\end{equation}	
	and
	\begin{equation}\label{optimaladj2}
		\begin{cases}
			\psi^i_{1t}-b(t)\left(a(x)\psi^i_{1x}\right)_x-B\sqrt{a}\psi^i_{1x}+c_{11}\psi^i_1+c_{12}\psi^i_2=-\frac{1}{\mu_i}\rho_*^{-2}\phi_1\cara_{_{{\mathcal{O}}_{i}}}+F_i, &\ \ \ \text{in} \ \ \ {Q}\\
			\psi^i_{2t}-b(t)\left(a(x)\psi^i_{2x}\right)_x-B\sqrt{a}\psi^i_{2x}+c_{21}\psi^i_1+c_{22}\psi^i_2=\overline{F}_i, &\ \ \ \text{in} \ \ \ {Q}\\
			\psi^i_1(0,t)=\psi^i_1(1,t)=\psi^i_2(0,t)=\psi^i_2(1,t)=0, & \ \ \ \text{on} \ \ \ (0,T)\\
			\psi^i_1(0)=0, \ \psi^i_2(0)=0, & \ \ \ \text{in} \ \ \ \Omega
		\end{cases}, \ \ \ i=1,2.
	\end{equation}	
	where $c_{ij}=D_j F_j(0,0)$, $c_{ij} \in L^{\infty}(Q)$, for $i,j=1,2$, and $\psi_1^T, \phi_2^T\in L^2{\Omega}$\\
	Before continuing, we consider that the following result is useful for the rest of the thesis.
%	\begin{lema}[Caccioppoli’s inequality]\label{cacio}
%Let $\mathcal{O}'$ be a subset of $\mathcal{O}$ such that $\mathcal{O}'\Subset \mathcal{O}$. Let a and a be the solution of (5.16) and (5.17)
%respectively. Then, there exists a positive constant $C$ such that
%	\end{lema}
%	\begin{proof}
%	
%	\end{proof}
%	
%	
%	
%	
	
	Since the first assumption of ${\mathcal{O}}_{1d}={\mathcal{O}}_{2d}$ holds and if we set , $\rho_i=\alpha_1\psi_i^1+\alpha_2\psi_i^2, \ i = 1, 2$, then we obtain
    \begin{equation}\label{optimaladj1_simp}
	\begin{cases}
		-\phi_{1t}-b(t)\left({a}(x)\phi_{1x}\right)_x+(B\sqrt{a}\phi_1)_{x}+c_{11}\phi_1+c_{21}\phi_2=\rho_1\cara_{_{{\mathcal{O}}_{d}}}+G_0, &  \text{in} \ \ \ {Q}\\
		-\phi_{2t}-b(t)\left({a}(x)\phi_{2x}\right)_x+(B\sqrt{a}\phi_2)_{x}+c_{12}\phi_1+c_{22}\phi_2=\rho_2\cara_{_{{\mathcal{O}}_{d}}}+\overline{G}_0,& \text{in} \ \ \ {Q}\\
		\rho_{1t}-b(t)\left(a(x)\rho_{1x}\right)_x-B\sqrt{a}\rho_{1x}+c_{11}\rho_1 + c_{12}\rho_2=-\rho_*^{-2}  \left( \frac{\alpha_1}{\mu_1}\cara_{_{{\mathcal{O}}_{1}}} +\frac{\alpha_2}{\mu_2}\cara_{_{{\mathcal{O}}_{2}}}  \right)    \phi_1+G, & \text{in} \ {Q}\\
		\rho_{2t}-b(t)\left(a(x)\rho_{2x}\right)_x-B\sqrt{a}\rho_{2x}+c_{21}\rho_1+c_{22}\rho_2=\overline{G}, & \text{in}  \ {Q}\\
		\phi_1(0,t)=\phi_1(1,t)=\phi_2(0,t)=\phi_2(1,t)=0, & \text{on} \ \ \ (0,T)\\
		\rho_1(0,t)=\rho_1(1,t)=\rho_2(0,t)=\rho_2(1,t)=0, & \text{on} \ \ \ (0,T)\\
		\phi_1(T)=\phi_1^T, \ \phi_2(T)=\phi_2^T \ \ \rho_1(0)=0, \ \rho_2(0)=0, & \text{in} \ \ \ \Omega
	\end{cases}
\end{equation}
where $G=\alpha_{1}F_1+\alpha_{2} F_2$ and $\overline{G}=\alpha_{1} \overline{F}_1+\alpha_{2}\overline{F}_2$.

\begin{propo}\label{prop:carleman_sistema}
	There exist positive constants $C$, $\lambda_0$ and $s_0$ such that, for any $s\geq s_0$, $\lambda\geq \lambda_0$ and any $\phi_1^T, \phi_2^T \in L^2((0,T)\times (0, 1))$, the corresponding solution $(\phi_1, \phi_2, \rho_1, \rho_2)$ to \eqref{optimaladj1_simp} satisfies\begin{multline}\label{carlemansistema}
		I(\phi_1)+I(\phi_2)+I(\rho_1)+I(\rho_2)\leq C\int_{0}^{T}\int_{0}^{1}e^{2s\varphi}\lambda^{14}s^{14}\sigma^{14} \left(|{G}_0|^2+|{G}|^2+|\overline{G}_0|^2+|\overline{G}|^2\right)\\
		+ C \int_{0}^{T}\int_{{\mathcal{O}}}s^{28}\lambda^{28} \sigma^{28} e^{2s\varphi} |\phi_{1}|^2
	\end{multline}
\end{propo}

The proof can be found in Appendix \ref{appendix_a}.

In order to get the global null controllability of the linearized system, we need a Carleman inequality with weights that do not vanish at $t = 0$. For that, consider a function $m \in C^\infty([0, T])$ satisfying $m(0) > 0$,
\begin{equation}
	\label{eq:def_m}
	m(t) \geq t^4(T-t)^4, \quad t \in (0,T/2], \qquad\qquad m(t) = t^4(T-t)^4, \quad t \in [T/2,T],    
\end{equation}
and define
$$
    \tau(t) = \frac{1}{m(t)}, \qquad \zeta(x,t) = \tau(t) \eta(x), \qquad A(x,t)  = \tau(t) (e^{\lambda(|\Psi|_\infty + \Psi)}-e^{3\lambda|\Psi|_\infty}).
$$

Moreover, define
\begin{multline*}
    \Gamma_0(\phi_1,\phi_2,\rho_1,\rho_2) = \int_{0}^{T}\int_{0}^{1} e^{2s A}(s\lambda) \zeta b^2 a (|\phi_{1x}|^2 + |\phi_{2x}|^2+|\rho_{1x}|^2 + |\rho_{2x}|^2)dxdt\\
    +\int_{0}^{T}\int_{0}^{1} e^{2s A}(s\lambda)^2 {\zeta}^2 b^2 (|\phi_1|^2 + |\phi_2|^2+|\rho_1|^2 + |\rho_2|^2)dxdt
\end{multline*}
%		$$
%		\Gamma(\phi,\varrho,\tilde \varrho) = \int_Q e^{2s A}((s\lambda) \zeta b^2 a (|\phi_x|^2+|\varrho_x|^2+|\tilde \varrho_x|^2) + (s\lambda)^2 {\zeta}^2 b^2 (|\phi|^2 + |\varrho|^2 + |\tilde \varrho|^2))dxdt.
%		$$
and
\begin{multline*}
	\Gamma(\phi_1,\phi_2,\psi^1_1,\psi^2_1,\psi^1_2,\psi^2_2)= \int_{0}^{T}\int_{0}^{1} e^{2s A}(s\lambda) \zeta b^2 a \left( |\phi_{1x}|^2 + |\phi_{2x}|^2+\sum_{i,j=1}^{2}|\psi_{ix}^j|^2 \right)dxdt\\
	+\int_{0}^{T}\int_{0}^{1} e^{2s A}(s\lambda)^2 {\zeta}^2 b^2 \left(|\phi_1|^2 + |\phi_2|^2+\sum_{i,j=1}^{2}|\psi_{i}^j|^2 \right)dxdt.
\end{multline*}
        
\begin{propo} \label{propocarleman23}%\cite{DemarqueLimacoViana_deg_sys2020}
	There exist positive constants $C$, $\lambda_0$ and $s_0$ such that, for any $s\geq s_0$, $\lambda\geq \lambda_0$ and any $\phi_1^T, \phi_2^T \in L^2((0,T)\times (0, 1))$, the corresponding solution $(\phi_1, \phi_2, \rho_1, \rho_2)$ to \eqref{optimaladj1_simp} verifies
	\begin{multline*}
        \Gamma_0(\phi_1,\phi_2,\rho_1,\rho_2) \leq C \left( \int_{0}^{T}\int_{0}^{1}e^{2sA}\lambda^{14}s^{14}\zeta^{14} \left(|{G}_0|^2+|{G}|^2+|\overline{G}_0|^2+|\overline{G}|^2\right)\right. \\
        \left.+ C \int_{0}^{T}\int_{{\mathcal{O}}}e^{2sA} s^{28}\lambda^{28} \zeta^{28} |\phi_{1}|^2\right).
	\end{multline*}
\end{propo}

The proof is given in Appendix \ref{appendix_b}.
%\begin{proof}
%			In the same way as the previous proposition, we can work in a general for, using the following system \eqref{optimaladj1_simp}.\\
%\end{proof}

%	+\int_{0}^{T}\int_{0}^{1}\chi e^{2s\varphi}\lambda^8s^8\sigma^8 \phi_2\rho_1\cara_{\mathcal{O}_{d}}=
Returning to the original variables, it is immediate to get the following result:

\begin{propo}\label{prop:carleman3}
	There exist positive constants $C$, $\lambda_0$ and $s_0$ such that, for any $s \geq s_0$, $\lambda \geq \lambda_0$ and any  $\phi_1^T, \phi_2^T \in L^2((0,T)\times (0, 1))$, the corresponding solution $(\phi_1,\phi_2,\psi^1_1,\psi^2_1,\psi^1_2,\psi^2_2)$ of \eqref{optimaladj1} and \eqref{optimaladj2} satisfies
    \vspace*{-0.5cm}
    \begin{multline}
        \label{Carleman for eq:adjoint_optimality_system3}
        \Gamma(\phi_1,\phi_2,\psi^1_1,\psi^2_1,\psi^1_2,\psi^2_2) \leq C \left[ \int_Q e^{2s A} s^{14} \lambda^{14} \zeta^{14} \left(\sum_{i=0}^{2}|F_i|^2+|\overline{F}_i|^2\right)dxdt\right.\\
        \left.+ \int_{0}^{T}\int_{\mathcal{O}} e^{2s A} s^{28} \lambda^{28} \zeta^{28} |\phi_1|^2 dxdt\right].
	\end{multline}
\end{propo}

As a corollary, we get the following observability inequality.
\begin{coro} %\cite{DemarqueLimacoViana_deg_sys2020}
	\label{cor:observability33}
	There exist positive constants $C$, $\lambda_0$ and $s_0$ such that, for any $s \geq s_0$, $\lambda \geq \lambda_0$ and any  $\phi_1^T, \phi_2^T \in L^2((0,T)\times (0, 1))$, the corresponding solution $(\phi_1,\phi_2,\psi^1_1,\psi^2_1,\psi^1_2,\psi^2_2)$ of \eqref{optimaladj1} and \eqref{optimaladj2} with $F_i=\overline{F}_i = 0, \ \ i=0,1,2$, satisfies
	\begin{equation}\label{Observability3}
		\|\phi_1(0)\|^2_{L^2(0,1)}+\|\phi_2(0)\|^2_{L^2(0,1)} + \sum_{i,j=1}^{2}\|\psi_i^j(T)\|^2_{L^2(0,1)} \leq C \int_{0}^{T}\int_{O} e^{2s A} s^{28} \lambda^{28} \zeta^{28} |\phi_1|^2dxdt.
	\end{equation}
\end{coro}

The proof is similar to Corollary 1 of \cite{DemarqueLimacoViana_deg_sys2020}.
    %\ref{cor:observability}.

 In the following sections we will need weights which depend only on $t$. Let us define
\begin{equation*}
	\begin{cases}
		A^*(t) = \displaystyle\max_{x \in (0,1)} A(x,t), \qquad \hat A(t) = \displaystyle\min_{x \in (0,1)} A(x,t), \\
		\zeta^*(t) = \displaystyle\max_{x \in (0,1)} \zeta(x,t), \qquad \hat \zeta(t) = \displaystyle\min_{x \in (0,1)} \zeta(x,t),
	\end{cases}
\end{equation*}
where we observe that $A^*(t) < 0$, $\hat A(t) < 0$ and that $\zeta^*(t) / \hat \zeta(t)$ does not depend on $t$ and is equal to some constant $\zeta_0 \in \R$. Moreover, if $\lambda$ is sufficiently large we can suppose
\begin{equation}
	\label{eq:comp_pesos3}
	3 A^*(t) < 2 \hat A(t) < 0.
\end{equation}

 Let us define 
				\begin{multline*}
	\hat{\Gamma}(\phi_1,\phi_2,\psi^1_1,\psi^2_1,\psi^1_2,\psi^2_2)= \int_{0}^{T}\int_{0}^{1} e^{2s \hat{A}}(s\lambda) \hat{\zeta} b^2 a \left( |\phi_{1x}|^2 + |\phi_{2x}|^2+\sum_{i,j=1}^{2}|\psi_{ix}^j|^2 \right)dxdt\\
	+\int_{0}^{T}\int_{0}^{1} e^{2s \hat{A}}(s\lambda)^2 \hat{\zeta}^2 b^2 \left(|\phi_1|^2 + |\phi_2|^2+\sum_{i,j=1}^{2}|\psi_{i}^j|^2 \right)dxdt.
\end{multline*}

Thus, Proposition \ref{prop:carleman3} and Corollary \ref{cor:observability33}  imply directly the following corollary where the weights depend only on $t$.

\begin{coro} %\cite{DemarqueLimacoViana_deg_sys2020}
	\label{cor:carleman_pesos_t3}
	There exist positive constants $C$, $\lambda_0$ and $s_0$ such that, for any $s \geq s_0$, $\lambda \geq \lambda_0$ and any  $\phi_1^T, \phi_2^T \in L^2((0,T)\times (0, 1))$, the corresponding solution $(\phi_1,\phi_2,\psi^1_1,\psi^2_1,\psi^1_2,\psi^2_2)$ of \eqref{optimaladj1} and \eqref{optimaladj2} satisfies
	\begin{multline}\label{pesos_t_Carleman for eq:adjoint_optimality_system3}
		\|\phi_1(0)\|^2_{L^2(0,1)}+\|\phi_2(0)\|^2_{L^2(0,1)}+\hat{\Gamma}(\phi_1,\phi_2,\psi^1_1,\psi^2_1,\psi^1_2,\psi^2_2)\\
		\leq C \left[ \int_Q e^{2s A^*} s^{14} \lambda^{14} (\zeta^*)^{14} \left(\sum_{i=0}^{2}|F_i|^2+|\overline{F}_i|^2\right)dxdt\right.
		\left.+ \int_{0}^{T}\int_{\mathcal{O}} e^{2s A^*} s^{28} \lambda^{28} (\zeta^*)^{28} |\phi_1|^2 dxdt\right]
	\end{multline}
\end{coro}

\subsection{Global null controllability of the linearized system}

Let us define the weights: 
\begin{equation}\label{eq:weights_rhos3}
	\left\{ \begin{array}{l}
		\rho_0 = e^{-sA^*}  (\zeta^*)^{-7}, \qquad   \rho_1 = e^{-sA^*}  (\zeta^*)^{-14},\\  \rho_2 = e^{-3sA^*/2}  \hat \zeta^{-1},  \qquad \hat{\rho} = e^{-sA^*} (\zeta^*)^{-21/2},      
	\end{array}\right.
\end{equation}
satisfying
\begin{equation}\label{eq:compara_rhos3}
	\rho_{1}\leq C \hat{\rho}\leq C\rho_{0}\leq C\rho_{2}, \qquad \hat{\rho}^{2}=\rho_{1}\rho_{0} \qquad \textcolor{black}{\text{and} \qquad \rho_2 \leq C \rho_1^2}.   
\end{equation}

 In particular, Corollary \ref{cor:carleman_pesos_t3} and \eqref{eq:comp_pesos3} imply 
\begin{multline}\label{eq:carleman_simples3}
\|\phi_1(0)\|^2_{L^2(0,1)}+\|\phi_2(0)\|^2_{L^2(0,1)}+\int_{0}^{T}\int_{0}^{1} \rho_2^{-2} \left(|\phi_1|^2 + |\phi_2|^2+\sum_{i,j=1}^{2}|\psi_{i}^j|^2 \right)\\
\leq C \left[ \int_Q \rho_0^{-2} \left(\sum_{i=0}^{2}|F_i|^2+|\overline{F}_i|^2\right)dxdt\right.
\left.+ \int_{0}^{T}\int_{\mathcal{O}} \rho_1^{-2} |\phi_1|^2 dxdt\right].
\end{multline}

 In the following theorem we show the global null controllability of the linearized system \eqref{eq:linearized_system_y} and \eqref{eq:linearized_system_p}. In particular, since the weight $\rho_0$ blows up at $t=T$, then \eqref{estimate for solution3} shows that $y(\cdot,T)=0$ in $[0,1]$. The same vanishing conclusion holds for $p^i$, $i=1,2$ (which are used to define the follower's controls $v^i$) and the leader control $h$.
%Furthermore, estimates \eqref{des Proposition 5} and \eqref{des Proposition 6} give additional estimates verified by the derivatives of the states $(y,p^1,p^2)$ that are needed for the local null controllability of the nonlinear system in Section \ref{control for nonlinear system}.

\begin{teo}\label{theorem case linear3}
	If $y^{0}_1, y_2^0\in H^1_a(\Omega)$ and  $\rho_2 F_i,  \rho_2\overline{F}_i \in L^2(Q), i=0,1,2$, then there exists a control $h\in L^{2}(O\times (0,T))$ with associated states $y_1, p^1_1, p^2_1,  y_2, p^1_2, p^2_2 \in C^{0}([0,T];L^{2}(0,1))\cap L^{2}(0,T;H^{1}_{a})$  from \eqref{eq:linearized_system_y} and \eqref{eq:linearized_system_p} such that
	\begin{multline}\label{estimate for solution3}
\int_Q \rho_0^2 \left(|y_1|^2+|y_2|^2\right) dxdt+ \int_Q \rho_0^2 \left(\sum_{i,j=1}^2|p_i^j|^2\right)dxdt+ \int_{0}^{T}\int_{\mathcal{O}} \rho_1^2 |h|^2dxdt \leq
C \kappa_{0}(F_i,\overline{F}_i,y^0_1,y_2^0),   
	\end{multline}
	where $$\kappa_{0}(F_i,\overline{F}_i,y^0_1,y_2^0)=  \|y_1^0\|^2_{L_2(0,1)}+ \|y_2^0\|^2_{L_2(0,1)}+\sum_{i=0}^{2}\left(\|\rho_2 F_i\|^2_{L^2(Q)}+\|\rho_2 \overline{F}_i\|^2_{L^2(Q)}\right).$$
	 In particular, $y_1(x,T)=y_2(x,T)=0$, for all $x\in [0,1]$.
	
	%Furthermore, 
	%\begin{equation}\label{des Proposition 5}
	%    \begin{array}{c}
		%\displaystyle\sup_{[0,T]}(\hat{\rho}^{2}\|y\|^{2}_{L^{2}(0,1)}) + \displaystyle\sup_{[0,T]}(\hat{\rho}^{2}\|p^{1}\|^{2}_{L^{2}(0,1)}) + \displaystyle\sup_{[0,T]}(\hat{\rho}^{2}\|p^{2}\|^{2}_{L^{2}(0,1)})\vspace{0.1cm}\\
		%+\displaystyle\int_{Q}\hat{\rho}^{2} a(x)(|y_{x}|^{2} + |p^{1}_{x}|^{2}+|p^{2}_{x}|^{2})dxdt\ \leq C \kappa_{0}(H,H_{1},H_{2},y_{0})  
		%    \end{array}
	%\end{equation}
	%and, if $y_{0}\in H^{1}_{a}(0,1)$ 
	%\begin{equation}\label{des Proposition 6}
	%    \begin{array}{c}
		%\displaystyle\sup_{[0,T]}(\rho_{1}^{2}\|\sqrt{a}y_{x} \|^{2}_{L^{2}(0,1)})+ \displaystyle\sup_{[0,T]}(\rho_{1}^{2}\|\sqrt{a}p^{1}_{x} \|^{2}_{L^{2}(0,1)}) + \displaystyle\sup_{[0,T]}(\rho_{1}^{2}\|\sqrt{a}p^{2}_{x} \|^{2}_{L^{2}(0,1)})\\
		%+ \displaystyle\int_{Q}\rho_{1}^{2}(|y_{t}|^{2}+|p^{1}_{t}|^{2}+|p^{2}_{t}|^{2}+|(a(x)y_{x})_{x}|^{2} + |(a(x)p^{1}_{x})_{x}|^{2}+ |(a(x)p^{2}_{x})_{x}|^{2})dxdt\\
		%\leq C \kappa_{1}(H,H_{1},H_{2},y_{0}),  
		%    \end{array}
	%\end{equation}
	%where $\kappa_{1}(H,H_{1},H_{1},y_{0})= |\rho_2 H|^2_{L^2(Q)} + |\rho_2 H_1|^2_{L^2(Q)} + |\rho_2 H_2|^2_{L^2(Q)} + \|y_0\|^2_{H^{1}_{a}}$. 
\end{teo}

%\textcolor{red}{Los coeficientes del sistema de optimalidad de orden uno y cero no son iguales para ambas ecuaciones, como en el caso de Miguel, para definir un único operador $L$ y $L^*$. Como salvar?}

\begin{proof}
	Let us denote by 
	\begin{equation}\label{operador}
		\mathcal{L}\begin{pmatrix}
			\varphi_1\\
			\varphi_2
		\end{pmatrix}=\begin{pmatrix}
		L_1(\varphi_1,\varphi_2)\\
		L_2(\varphi_1,\varphi_2)
		\end{pmatrix}=\begin{pmatrix}
		\varphi_{1t} - b(t)(a(x)\varphi_{1x})_{x} - B\sqrt{a} \varphi_{1x} + c_{11}(x,t)\varphi_1 + c_{12}(x,t)\varphi_2\\
\varphi_{2t} - b(t)(a(x)\varphi_{2x})_{x} - B\sqrt{a} \varphi_{2x} + c_{21}(x,t)\varphi_1+c_{22}(x,t)\varphi_2
		\end{pmatrix}
	\end{equation}
	and
		\begin{equation}\label{operadoradjun}
		\mathcal{L}^\ast\begin{pmatrix}
			\varphi_1\\
			\varphi_2
		\end{pmatrix}=\begin{pmatrix}
			L_1^\ast(\varphi_1,\varphi_2)\\
			L_2^*(\varphi_1,\varphi_2)
		\end{pmatrix}=\begin{pmatrix}
			-\varphi_{1t} - b(t)(a(x)\varphi_{1x})_{x}+ (B\sqrt{a}\varphi_1)_{x} + c_{11}(x,t)\varphi_1+c_{21}(x,t)\varphi_2\\
			-\varphi_{2t} - b(t)(a(x)\varphi_{2x})_{x}+ (B\sqrt{a}\varphi_2)_{x} + c_{12}(x,t)\varphi_1 + c_{22}(x,t)\varphi_2
		\end{pmatrix}
	\end{equation}
	%(\textcolor{green}{Temos que ver sobre a função regular que substitui a função caracteristica. Se vamos defini-la no inicio ou somente aqui, como no trabalho do Danny}).
	Then, we define 
	\begin{eqnarray*}
\mathcal{X}_{0}& =&\left\{  (\phi_1,\psi^{1}_1,\psi^{2}_1;\phi_2,\psi^{1}_2,\psi^{2}_2)\in \left[C^{2}(\overline{Q})\right]^{6}; \phi_1(0,t)=\phi_1(1,t)=\phi_2(0,t)=\phi_2(1,t) = 0\ \right.\\
  && \ \ \ \ \ \left. \text{a.e in}\ (0,T),\ \psi_{i}^j(0,t)=\psi_{i}^j(1,t)=0\, \text{a.e in}\, (0,T),  \text{and} \   \psi_i^{j}(\cdot,0)=0\ \text{in}\ \Omega,  \ i,j=\lbrace 1, 2\rbrace\right\}
	\end{eqnarray*}
	and the application $b:\mathcal{X}_{0}\times \mathcal{X}_{0}\rightarrow \mathbb{R}$ given by
	\begin{eqnarray}\label{eq:def_b3}
		&&b\left( (\widetilde{\phi}_1,\widetilde{\psi}^{1}_1,\widetilde{\psi}^{2}_1;\widetilde{\phi}_2,\widetilde{\psi}^{1}_2,\widetilde{\psi}^{2}_2),(\phi_1,\psi^{1}_1,\psi^{2}_1;\phi_2,\psi^{1}_2,\psi^{2}_2)   \right)\nonumber\\
		&=&\int_{Q}\rho_{0}^{-2}(L_1^{\ast}(\widetilde{\phi}_1,\widetilde{\phi}_2)-\alpha_{1}\widetilde{\psi}^{1}_1 \cara_{\mathcal{O}_{d}}-\alpha_{2}\widetilde{\psi}^{2}_1\cara_{\mathcal{O}_{d}})(L_1^{\ast}({\phi}_1,{\phi}_2)-\alpha_{1}{\psi}^{1}_1 \cara_{\mathcal{O}_{d}}-\alpha_{2}{\psi}^{2}_1\cara_{\mathcal{O}_{d}})\ dx\ dt \nonumber	\\
			&&+\int_{Q}\rho_{0}^{-2}(L_2^{\ast}(\widetilde{\phi}_1,\widetilde{\phi}_2)-\alpha_{1}\widetilde{\psi}^{1}_2 \cara_{\mathcal{O}_{d}}-\alpha_{2}\widetilde{\psi}^{2}_2\cara_{\mathcal{O}_{d}})(L_2^{\ast}({\phi}_1,{\phi}_2)-\alpha_{1}{\psi}^{1}_2 \cara_{\mathcal{O}_{d}}-\alpha_{2}{\psi}^{2}_2\cara_{\mathcal{O}_{d}})\ dx\ dt	\nonumber\\
		&&+\int_{Q}\rho_{0}^{-2}\left(L_1(\widetilde{\psi}^{1}_1,\widetilde{\psi}^{1}_2) + \frac{1}{\mu_{1}}\rho_*^{-2}\widetilde{\phi}_1\cara_{\mathcal{O}_{1}}\right) \left(L_1({\psi}^{1}_1,{\psi}^{1}_2) + \frac{1}{\mu_{1}}\rho_*^{-2}\phi_1\cara_{\mathcal{O}_{1}}\right)\ dx\ dt\nonumber\\
		&&+\int_{Q}\rho_{0}^{-2}L_2(\widetilde{\psi}^{1}_1,\widetilde{\psi}^{1}_2)L_2({\psi}^{1}_1,{\psi}^{1}_2) \ dx\ dt\nonumber\\
			&&+\int_{Q}\rho_{0}^{-2}\left(L_1(\widetilde{\psi}^{2}_1,\widetilde{\psi}^{2}_2) + \frac{1}{\mu_{2}}\rho_*^{-2}\widetilde{\phi}_1\cara_{\mathcal{O}_{2}}\right) \left(L_1({\psi}^{2}_1,{\psi}^{2}_2) + \frac{1}{\mu_{2}}\rho_*^{-2}\phi_1\cara_{\mathcal{O}_{2}}\right)\ dx\ dt\nonumber\\
		&&+\int_{Q}\rho_{0}^{-2}L_2(\widetilde{\psi}^{2}_1,\widetilde{\psi}^{2}_2)L_2({\psi}^{2}_1,{\psi}^{2}_2) \ dx\ dt\nonumber\\
		&&+\int_{0}^{T}\int_{\mathcal{O}}\rho_{1}^{-2}\widetilde{\phi}_1\phi_1\ dx\ dt,\ \ \qquad \forall (\widetilde{\phi}_1,\widetilde{\psi}^{1}_1,\widetilde{\psi}^{2}_1;\widetilde{\phi}_2,\widetilde{\psi}^{1}_2,\widetilde{\psi}^{2}_2),(\phi_1,\psi^{1}_1,\psi^{2}_1;\phi_2,\psi^{1}_2,\psi^{2}_2) \in \mathcal{X}_{0},
	\end{eqnarray}
	which is bilinear on $\mathcal{X}_{0}$ and also defines an inner product. Indeed, taking $(\widetilde{\phi}_1,\widetilde{\psi}^{1}_1,\widetilde{\psi}^{2}_1;\widetilde{\phi}_2,\widetilde{\psi}^{1}_2,\widetilde{\psi}^{2}_2)=(\phi_1,\psi^{1}_1,\psi^{2}_1;\phi_2,\psi^{1}_2,\psi^{2}_2)$ in \eqref{eq:def_b3}, we have that $b(\cdot,\cdot)$ is positive definite from  \eqref{Observability3}. The other properties are straightforwardly verified.
	
	Let us consider the space $\mathcal{X}$ the completion of $\mathcal{X}_{0}$ for the norm associated to $b(\cdot,\cdot)$ (which we denote by $\|\cdot\|_{\mathcal{X}}$). Then, $b(\cdot,\cdot)$ is a symmetric, continuous and coercive bilinear form on $\mathcal{X}$. \\
	
	Now, let us define the functional  $\ell :\mathcal{X}\rightarrow\mathbb{R}$ as 
	\begin{multline*}
        \langle\ell, (\phi_1,\psi^{1}_1,\psi^{2}_1;\phi_2,\psi^{1}_2,\psi^{2}_2)\rangle = \displaystyle\int_{0}^{1}y^0_{1}\phi_1(0)dx+\int_{0}^{1}y^0_{2}\phi_2(0)dx \\
        +\displaystyle\int_{Q}\left(F_0 \phi_1+\overline{F}_0 \phi_2 + F_{1}\psi^{1}_1+\overline{F}_1 \psi_2^1 +F_{2}\psi^{2}_1+\overline{F}_2 \psi_2^2\right)dx dt.
	\end{multline*}
	
	Note that $\ell$ is a bounded linear form on $\mathcal{X}$. Indeed, applying the Cauchy-Schwartz inequality 
	%$|u\cdot v|\leq |u||v|$, for $u,v\in 
	in $\mathbb{R}^{8}$ and using \eqref{eq:carleman_simples3}, we get
	\begin{equation}\label{l limitado}
		\begin{array}{l}
			|\langle\ell, (\phi,\psi^{1},\psi^{2})\rangle|\leq  |y^{0}_1|_{L^{2}(0,1)}|\phi_1(0)|_{L^{2}(0,1)} +|y^{0}_2|_{L^{2}(0,1)}|\phi_2(0)|_{L^{2}(0,1)}+ |\rho_{2}F_0|_{L^{2}(Q)}|\rho_{2}^{-1}\phi_1|_{L^{2}(Q)} \\
			\hspace{3cm} +|\rho_{2}\overline{F}_0|_{L^{2}(Q)}|\rho_{2}^{-1}\phi_2|_{L^{2}(Q)}+ |\rho_{2}F_{1}|_{L^{2}(Q)}|\rho_{2}^{-1}\psi^{1}_1|_{L^{2}(Q)} +|\rho_{2}\overline{F}_{1}|_{L^{2}(Q)}|\rho_{2}^{-1}\psi^{1}_2|_{L^{2}(Q)}\\
			\hspace{3cm}
			+ |\rho_{2}F_{2}|_{L^{2}(Q)}|\rho_{2}^{-1}\psi^{2}_1|_{L^{2}(Q)} +|\rho_{2}\overline{F}_{2}|_{L^{2}(Q)}|\rho_{2}^{-1}\psi^{2}_2|_{L^{2}(Q)}\\
			\leq C\left(|y^0_{1}|^{2}_{L^{2}(0,1)}+|y^0_{2}|^{2}_{L^{2}(0,1)} + |\rho_{2}F_0|_{L^{2}(Q)} +|\rho_{2}\overline{F}_0|_{L^{2}(Q)}+ |\rho_{2}F_{1}|_{L^{2}(Q)} +|\rho_{2}\overline{F}_{1}|_{L^{2}(Q)}\right.\\
			\hspace*{4cm}  \left.+ |\rho_{2}F_{2}|_{L^{2}(Q)} +|\rho_{2}\overline{F}_{2}|_{L^{2}(Q)}\right)^{1/2}\left(b((\phi,\psi^{1},\psi^{2}),(\phi,\psi^{1},\psi^{2}))\right)^{1/2}\\
			\leq C\left(|y^0_{1}|^{2}_{L^{2}(0,1)}+|y^0_{2}|^{2}_{L^{2}(0,1)} + |\rho_{2}F_0|_{L^{2}(Q)} +|\rho_{2}\overline{F}_0|_{L^{2}(Q)}+ |\rho_{2}F_{1}|_{L^{2}(Q)} +|\rho_{2}\overline{F}_{1}|_{L^{2}(Q)}\right.\\
			\hspace*{4cm}  \left.+ |\rho_{2}F_{2}|_{L^{2}(Q)} +|\rho_{2}\overline{F}_{2}|_{L^{2}(Q)}\right)^{1/2}\|(\phi_1,\psi^{1}_1,\psi^{2}_1;\phi_2,\psi^{1}_2,\psi^{2}_2)\|_{\mathcal{X}}
		\end{array}
	\end{equation}
	for all $(\phi_1,\psi^{1}_1,\psi^{2}_1;\phi_2,\psi^{1}_2,\psi^{2}_2)\in\mathcal{X}$.\\

	 Consequently, in view of Lax-Milgram's lemma, there is a unique $(\hat{\phi}_1,\hat{\psi}^{1}_1,\hat{\psi}^{2}_1;\hat{\phi}_2,\hat{\psi}^{1}_2,\hat{\psi}^{2}_2)\in\mathcal{X}$ satisfying
	 \begin{multline}\label{eq: por Lax-M3} 
	 	b((\hat{\phi}_1,\hat{\psi}^{1}_1,\hat{\psi}^{2}_1;\hat{\phi}_2,\hat{\psi}^{1}_2,\hat{\psi}^{2}_2),(\phi_1,\psi^{1}_1,\psi^{2}_1;\phi_2,\psi^{1}_2,\psi^{2}_2)) = \langle\ell, (\phi_1,\psi^{1}_1,\psi^{2}_1;\phi_2,\psi^{1}_2,\psi^{2}_2)\rangle,\\
	 	 \forall (\phi_1,\psi^{1}_1,\psi^{2}_1;\phi_2,\psi^{1}_2,\psi^{2}_2)\in\mathcal{X}.
	 \end{multline}

\noindent	Let us set
	\begin{equation}\label{definição de y, pi, h3}
		\left\{\begin{array}{lll}
			y_1 = \rho_{0}^{-2}(L_1^{\ast}(\hat{\phi}_1,\hat{\phi}_2)-\alpha_{1}\hat{\psi}^{1}_1 \cara_{\mathcal{O}_{d}}-\alpha_{2}\hat{\psi}^{2}_1\cara_{\mathcal{O}_{d}})&\text{in} & Q,\\
			y_2 = \rho_{0}^{-2}(L_2^{\ast}(\hat{\phi}_1,\hat{\phi}_2)-\alpha_{1}\hat{\psi}^{1}_2 \cara_{\mathcal{O}_{d}}-\alpha_{2}\hat{\psi}^{2}_2\cara_{\mathcal{O}_{d}})&\text{in} & Q,\\
			p^{i}_1 = \rho_{0}^{-2}\left(L_1(\hat{\psi}^{i}_1,\hat{\psi}^{i}_2) + \frac{1}{\mu_{i}}\rho_*^{-2}\hat{\phi}_1\cara_{\mathcal{O}_{i}}\right),\, i=\{1,2\} &\text{in} & Q,\\
			p^{i}_2 = \rho_{0}^{-2}L_2(\hat{\psi}^{i}_1,\hat{\psi}^{i}_2)  &\text{in} & Q,\\
			h = -\rho_{1}^{-2}\hat{\phi}_1\cara_{\mathcal{O}} &\text{in} & Q.
		\end{array}\right.
	\end{equation}
    
	Then, replacing \eqref{definição de y, pi, h3} in \eqref{eq: por Lax-M3} we have
	\begin{multline*}
		\int_{Q}y_1 B_0\,dx\,dt +\int_{Q}y_2 \overline{B}_0\,dx\,dt+ \displaystyle\int_{Q}(p^{1}_1B_{1}+p^1_2 \overline{B}_1 +p^{2}_1B_{2}+p^2_2 \overline{B}_2)\,dx\,dt\\
		=\displaystyle\int_{0}^{1}y^{0}_1\phi_1(0)dx+\int_{0}^{1}y^{0}_2\phi_2(0)dx + \displaystyle\int_{O\times(0,T)} h \phi_1\,dx\,dt\\
		+\int_{Q}\left(F_0 \phi_1+\overline{F}_0 \phi_2 + F_{1}\psi^{1}_1+\overline{F}_1 \psi_2^1 +F_{2}\psi^{2}_1+\overline{F}_2 \psi_2^2\right)dx dt,
	\end{multline*}
	where $(\phi_1,\psi^{1}_1,\psi^{2}_1;\phi_2,\psi^{1}_2,\psi^{2}_2)$ is a solution to the system
	\begin{equation*}
		\left\{\begin{array}{lll}
		L_1^{\ast}({\phi}_1,{\phi}_2)=B_0+\alpha_{1}{\psi}^{1}_1 \cara_{\mathcal{O}_{d}}+\alpha_{2}{\psi}^{2}_1\cara_{\mathcal{O}_{d}}  &\text{in}& Q,  \\
				L_2^{\ast}({\phi}_1,{\phi}_2)=\overline{B}_0+\alpha_{1}{\psi}^{1}_2 \cara_{\mathcal{O}_{d}}+\alpha_{2}{\psi}^{2}_2\cara_{\mathcal{O}_{d}}  &\text{in}& Q,  \\
				L_1({\psi}^{i}_1,{\psi}^{i}_2) = B_i + \frac{1}{\mu_{i}}\rho_*^{-2}{\phi}_1\cara_{\mathcal{O}_{i}},\, i=\{1,2\} &\text{in} & Q,\\
			L_2({\psi}^{i}_1,{\psi}^{i}_2)  = \overline{B}_i  &\text{in} & Q,\\
			\phi_i(0,t)=\phi_i(1,t)=0 & \text{on} & (0,T), \\ \psi_i^j(0,t)=\psi_i^j(1,t)=0, \ \ \ \ \ \ \ \ \ i,j=\{1,2\}& \text{on}& (0,T), \, \\
			\phi_i(\cdot,T)=0,\, \psi_i^{j}(\cdot,0)=0 &\text{in}& \Omega.
		\end{array}\right.
	\end{equation*}
    
    Therefore, $(y_1,p^{1}_1,p^{2}_1;y_2,p^{1}_2,p^{2}_2)$ is a solution by transposition of 
 
 	\begin{equation}\label{optimal123}
 	\begin{cases}
 		L_1({y}_1,{y}_2)=F_0+{h}\cara_{_{{\mathcal{O}}}}-\frac{1}{\mu_1}\rho_*^{-2}p_1^1\cara_{_{{\mathcal{O}}_1}}-\frac{1}{\mu_2}\rho_*^{-2}p_1^2\cara_{_{{\mathcal{O}}_2}}, & \ \ \ \text{in} \ \ \ {Q},\\
 		L_2({y}_1,{y}_2)=\overline{F}_0, & \ \ \ \text{in} \ \ \ {Q},\\
 			L_1^*({p}_1^i,{p}_2^i)=F_i+\alpha_i y_1\cara_{_{{\mathcal{O}}_{id}}}, \ \ \ \ i=\{1,2\} &\ \ \ \text{in} \ \ \ {Q},\\
 			L_2^*({p}_1^i,{p}_2^i)=\overline{F}_i+\alpha_iy_2  \cara_{_{{\mathcal{O}}_{id}}}, &\ \ \ \text{in} \ \ \ {Q},\\
 		y_1(0,t)=y_2(1,t)=y_2(0,t)=y_2(1,t)=0, & \ \ \ \text{on} \ \ \ (0,T),\\
 		p^i_1(0,t)=p^i_1(1,t)=p^i_2(0,t)=p^i_2(1,t)=0, & \ \ \ \text{on} \ \ \ (0,T),\\
 		p^i_1(T)=0, \ p^i_2(T)=0, & \ \ \ \text{in} \ \ \ \Omega,\\
 		y_1(0)=y_1^0(x), \ y_2(0)=y_2^0(x) & \ \ \ \text{in} \ \ \ \Omega.
 	\end{cases}
 \end{equation}	

  Also, as $(\hat{\phi}_1,\hat{\psi}^{1}_1,\hat{\psi}^{2}_1;\hat{\phi}_2,\hat{\psi}^{1}_2,\hat{\psi}^{2}_2)\in\mathcal{X}$ and 
	$F_i, \overline{F}_i\in L^{2}(Q), i=0,1,2$,  
	\color{black}
	using the well-posedness result  applied to a linear system  equation, we obtain
	%Proposition \ref{Regularidade para linear system}, we obtain
	$$y_1,p^{1}_1,p^{2}_1;y_2,p^{1}_2,p^{2}_2\in C^{0}([0,T];L^{2}(0,1))\cap L^{2}(0,T;H^{1}_{a}).$$
	\color{black}
	%\textcolor{red}{Esta faltando una proposition tipo la 3.1 del artículo con Joao y Suerlan. Esa proposicion tiene el label "Regularidade para linear system",por eso sale el ??. Corresponderia al Teorema \ref{teo34noauto} de este pdf.}
	
	\noindent Moreover, from \eqref{eq:def_b3} and \eqref{eq: por Lax-M3}, one has
	\begin{multline*}
\left(b((\hat{\phi}_1,\hat{\psi}^{1}_1,\hat{\psi}^{2}_1;\hat{\phi}_2,\hat{\psi}^{1}_2,\hat{\psi}^{2}_2),(\hat{\phi}_1,\hat{\psi}^{1}_1,\hat{\psi}^{2}_1;\hat{\phi}_2,\hat{\psi}^{1}_2,\hat{\psi}^{2}_2))\right)^{1/2} \\
\leq C \left(|y^0_{1}|^{2}_{L^{2}(0,1)}+|y^0_{2}|^{2}_{L^{2}(0,1)} + |\rho_{2}F_0|^2_{L^{2}(Q)} +|\rho_{2}\overline{F}_0|^2_{L^{2}(Q)}+ |\rho_{2}F_{1}|^2_{L^{2}(Q)} +|\rho_{2}\overline{F}_{1}|^2_{L^{2}(Q)}\right.\\
\hspace*{4cm}  \left.+ |\rho_{2}F_{2}|^2_{L^{2}(Q)} +|\rho_{2}\overline{F}_{2}|^2_{L^{2}(Q)}\right)^{1/2},
	\end{multline*}	
	that is,
	\begin{multline*}
		\int_{Q}\rho_{0}^{2}\left(|y_1|^{2}+|y_2|^{2}\right)\ dxdt       + \displaystyle\sum_{i,j=1}^{2}\displaystyle\int_{Q}\rho_{0}^{2}|p_{i}^j|^{2}\ dxdt +\displaystyle\int_{O\times (0,T)}\rho_{1}^{2}|h|^{2} dx dt\\
		\leq C\left(|y^0_{1}|^{2}_{L^{2}(0,1)}+|y^0_{2}|^{2}_{L^{2}(0,1)} + \sum_{i=0}^{2}|\rho_{2}F_i|^2_{L^{2}(Q)}+|\rho_{2}\overline{F}_i|^2_{L^{2}(Q)}\right),
	\end{multline*}
	which proves \eqref{estimate for solution3}.
	
\end{proof}

%%\newpage
%\noindent 
In order to get the local null controllability of the nonlinear system we will need the following additional estimates.

\begin{propo}
	\label{addicional_estimates_case_linear3}
	Under the hypothesis of Theorem \ref{theorem case linear3}, we have furthermore that the control $h\in L^{2}(\mathcal{O}\times (0,T))$ and the associated states $y_1, p^{1}_1, p^{2}_1, y_2, p^{1}_2, p^{2}_2\in C^{0}([0,T];L^{2}(0,1))\cap L^{2}(0,T;H^{1}_{a})$, solution of \eqref{eq:linearized_system_y} and \eqref{eq:linearized_system_p}, satisfy the additional estimates 
	\begin{multline}\label{des Proposition 53}
		\sup_{[0,T]}(\hat{\rho}^{2}\|y_1\|^{2}_{L^{2}(0,1)})+\sup_{[0,T]}(\hat{\rho}^{2}\|y_2\|^{2}_{L^{2}(0,1)}) + \sum_{i,j=1}^{2}\sup_{[0,T]}(\hat{\rho}^{2}\|p^{j}_i\|^{2}_{L^{2}(0,1)}) \\
		+\displaystyle\int_{Q}\hat{\rho}^{2} a(x)\left(|y_{1x}|^{2} +|y_{2x}|^{2}+ \sum_{i,j=1}^{2}|p^{j}_{ix}|^{2}\right)dxdt\ \leq C \kappa_{0}(F_i,\overline{F}_i,y^0_1,y_2^0).
	\end{multline}
%	and, if $y_{0}\in H^{1}_{a}(0,1)$ 
%	\begin{multline}\label{des Proposition 6}
%		\displaystyle\sup_{[0,T]}(\rho_{1}^{2}\|\sqrt{a}y_{x} \|^{2}_{L^{2}(0,1)})+ \displaystyle\sup_{[0,T]}(\rho_{1}^{2}\|\sqrt{a}p^{1}_{x} \|^{2}_{L^{2}(0,1)}) + \displaystyle\sup_{[0,T]}(\rho_{1}^{2}\|\sqrt{a}p^{2}_{x} \|^{2}_{L^{2}(0,1)})\\
%		+ \displaystyle\int_{Q}\rho_{1}^{2}(|y_{t}|^{2}+|p^{1}_{t}|^{2}+|p^{2}_{t}|^{2}+|(a(x)y_{x})_{x}|^{2} + |(a(x)p^{1}_{x})_{x}|^{2}+ |(a(x)p^{2}_{x})_{x}|^{2})dxdt\\
%		\leq C \kappa_{1}(H,H_{1},H_{2},y_{0}),  
%	\end{multline}
%	where $\kappa_{1}(H,H_{1},H_{1},y_{0})= |\rho_2 H|^2_{L^2(Q)} + |\rho_2 H_1|^2_{L^2(Q)} + |\rho_2 H_2|^2_{L^2(Q)} + \|y_0\|^2_{H^{1}_{a}}$. 
\end{propo}
\begin{proof}
	%Let us denote by $L\varphi =  \varphi_{t} - (a(x)\varphi_{x})_{x} + c(x,t)\phi$ and $L^{\ast}\varphi=-\varphi_{t} - (a(x)\varphi_{x})_{x} + c(x,t)\phi$.
	%(\textcolor{green}{Temos que ver sobre a função regular que substitui a função caracteristica. Se vamos defini-la no inicio ou somente aqui, como no trabalho do Danny}). 
	We proceed following the steps of \cite{DemarqueLimacoViana2018}. %Therefore, to establish estimates like \eqref{des Proposition 5},  we need to 
    First we change variables in the form $\tilde{p}^{j}_i(x,t)=p^{j}_i(x,T-t)$ with $i,j=1,2$, and $ \tilde{F}_{i}(x,t)=F_{i}(x,T-t)$, with $i=\lbrace 0,1,2\rbrace$, so that we obtain a system in new variables $(\tilde {y}_1,\tilde {y}_2,\tilde{p}^{j}_i,\tilde{F}_{i})$ in which all the equations are forward in time. The new system has initial conditions $(\tilde{y}(\cdot,0),\tilde{p}^{1}(\cdot,0),\tilde{p}^{2}(\cdot, 0 ))=(y^{0},0,0)$ and the weighted estimates for the solution and the regularity results will be equivalent to those of the original system \eqref{eq:linearized_system_y} and \eqref{eq:linearized_system_p}. Keeping this in mind, for simplicity, we will maintain the notations. Thus, we consider the linear system, for $i,j=1,2$,
    \begin{equation}\label{eq:linearized_system new3}
		\begin{cases}
			y_{1t}-b(t)\left({a}(x)y_{1x}\right)_x-B\sqrt{a}y_{1x} + c_{11}y_1 + c_{12}y_2 = F_0+{h}\cara_{_{{\mathcal{O}}}}-\frac{1}{\mu_1}\rho_*^{-2}p_1^1\cara_{_{{\mathcal{O}}_1}}-\frac{1}{\mu_2}\rho_*^{-2}p_1^2\cara_{_{{\mathcal{O}}_2}}, & \ \ \ \text{in} \ \ \ {Q},\\
			y_{2t}-b(t)\left({a}(x)y_{2x}\right)_x-B\sqrt{a}y_{2x}+c_{21}y_1+c_{22}y_2=\overline{F}_0, & \ \ \ \text{in} \ \ \ {Q},\\
			p^i_{1t}-b(t)\left(a(x)p^i_{1x}\right)_x+\left(B\sqrt{a}p^i_1\right)_x+c_{11}p^i_1+c_{21}p^i_2=F_i+\alpha_i y_1\cara_{_{{\mathcal{O}}_{id}}}, \ \ \ \ i=\{1,2\} &\ \ \ \text{in} \ \ \ {Q},\\
			p^i_{2t}-b(t)\left(a(x)p^i_{2x}\right)_x+\left(B\sqrt{a}p^i_2\right)_x+c_{12}p^i_1 + c_{22}p^i_2=\overline{F}_i+\alpha_iy_2  \cara_{_{{\mathcal{O}}_{id}}}, &\ \ \ \text{in} \ \ \ {Q},\\
			y_1(0,t)=y_2(1,t)=y_2(0,t)=y_2(1,t)=0, & \ \ \ \text{on} \ \ \ (0,T),\\
			p^i_1(0,t)=p^i_1(1,t)=p^i_2(0,t)=p^i_2(1,t)=0, & \ \ \ \text{on} \ \ \ (0,T),\\
			p^i_1(T)=0, \ p^i_2(T)=0, & \ \ \ \text{in} \ \ \ \Omega,\\
			y_1(0)=y_1^0(x), \ y_2(0)=y_2^0(x) & \ \ \ \text{in} \ \ \ \Omega.
		\end{cases}
	\end{equation}	
	
	Let us multiply $\eqref{eq:linearized_system new3}_{1}$ by $\hat{\rho}^{2} y_1$, $\eqref{eq:linearized_system new3}_{2}$ by $\hat{\rho}^{2} y_2$, $\eqref{eq:linearized_system new3}_{3}$ by $\hat{\rho}^{2} p^{i}_1$, $\eqref{eq:linearized_system new3}_{4}$ by $\hat{\rho}^{2} p^{i}_2$ and integrate over $[0,1]$. Hence, using that $\hat{\rho}^{2} = \rho_{0}\rho_{1}$, and $\rho_{1}\leq C\rho_{2}$, we compute
	{
		\begin{equation}\label{second estimate3}
			\begin{array}{l}
				\dfrac{1}{2}\dfrac{d}{dt}\displaystyle\int_{0}^{1}\hat{\rho}^{2}(|y_1|^{2}+|y_2|^{2})dx +\sum_{i,j=1}^{2}\dfrac{1}{2}\dfrac{d}{dt}\displaystyle\int_{0}^{1}\hat{\rho}^{2} | p^{j}_i|^{2}dx + b(t) \displaystyle\int_{0}^{1}\hat{\rho}^{2}a(x)(|y_{1x}|^{2}+|y_{2x}|^{2}) dx+ \\
				b(t) \displaystyle \sum_{i,j=1}^{2}\displaystyle\int_{0}^{1}\hat{\rho}^{2}a(x)|p_{ix}^{j}|^{2}dx\hspace{0.1cm}+ \displaystyle\int_0^1 \hat{\rho}^2 d_1(x,t) \sqrt{a(x)}(y_{1x} y_1+y_{2x} y_2) dx \\
				+\displaystyle\sum_{i,j=1}^{2} \int_0^1 \hat{\rho}^2 d_2(x,t) \sqrt{a(x)}p^j_{ix} p^j_{i} dx  \\
				\leq C\left(\displaystyle\int_{0}^{1}\rho_{0}^{2}(|y_1|^{2}+|y_2|^{2})dx + \sum_{i,j=1}^{2}\displaystyle\int_{0}^{1}\rho_{0}^{2}|p^{i}_j|^{2}dx + \displaystyle\int_{O}\rho_{1}^{2}|h|^{2}dx +
				\sum_{i=0}^{2}\displaystyle\int_{0}^{1}\rho_{2}^{2}(|F_{i}|^{2}+|\overline{F}_{i}|^{2})dx\right)\hspace{0.1cm}\\
				+ \ \mathcal{M},
				%+ |\mathcal{N}|,
			\end{array}
	\end{equation}}
	where $\mathcal{M} = \displaystyle\int_{0}^{1}\hat{\rho}(\hat{\rho})_{t}(|y_1|^{2}+|y_2|^{2})dx +\sum_{i,j=1}^{2}\displaystyle\int_{0}^{1}\hat{\rho}(\hat{\rho})_{t}|p^{j}_i|^{2}dx,$
	and $(\cdot)_{t}=\frac{d}{dt}(\cdot)$. To simplify notation, in the following we will omit the summation sign. Recall that
	$A^*(t) = C_1 \tau(t)$, and $\zeta^*(t) = C_2 \tau(t)$, then we have that $A^*_{t}=\bar C (\zeta^*)_{t}$ and consequently
	\begin{equation*}
		\begin{split}
			\hat{\rho}(\hat{\rho})_{t} &= e^{-sA^*}(\zeta^*)^{-3}\left(-se^{-sA^*} A^*_{t}(\zeta^*)^{-3} -3 e^{-sA^*}(\zeta^*)^{-4}(\zeta^*)_{t}\right)    \\
			&= -e^{-2sA^*}(\zeta^*)^{-4}(\zeta^*)_{t}\left(s(\zeta^*)^{-2}\bar{C} + 3(\zeta^*)^{-3} \right)  \\
			&=-\rho_{0}^{2}(\zeta^*)_{t}\left(s(\zeta^*)^{-2}\bar{C} + 3(\zeta^*)^{-3} \right).
		\end{split}
	\end{equation*}
	Thus, for any $t\in [0,T)$,
	\begin{equation*}
		\begin{array}{l}
			|\hat{\rho}(\hat{\rho})_{t}|\leq C\rho_{0}^{2}\tau^{2}|s(\zeta^*)^{-2}\bar{C} + 3(\zeta^*)^{-3}|  
			\leq C\rho_{0}^{2}|s\bar{C}+3(\zeta^*)^{-1}| \leq C\rho_{0}^{2},
		\end{array}
	\end{equation*}
	and then, we obtain
	\begin{equation*}
		%\label{estimate of M}
		\mathcal{M}\leq C \displaystyle\int_{0}^{1}{\rho_{0}^{2}}(|y_1|^{2}+|y_2|^{2} + \sum_{i,j=1}^{2} |p^{j}_i|^{2})dx.
	\end{equation*}
	
	\noindent Therefore, using Young's inequality, \eqref{second estimate3} becomes, for a small $\epsilon>0$. 
    %( now we will only write the subscript and superscript notation that will indicate summation where necessary).
	\begin{multline*}
		\dfrac{1}{2}\dfrac{d}{dt}\displaystyle\int_{0}^{1}\hat{\rho}^{2}(|y_1|^{2}+|y_2|^{2}+| p^{j}_i|^{2})dx + b(t) \displaystyle\int_{0}^{1}\hat{\rho}^{2} a(x)(|y_{1x}|^{2}+|y_{2x}|^{2}+|p^{j}_{ix}|^{2})dx\\
		\leq C\left(\displaystyle\int_{0}^{1}\rho_{0}^{2}(|y_1|^{2} +|y_2|^{2}+ |p^{j}_i|^{2})dx + \displaystyle\int_{O}\rho_{1}^{2}|h|^{2}dx + \displaystyle\int_{0}^{1}\rho_{2}^{2}(|F_{i}|^{2}+|\overline{F}_{i}|^{2})dx\right)\\
		+ \epsilon \displaystyle\int_{0}^{1}\hat{\rho}^{2} a(x)(|y_{1x}|^{2}+|y_{2x}|^{2}+|p^{j}_{ix}|^{2})dx + C_\epsilon \left( \int_0^1 |y_1|^{2} +|y_2|^{2}+ |p^{j}_i|^{2} \right).
	\end{multline*}
	Thus, since $b(t)$ is bounded and $\rho_0$ is bounded by below, there is constant $D>0$ such that
	\begin{multline*}
		\dfrac{1}{2}\dfrac{d}{dt}\displaystyle\int_{0}^{1}\hat{\rho}^{2}(|y_1|^{2}+|y_2|^{2}+| p^{j}_i|^{2})dx + \displaystyle\int_{0}^{1}\hat{\rho}^{2} a(x)(|y_{1x}|^{2}+|y_{2x}|^{2}+|p^{j}_{ix})dx\\
		\leq D\left(\displaystyle\int_{0}^{1}\rho_{0}^{2}(|y_1|^{2} +|y_2|^{2}+ |p^{j}_i|^{2})dx + \displaystyle\int_{O}\rho_{1}^{2}|h|^{2}dx + \displaystyle\int_{0}^{1}\rho_{2}^{2}(|F_{i}|^{2}+|\overline{F}_{i}|^{2})dx\right),
	\end{multline*}
	and, integrating in time,  we conclude the desired result by the Linear Theorem.\\

\end{proof}

\section{Local null controllability of the nonlinear system }\label{control for nonlinear system}

%All along this section we use the weights defined in \eqref{eq:weights_rhos} and 
In this section we use Liusternik's (right) inverse function theorem (see \cite{Alekseev}) to obtain our local controllability results. Here $B_{r}(x_0)$ 
%and $B_{\delta}(\zeta_{0})$ are 
denote the open ball of radius $r$ centered at $x_0$.
%and $\zeta_0$, respectively.
\begin{teo}[Liusternik]\label{Liusternik}
Let $ \mathcal{Y}$ and $ \mathcal{Z}$ be Banach spaces and let $\mathcal{A}:B_{r}(0)\subset  \mathcal{Y}\rightarrow  \mathcal{Z}$ be a $\mathcal{C}^{1}$ mapping. Let as assume that $\mathcal{A}^{\prime}(0)$ is onto and let us set $\mathcal{A}(0)=\zeta_{0}$. Then, there exist $\delta >0$, a mapping $W: B_{\delta}(\zeta_{0})\subset  \mathcal{Z}\rightarrow  \mathcal{Y}$ and a constant $K>0$ such that
\begin{equation*}
    W(z)\in B_{r}(0),\,\, \mathcal{A}(W(z))=z\,\, \text{and}\,\, \Vert W(z)\Vert_{ \mathcal{Y}}\leq K\Vert z-\zeta_{0}\Vert_{ \mathcal{Z}}\, \, \forall\, z\in B_{\delta}(\zeta_{0}).
\end{equation*}
In particular, $W$ is a local inverse-to-the-right of $\mathcal{A}$.
\end{teo}
 
%In order to do so, 
For our problem, we define a map $\mathcal{A} : \mathcal{Y} \to \mathcal{Z}$ between suitable Banach spaces $\mathcal{Y}$ and $\mathcal{Z}$ whose definition and properties came from the controllability result of the linearized system and the additional estimates shown in Theorem \ref{theorem case linear3} and Proposition \ref{addicional_estimates_case_linear3}  in order to verify that the map $\mathcal{A}$ is well defined and verifies Liusternik's theorem hypothesis.

From the linearized system \eqref{eq:linearized_system_y} and \eqref{eq:linearized_system_p}, we denote, for $i=1,2$,
%
%All along this section we use the weights defined in \eqref{eq:weights_rhos}.
%
%We are interested in using Liusternik's inverse function theorem (on the right) to obtain our result, which can be found in \cite{Alekseev}, and is given below, where $B_{r}(0)$ and $B_{\delta}(\zeta_{0})$ are open ball, respectively of radius $r$ and $\delta$.
%\begin{teo}[{\bf Liusternik}]\label{Liusternik}
%	Let $ \mathcal{Y}$ and $ \mathcal{Z}$ be Banach spaces and let $\mathcal{A}:B_{r}(0)\subset  \mathcal{Y}\rightarrow  \mathcal{Z}$ be a $\mathcal{C}^{1}$ mapping. Let as assume that $\mathcal{A}^{\prime}(0)$ is onto and let us set $\mathcal{A}(0)=\zeta_{0}$. Then, there exist $\delta >0$, a mapping $W: B_{\delta}(\zeta_{0})\subset  \mathcal{Z}\rightarrow  \mathcal{Y}$ and a constant $K>0$ such that
%	\begin{equation*}
%		W(z)\in B_{r}(0),\,\, \mathcal{A}(W(z))=z\,\, \text{and}\,\, \Vert W(z)\Vert_{ \mathcal{Y}}\leq K\Vert z-\zeta_{0}\Vert_{ \mathcal{Z}}\, \, \forall\, z\in B_{\delta}(\zeta_{0}).
%	\end{equation*}
%	In particular, $W$ is a local inverse-to-the-right of $\mathcal{A}$.
%\end{teo}
%
%
%
%
%In order to do so, we define a map $\mathcal{A} : \mathcal{Y} \to \mathcal{Z}$ between suitable Banach spaces $\mathcal{Y}$ and $\mathcal{Z}$ whose definition and properties came from the controllability result of the linearized system and the additional estimates shown in Theorem \ref{theorem case linear}  in order to verify that the map $\mathcal{A}$ is well defined and verifies Liusternik's theorem hypothesis.
\begin{eqnarray*}
F_0&=&y_{1t}-b(t)\left({a}(x)y_{1x}\right)_x-B\sqrt{a}y_{1x}+c_{11}y_1 + c_{12}y_2 -{h}\cara_{_{{\mathcal{O}}}}+\frac{1}{\mu_1}\rho_*^{-2}p_1^1\cara_{_{{\mathcal{O}}_1}}+\frac{1}{\mu_2}\rho_*^{-2}p_1^2\cara_{_{{\mathcal{O}}_2}},\\
\overline{F}_0&=&y_{2t}-b(t)\left({a}(x)y_{2x}\right)_x-B\sqrt{a}y_{2x}+c_{21}y_1+c_{22}y_2,\\
F_i&=&p^i_{1t}-b(t)\left(a(x)p^i_{1x}\right)_x+\left(B\sqrt{a}p^i_1\right)_x+c_{11}p^i_1+c_{21}p^i_2-\alpha_i y_1\cara_{_{{\mathcal{O}}_{id}}},\\
\overline{F}_i&=&p^i_{2t}-b(t)\left(a(x)p^i_{2x}\right)_x+\left(B\sqrt{a}p^i_2\right)_x+c_{12}p^i_1 + c_{22}p^i_2-\alpha_iy_2  \cara_{_{{\mathcal{O}}_{id}}}, % \ \ \ \quad i=1,2.
\end{eqnarray*} 
Now, let us define the space
\begin{equation}
	\label{eq:espaceY}
	\begin{array}{c}
		\mathcal{Y} = \{  (y_1,p^1_1,p^2_1,y_2,p^1_2,p^2_2,h)\in [L^{2}(\Omega\times (0,T))]^{6}\times L^{2}( \mathcal{O}\times (0,T)) \ : \ y_i(\cdot,t), p^{j}_i(\cdot,t), i,j=1,2,\\ \text{are absolutely continuous in}\ [0, 1],\ \text{a.e. in}\ [0, T], \ \rho_{1}h\in L^{2}( \mathcal{O} \times (0,T)), \\
		% \text{for}\ H=y_t - (l(0) a(x) y_x)_x - h 1_{O} + \frac{1}{\mu_1} p^1 1_{O_1} + \frac{1}{\mu_2} p^2 1_{O_2}\ \text{and}\ \\H_{i}= -p_t^i - \left(l(0) a(x) p_x^i \right)_x  - \alpha_i y 1_{O_{i,d}}, \ i=1,2, 
		\rho_{0}y_i, \rho_{0}p^{j}_i, \rho_2 F_0, \rho_2 \overline{F}_0 \in L^2(Q),
		\rho_2 F_i, \rho_2 \overline{F}_i \in L^2(Q), \ i=1,2 \\
		y_i(1, t) \equiv p^{j}_i(1,t) \equiv  y_i(0,t) \equiv p^{j}_i(0,t) \equiv 0 \ \text{a.e in}\ [0, T], \ \ i,j=1,2,\\
		 y_1(\cdot,0), y_2(\cdot,0) \in H_a^1(\Omega) \}.
	\end{array}
\end{equation}

Then, $\mathcal{Y}$ is a Hilbert space with the norm
%$\Vert .\Vert_{\mathcal{Y}}$, where
\begin{multline*}
\Vert (y_1,p^1_1,p^2_1,y_2,p^1_2,p^2_2,h)\Vert^{2}_{\mathcal{Y}} = \Vert \rho_{0}y_1\Vert^{2}_{L^{2}(Q)} +\Vert \rho_{0}y_2\Vert^{2}_{L^{2}(Q)}+ \sum_{i,j=1}^{2}\Vert \rho_{0}p^{j}_i\Vert^{2}_{L^{2}(Q)}+ \Vert\rho_{1}h\Vert^{2}_{L^{2}(\mathcal{O}\times (0,T))}\\
+\sum_{i=0}^{2}\Vert\rho_{2} F_i\Vert^{2}_{L^{2}(Q)}+\Vert\rho_{2} \overline{F}_i\Vert^{2}_{L^{2}(Q)}+ \Vert y_1(\cdot,0)\Vert^{2}_{H^{1}_{a}(\Omega)}+ \Vert y_2(\cdot,0)\Vert^{2}_{H^{1}_{a}(\Omega)}.
\end{multline*}

Due to Proposition \ref{addicional_estimates_case_linear3}, for any $(y_1,p^1_1,p^2_1,y_2,p^1_2,p^2_2,h)\in \mathcal{Y}$ we get: 
\begin{multline}\label{adicionalnorma}
	\sup_{[0,T]}(\hat{\rho}^{2}\|y_1\|^{2}_{L^{2}(0,1)})+\sup_{[0,T]}(\hat{\rho}^{2}\|y_2\|^{2}_{L^{2}(0,1)}) + \sum_{i,j=1}^{2}\sup_{[0,T]}(\hat{\rho}^{2}\|p^{j}_i\|^{2}_{L^{2}(0,1)}) \\
	+\displaystyle\int_{Q}\hat{\rho}^{2} a(x)\left(|y_{1x}|^{2} +|y_{2x}|^{2}+ \sum_{i,j=1}^{2}|p^{j}_{ix}|^{2}\right)dxdt\ \leq C \Vert (y_1,p^1_1,p^2_1,y_2,p^1_2,p^2_2,h)\Vert^{2}_{\mathcal{Y}}. 
\end{multline}

Now, let us introduce the Banach space $\mathcal{Z} = \mathcal{F} \times \mathcal{F} \times \mathcal{F} \times\mathcal{F} \times \mathcal{F} \times \mathcal{F}\times H_a^1(\Omega)\times H_a^1(\Omega)$ such that $\mathcal{F}=\{ z \in L^2(Q) \ : \ \rho_2 z \in L^2(Q) \}$.

 Finally, consider the map $\mathcal{A} : \mathcal{Y} \to \mathcal{Z}$ such that 
$(y_1,p^1_1,p^2_1,y_2,p^1_2,p^2_2,h) \mapsto \mathcal{A}(y_1,p^1_1,p^2_1,y_2,p^1_2,p^2_2,h)$ and 
$\mathcal{A}(y_1,p^1_1,p^2_1,y_2,p^1_2,p^2_2,h) = (\mathcal{A}^1_0,\mathcal{A}^2_0, \mathcal{A}^1_1,\mathcal{A}^2_1, \mathcal{A}^1_2,\mathcal{A}^2_2, \mathcal{A}_3^1, \mathcal{A}^2_3)$. For simplicity we use the following notation  $(y_1,p^1_1,p^2_1,y_2,p^1_2,p^2_2,h) \equiv (y_1,y_2,p_i^j,h)$, where the components are given by
\begin{equation}\label{aplicação A3}
\hspace*{-0.8cm}	\left\{\begin{array}{l}
		\mathcal{A}^1_0(y_1,y_2,p_i^j,h) = y_{1t}-b(t)\left({a}(x)y_{1x}\right)_x-B\sqrt{a}y_{1x}+F_1(y_1,y_2)
		-{h}\cara_{_{{\mathcal{O}}}}+\frac{1}{\mu_1}\rho_*^{-2}p_1^1\cara_{_{{\mathcal{O}}_1}}+\frac{1}{\mu_2}\rho_*^{-2}p_1^2\cara_{_{{\mathcal{O}}_2}},\vspace{0.1cm}\\
		\mathcal{A}^2_0(y_1,y_2,p_i^j,h)  =	y_{2t}-b(t)\left({a}(x)y_{2x}\right)_x-B\sqrt{a}y_{2x}+F_2(y_1,y_2),\\
		\mathcal{A}^1_i(y_1,y_2,p_i^j,h)  = -p^i_{1t}-b(t)\left(a(x)p^i_{1x}\right)_x+\left(B\sqrt{a}p^i_1\right)_x+ D_1 F_1(y_1,y_2)p^i_1 + D_1 F_2(y_1,y_2) p^i_2 \\
        \qquad\qquad\qquad\qquad - \alpha_i\left( y_1-y_{1d}^i  \right)\cara_{_{{\mathcal{O}}_{id}}}, \quad i=1, 2, \\
		\mathcal{A}^2_i(y_1,y_2,p_i^j,h)  = -p^i_{2t}-b(t)\left(a(x)p^i_{2x}\right)_x+\left(B\sqrt{a}p^i_2\right)_x + D_2 F_1(y_1,y_2)p^i_1 + D_2 F_2(y_1,y_2)p^i_2 \\
        \qquad\qquad\qquad\qquad -\alpha_i\left( y_2-y_{2d}^i  \right)\cara_{_{{\mathcal{O}}_{id}}}, \quad i=1, 2,\\
		\mathcal{A}^1_3(y_1,y_2,p_i^j,h)  = y_1(.,0),\\
		\mathcal{A}^2_3(y_1,y_2,p_i^j,h)  = y_2(.,0).
	\end{array}\right.
\end{equation}

We prove that we can apply the Theorem \ref{Liusternik} to the mapping $\mathcal{A}$ in \eqref{aplicação A3}, through the following three lemmas:
\begin{lema}\label{A bem definido3}
	Let $\mathcal{A}: \mathcal{Y}\rightarrow  \mathcal{Z}$ be given by \eqref{aplicação A3}. Then, $\mathcal{A}$ is well defined and continuous. 
\end{lema}
\begin{proof}
	We want to show that $\mathcal{A}(y_1,y_2,p_i^j,h)$ belongs to $ \mathcal{Z}$, for every $(y_1,y_2,p_i^j,h)\in  \mathcal{Y}$. Therefore, we show that each $\mathcal{A}^j_{k}(y_1,y_2,p_i^j,h)$, with $k=\lbrace 0,1, 2, 3\rbrace$, $j=1,2$, defined in \eqref{aplicação A3}, belongs to its respective space. First, for $\mathcal{A}^1_{0}$, we have 
	\begin{equation*}
		\begin{array}{lll}
			\|\mathcal{A}^1_{0}(y_1,y_2,p_i^j,h)\|^{2}_{\mathcal{F}}&\leq & \displaystyle\int_{Q}\rho^{2}_{2}|F_0|^{2}dxdt 
			+ \displaystyle\int_{Q}\rho^{2}_{2}\left|F_1(y_1,y_2) -c_{11}y_1 - c_{12}y_2 \right|^{2}dxdt
			=  I_{1} +  I_{2}.
		\end{array}
	\end{equation*}
	
	It is immediate from the definition of the space $\mathcal{Y}$ that $I_{1}\leq C \|(y_1,y_2,p_i^j,h)\|^{2}_{\mathcal{Y}}$.\\

Furthermore, since $F$ is in the class $C^2$ with bounded derivatives, using the Mean value theorem, we obtain
\begin{equation*}
	\begin{array}{l}
		I_{2}=\displaystyle\int_{Q}\rho^{2}_{2}\left| F_1(y_1,y_2) -F_1(0,0) -c_{11}y_1 - c_{12}y_2 \right|^{2}dxdt  \\
		\quad \leq C\displaystyle\int_{Q}\rho^{2}_{2}\left|F_1(y_1,y_2) -F_1(0,0) \right|^{2}dxdt+\displaystyle\int_{Q}\rho^{2}_{2} \left( \left|c_{11}y_1 \right|^{2} + \left|c_{12}y_2 \right|^{2} \right)dxdt\\
		\quad \leq C\displaystyle\int_{Q}\rho^{2}_{2}\left|DF_1(\theta(y_1,y_2))\right|^{2} \left( |y_1|^2 + |y_2|^2 \right)dxdt + \|c_{11}\|_{L^\infty(Q)}\displaystyle\int_{Q}\rho^{2}_{2}\left|y_1 \right|^{2}dxdt \\
        \qquad+ \|c_{12}\|_{L^\infty(Q)}\displaystyle\int_{Q}\rho^{2}_{2}\left|y_2 \right|^{2}dxdt\\
		\quad \leq C\displaystyle\int_{Q}\rho^{2}_{2} \left( \left|y_1 \right|^{2} + \left|y_2 \right|^{2} \right)dxdt\\
		\quad \leq C\|(y_1,y_2,p_i^j,h)\|_{\mathcal{Y}}.
	\end{array}
\end{equation*}
Thus, $\mathcal{A}^1_{0}(y_1,y_2,p_i^j,h)\in \mathcal{F}.$\\

For $\mathcal{A}^2_{0}(y_1,y_2,p_i^j,h)$, we have
	\begin{equation*}
	\begin{array}{lll}
		\|\mathcal{A}^2_{0}(y_1,y_2,p_i^j,h)\|^{2}_{\mathcal{F}}&\leq & \displaystyle\int_{Q}\rho^{2}_{2}|\overline{F}_0|^{2}dxdt 
		+ \displaystyle\int_{Q}\rho^{2}_{2}\left|F_2(y_1,y_2) -c_{21}y_1 - c_{22}y_2 \right|^{2}dxdt
		= I_{3} +  I_{4}.
	\end{array}
\end{equation*}

It is also immediate
%through the  definition of the space $\mathcal{Y}$
that $I_{3}\leq C \|(y_1,y_2,p_i^j,h)\|^{2}_{\mathcal{Y}}$. 
Similarly to what was done for $I_2$, we get 
$$I_4 \leq C\displaystyle\int_{Q}\rho^{2}_{2} \left(\left|y_1 \right|^{2} + \left|y_2 \right|^{2} \right)dxdt\leq C\|(y_1,y_2,p_i^j,h)\|_{\mathcal{Y}}.$$

\begin{comment}
    
Furthermore, since $F_2$ is in the class $C^2$, we obtain
\begin{equation*}
	\begin{array}{l}
		I_{4}=\displaystyle\int_{Q}\rho^{2}_{2}\left|F_2(y_2) -F_2(0) -c_{22}y_2 \right|^{2}dxdt  \\
		\quad \leq C\displaystyle\int_{Q}\rho^{2}_{2}\left|F_2(y_2) -F_2(0) \right|^{2}dxdt+\displaystyle\int_{Q}\rho^{2}_{2}\left|-c_{22}y_2 \right|^{2}dxdt\\
		\quad \leq \ \ \ \ \ C\displaystyle\int_{Q}\rho^{2}_{2}\left|F_2'(y_2)\right|^{2} |y_2|^2dxdt+\|c_{22}\|_{L^\infty(Q)}\displaystyle\int_{Q}\rho^{2}_{2}\left|y_2 \right|^{2}dxdt\\
		\quad \leq \ \ C\displaystyle\int_{Q}\rho^{2}_{2}\left|y_2 \right|^{2}dxdt\\
		\quad \leq C\|(y_1,y_2,p_i^j,h)\|_{\mathcal{Y}}.
	\end{array}
\end{equation*}

\end{comment}

\noindent Now, let us analyze $\mathcal{A}^1_{i}(y_1,y_2,p_i^j,h)$, with $i=1,2$. As before, we 
\begin{equation*}
	\begin{split}
		\Vert\mathcal{A}^1_{i}(y_1,y_2,p_i^j,h)\Vert^{2}_{\mathcal{F}}
		&\leq  \displaystyle\int_{Q}\rho^{2}_{2}|F_i|^{2}dxdt 
        %&\quad & 
        + \displaystyle\int_{Q}\rho^{2}_{2}\left| D_1 F_1(y_1,y_2)p^i_1 + D_1 F_2(y_1,y_2) p^i_2 - c_{11}p^i_1 - c_{12}p^i_2 \right|^{2}dxdt \\
		%&\quad & + \ \displaystyle\int_{Q}\rho^{2}_{2}\left| F_1'(y_1)p^i_1 - c_{11}p^i_1 \right|^{2}dxdt \\
		&\quad  + \
		\displaystyle\int_{Q}\rho^{2}_{2}\left| \alpha_i y^i_{1,d} \cara_{{{\mathcal{O}}_{1d}}}\right|^{2}dxdt \\
		& =  I_{5} + I_{6}+I_7.
	\end{split}
\end{equation*}

\noindent From the definition of the space $\mathcal{Y}$ we have that $I_{5}\leq C \|(y_1,y_2,p_i^j,h)\|^{2}_{\mathcal{Y}}$ and, from the hypothesis $\rho^2 y^i_{1,d} \in L^2(Q)$, we have that $I_7$ is also bounded. Moreover, as above,
\begin{equation*}
    \begin{split}
        I_{6} &\leq C\displaystyle\int_{Q}\rho^{2}_{2} |p^i_1|^2dxdt + \|c_{11}\|_{L^\infty(Q)}\displaystyle\int_{Q}\rho^{2}_{2} \left( |p^i_1|^{2} + |p^i_2|^{2} \right)dxdt + \|c_{12}\|_{L^\infty(Q)}\displaystyle\int_{Q}\rho^{2}_{2} |p^i_2|^{2}dxdt \\
        &\leq C\|(y_1,y_2,p_i^j,h)\|^2_{\mathcal{Y}}.      
    \end{split}
\end{equation*}

\begin{comment}

Since $F_1$ is in the class $C^2$,  we obtain 
\begin{equation*}
	\begin{array}{l}
		I_{6}=\displaystyle\int_{Q}\rho^{2}_{2}\left|F_1'(y_1)p^i_1 -c_{11}p^i_1 \right|^{2}dxdt  \\
		\quad \leq C\displaystyle\int_{Q}\rho^{2}_{2}\left|F_1'(y_1) \right|^{2}\left|p^i \right|^{2}dxdt+\displaystyle\int_{Q}\rho^{2}_{2}\left|-c_{11}p^i_1 \right|^{2}dxdt\\
		\quad \leq C\displaystyle\int_{Q}\rho^{2}_{2} |p^i_1|^2dxdt+\|c_{11}\|_{L^\infty(Q)}\displaystyle\int_{Q}\rho^{2}_{2}\left|p^i_1 \right|^{2}dxdt\\
		\quad \leq  C\displaystyle\int_{Q}\rho^{2}_{2}\left|p^i_1 \right|^{2}dxdt\\
		\quad \leq C\|(y_1,y_2,p_i^j,h)\|^2_{\mathcal{Y}}.
	\end{array}
\end{equation*}
From the  hypothesis $\rho^2 y^i_{1,d} \in L^2(Q)$, we have that $I_7$ is bounded.
    
\end{comment}

Now, let us analyze $\mathcal{A}^2_{i}(y_1,y_2,p_i^j,h)$, with $i=\lbrace 1,2\rbrace$. We have
\begin{equation*}
	\begin{split}
		\Vert\mathcal{A}^2_{i}(y_1,y_2,p_i^j,h)\Vert^{2}_{\mathcal{F}}
		&\leq  \displaystyle\int_{Q}\rho^{2}_{2}|\overline{F}_i|^{2}dxdt + \displaystyle\int_{Q}\rho^{2}_{2}\left| D_2 F_1(y_1,y_2)p^i_1 + D_2 F_2(y_1,y_2)p^i_2 - c_{21}p^i_1 - c_{22}p^i_2 \right|^{2}dxdt 
		\\
        &\quad
        %+ \displaystyle\int_{Q}\rho^{2}_{2}\left| F_2'(y_2)p^i_2 - c_{22}p^i_2 \right|^{2}dxdt
		+ \displaystyle\int_{Q}\rho^{2}_{2}\left| \alpha_i y^i_{2,d} \cara_{{{\mathcal{O}}_{2d}}}\right|^{2}dxdt \\
		&= I_{8} + I_{9}+I_{10}.
	\end{split}
\end{equation*}
Analogous estimates give that
$I_k \leq C \|(y_1,y_2,p_i^j,h)\|^{2}_{\mathcal{Y}}$, for $k=8,9,10$.
Thus $\mathcal{A}$ is well-defined.

\begin{comment}
 From from definition of the space $\mathcal{Y}$, we have that $I_{8}\leq C \|(y_1,y_2,p_i^j,h)\|^{2}_{\mathcal{Y}}$. \\
Since $F_2$ is in the class $C^2$  we obtain 
\begin{equation*}
	\begin{array}{l}
		I_{9}=\displaystyle\int_{Q}\rho^{2}_{2}\left|F_2'(y_2)p^i_2 -c_{22}p^i_2 \right|^{2}dxdt  \\
		\quad \leq C\displaystyle\int_{Q}\rho^{2}_{2}\left|F_2'(y_2) \right|^{2}\left|p^i_2 \right|^{2}dxdt+\displaystyle\int_{Q}\rho^{2}_{2}\left|-c_{22}p^i_2 \right|^{2}dxdt\\
		\quad \leq \ \ \ \ \ C\displaystyle\int_{Q}\rho^{2}_{2} |p^i_2|^2dxdt+\|c_{22}\|_{L^\infty(Q)}\displaystyle\int_{Q}\rho^{2}_{2}\left|p^i_2 \right|^{2}dxdt\\
		\quad \leq \ \ C\displaystyle\int_{Q}\rho^{2}_{2}\left|p^i_2 \right|^{2}dxdt\\
		\quad \leq C\|(y_1,y_2,p_i^j,h)\|^2_{\mathcal{Y}}.
	\end{array}
\end{equation*}
Finally, from the hypothesis $\rho^2 y^i_{2,d} \in L^2(Q)$, we have that $I_{10}$ is bounded.
 Thus $\mathcal{A}$ is well-defined.

\end{comment}

\end{proof}

\begin{lema}\label{DA continuo3}
The mapping $\mathcal{A}: \mathcal{Y}\longrightarrow  \mathcal{Z}$ is continuously differentiable.
\end{lema}
\begin{proof}
First we prove that $\mathcal{A}$ is Gateaux differentiable at any $(y_1,y_2,p_i^j,h) \in \mathcal{Y}$ and let us compute the
$\textit{G-derivative}$ ${\mathcal{A}}^{\prime}(y_1,y_2,p_i^j,h)$.
Consider the linear mapping $D \mathcal{A}: \mathcal{Y} \to \mathcal{Z}$ given by
$$
D\mathcal{A}(y_1,y_2,p_i^j,h) = (D\mathcal{A}^1_0,D\mathcal{A}^2_0,D\mathcal{A}^1_1,D\mathcal{A}^2_1,D\mathcal{A}^1_2,D\mathcal{A}^2_2,D\mathcal{A}^1_3,D\mathcal{A}^2_3),
$$
where, for $i=1,2$,
\begin{equation*}
    	%\left\{\begin{split}
    \begin{cases}
		D\mathcal{A}^1_0(\bar y_1,\bar y_2,\bar p_i^j,\bar h) = &  \, \bar y_{1t} - b(t)(a(x) \bar y_{1x})_x -B\sqrt{a} \bar y_{1x} +  D_1 F_1(y_1,y_2)\bar y_1 + D_2 F_1(y_1,y_2)\bar y_2 - \bar h \cara_{\mathcal{O}} \\
        &+ \frac{1}{\mu_1} \bar p^1_1 \cara_{\mathcal{O}_{1}} + \frac{1}{\mu_2} \bar p^2_1 \cara_{\mathcal{O}_{2}}, \\
		D\mathcal{A}^2_0(\bar y_1,\bar y_2,\bar p_i^j,\bar h) = &  \, \bar y_{2t} - b(t)(a(x) \bar y_{2x})_x - B\sqrt{a} \bar y_{2x} + D_1 F_2(y_1,y_2)\bar y_1 + D_2 F_2(y_1,y_2)\bar y_2, \\
        D\mathcal{A}^1_i(\bar y_1,\bar y_2,\bar p_i^j,\bar h) =& - \bar p_{1t}^i - (a(x) \bar p_{1x}^i)_x +\left(B\sqrt{a} \bar p_1^i\right)_x + D_{11}^2 F_1(y_1,y_2)\bar y_1 p^i_1 + D_{12}^2 F_1(y_1,y_2)\bar y_2 p^i_1 \\
        &+ D_{11}^2 F_2(y_1,y_2)\bar y_1 p^i_2 + D_{12}^2 F_2(y_1,y_2)\bar y_2 p^i_2 
        %+ F_1'(y_1)\bar p^i_1 +c_{21}\bar p^i_2 
        + D_1 F_1(y_1,y_2) \bar p^i_1 + D_1 F_2(y_1,y_2) \bar p^i_2 \\
        &- \alpha_i \bar y_1 \cara_{\mathcal{O}_{id}}, \\
		%D\mathcal{A}^1_i(\bar y_1,\bar y_2,\bar p_i^j,\bar h) =& - \bar p_{1t}^i - (a(x) \bar p_{1x}^i)_x +\left(B\sqrt{a} \bar p_1^i\right)_x+ F_1''(y_1)\bar y_1 p^i_1 + F_1'(y_1)\bar p^i_1 +c_{21}\bar p^i_2- \alpha_i \bar y_1 \cara_{\mathcal{O}_{id}}, \\
		D\mathcal{A}^2_i(\bar y_1,\bar y_2,\bar p_i^j,\bar h) =& - \bar p_{2t}^i - (a(x) \bar p_{2x}^i)_x +\left(B\sqrt{a} \bar p_2^i\right)_x + D_{21}^2 F_1(y_1,y_2)\bar y_1 p^i_1 + D_{22}^2 F_1(y_1,y_2)\bar y_2 p^i_1 \\
        &+ D_{21}^2 F_2(y_1,y_2)\bar y_1 p^i_2 + D_{22}^2 F_2(y_1,y_2)\bar y_2 p^i_2 
        %+ F_1'(y_1)\bar p^i_1 +c_{21}\bar p^i_2 
        + D_2 F_1(y_1,y_2) \bar p^i_1 + D_2 F_2(y_1,y_2) \bar p^i_2  \\
        %F_2''(y_2)\bar y_2 p^i_2 + F_2'(y_2)\bar p^i_2 
        &- \alpha_i \bar y_2 \cara_{\mathcal{O}_{id}}, \\
        %D\mathcal{A}^2_i(\bar y_1,\bar y_2,\bar p_i^j,\bar h) =& - \bar p_{2t}^i - (a(x) \bar p_{2x}^i)_x +\left(B\sqrt{a} \bar p_2^i\right)_x+ F_2''(y_2)\bar y_2 p^i_2 + F_2'(y_2)\bar p^i_2 - \alpha_i \bar y_2 \cara_{\mathcal{O}_{id}}, \\
		D\mathcal{A}^1_3(\bar y_1,\bar y_2,\bar p_i^j,\bar h) =& \, \bar y_1(0),\\
		D\mathcal{A}^2_3(\bar y_1,\bar y_2,\bar p_i^j,\bar h) =& \, \bar y_2(0).
	%\end{split}\right.    
    \end{cases}
\end{equation*}

We have to show that, for $i=0,1,2,3$ and $j=1,2$, 
$$
\frac{1}{\lambda}\left[ \mathcal{A}^j_i (( y_1, y_2, p_i^j, h)+\lambda(\bar y_1,\bar y_2,\bar p_i^j,\bar h)) - \mathcal{A}^j_i ( y_1, y_2, p_i^j, h) \right] \to D\mathcal{A}^j_i(\bar y_1,\bar y_2,\bar p_i^j,\bar h),
$$
strongly in the corresponding factor of $\mathcal{Z}$ as $\lambda \to 0$.

Indeed, we have
\begin{equation*}		
	\begin{split}
		&\left\| \frac{1}{\lambda}\left[ \mathcal{A}^1_0 (( y_1, y_2, p_i^j, h)+\lambda(\bar y_1,\bar y_2,\bar p_i^j,\bar h)) - \mathcal{A}^1_0( y_1, y_2, p_i^j, h) \right] - D\mathcal{A}^1_0(\bar y_1,\bar y_2,\bar p_i^j,\bar h) \right\|^{2}_{L^2(\rho_2^2,Q)} \\
		&=\left\|
		\frac{1}{\lambda}\left[ y_{1t}+\lambda \bar y_{1t} - b(t)\left( a(x) (y_{1x}+\lambda \bar y_{1x}) \right)_x - B\sqrt{a}(y_{1x}+\lambda \bar y_{1x}) + F_1(y_1+\lambda \bar y_1, y_2+\lambda \bar y_2) \right.  \right. \\
		& \left. \left.  \ \ \ \ \ \ \ \ \ \  - (h+\lambda \bar h) \cara_{\mathcal{O}} + \frac{1}{\mu_1}  (p^1_1+\lambda \bar p^1_1) \cara_{\mathcal{O}_1} + \frac{1}{\mu_2} (p^2_1+\lambda \bar p^2_1) \cara_{\mathcal{O}_2} \right. \right.\\
		& \left. \left. \ \ \ \ \ \ \ \ \ \ \  -y_{1t} + b(t)\left( a(x) y_{1x} \right)_x + B\sqrt{a}y_{1x} - F_1(y_1,y_2) + h 1_{O}  - \frac{1}{\mu_1} p^1_1 \cara_{\mathcal{O}_1} - \frac{1}{\mu_2} p^2_1 \cara_{\mathcal{O}_2} \right] \right.\\
		& \ \ \ \ \ \ \ \ \ \ \left. -\bar y_{1t} + b(t)(a(x) \bar y_{1x})_x +B\sqrt{a} \bar y_{1x} - D_1 F_1(y_1,y_2)\bar y_1 - D_2 F_1(y_1,y_2)\bar y_2 + \bar h \cara_{\mathcal{O}} -\frac{1}{\mu_1} \bar p^1_1 \cara_{\mathcal{O}_{1}} -\frac{1}{\mu_2} \bar p^2_1 \cara_{\mathcal{O}_{2}} \right\|^{2}_{L^2(\rho_2^2,Q)} \\
		& = \int_Q \rho_2^2 \left| \frac{1}{\lambda}( F_1(y_1+\lambda\bar y_1,y_2 + \lambda\bar y_2) - F_1(y_1,y_2)) - D_1 F_1(y_1,y_2)\bar y_1 - D_2 F_1(y_1,y_2)\bar y_2 \right|^2 = J_1.
		%& + \int_Q \rho_2^2 \left| l(\int_0^1 (y+\lambda\bar y)) - l(\int_0^1 y) \right|^2 |(a(x)\bar y_x)_x|^2 = J_1 + J_2,
	\end{split}
\end{equation*}

From the Mean Value Theorem, there exists 
%$c \in [a,b]$, where $a=\min \{y_1, y_1 +\lambda \bar y_1 \}$ and $b(t) = \max \{ y_1 , y_1 +\lambda \bar y_1 \}$, 
$\theta \in [0,1]$ such that
such that
\begin{eqnarray*}
%	J_1&=&\int_Q \rho_2^2 \left| \frac{1}{\lambda}(F_1(y_1+\lambda\bar y_1)-F_1(y_1)) - F_1'(y_1)\bar y_1 \right|^2\\
	J_1&=&\int_Q \rho_2^2 \left| \frac{1}{\lambda} \nabla F_1(y_1+\theta\lambda\bar y_1,y_2+\theta\lambda\bar y_2)\cdot (\lambda \bar y_1,\lambda \bar y_2) - D_1 F_1(y_1,y_2)\bar y_1 - D_2 F_1(y_1,y_2)\bar y_2 \right|^2\\
    &=& \int_Q \rho_2^2 \left| \nabla F_1(y_1+\theta\lambda\bar y_1,y_2+\theta\lambda\bar y_2)\cdot (\bar y_1,\bar y_2) - \nabla F_1(y_1,y_2) \cdot (\bar y_1,\bar y_2) \right|^2\\
    &=& \int_Q \rho_2^2 \left| \nabla F_1(y_1+\theta\lambda\bar y_1,y_2+\theta\lambda\bar y_2) - \nabla F_1(y_1,y_2) \right|^2 (\bar y_1^2 + \bar y_2^2 ).
	%&=&\int_Q \rho_2^2 \left| F_1'(c) \bar y_1 - F'_1(y_1)\bar y_1 \right|^2.
\end{eqnarray*}
Since $F_1$ is in the class $C^2$, we get that $J_1$ converges to zero as $\lambda \to 0$ by Lebesgue's dominated convergence theorem. 
%because $c\rightarrow y_1$.

Analogously, for $D\mathcal{A}^2_0(\bar y_1,\bar y_2,\bar p_i^j,\bar h)$ we get
\begin{equation*}		
	\begin{split}
		&\left\| \frac{1}{\lambda}\left[ \mathcal{A}^2_0 (( y_1, y_2, p_i^j, h)+\lambda(\bar y_1,\bar y_2,\bar p_i^j,\bar h)) - \mathcal{A}^2_0( y_1, y_2, p_i^j, h) \right] - D\mathcal{A}^2_0(\bar y_1,\bar y_2,\bar p_i^j,\bar h) \right\|^{2}_{L^2(\rho_2^2,Q)} \\
		& = \int_Q \rho_2^2 \left| \frac{1}{\lambda}(F_2(y_1+\lambda\bar y_1,y_2+\lambda\bar y_2)-F_2(y_1,y_2)) - D_1 F_2(y_1,y_2)\bar y_1 - D_2 F_2(y_1,y_2)\bar y_2 \right|^2 = J_2.
		%& + \int_Q \rho_2^2 \left| l(\int_0^1 (y+\lambda\bar y)) - l(\int_0^1 y) \right|^2 |(a(x)\bar y_x)_x|^2 = J_1 + J_2,
	\end{split}
\end{equation*}
As before, using the Mean Value Theorem we have, for some $\theta \in [0,1]$,  %$J_2=\int_Q \rho_2^2 \left| F_2'(c) \bar y_2 - F_2'(y_2)\bar y_2 \right|^2$,
$$
J_2 = \int_Q \rho_2^2 \left| \nabla F_1(y_1+\theta\lambda\bar y_1,y_2+\theta\lambda\bar y_2) - \nabla F_1(y_1,y_2) \right|^2 (\bar y_1^2 + \bar y_2^2 ),
$$
and since $F_2$ is of class $C^2$, then $J_2$ converges to zero as $\lambda \to 0$.

\begin{comment}
From the Mean Value Theorem, there exists $c \in [a,b]$, where $a=\min \{y_2, y_2 +\lambda \bar y_2 \}$ and $b(t) = \max \{ y_2 , y_2 +\lambda \bar y_2 \}$, such that
\begin{eqnarray*}
	J_2&=&\int_Q \rho_2^2 \left| \frac{1}{\lambda}(F_2(y_2+\lambda\bar y_2)-F_2(y_2)) - F_2'(y_2)\bar y_2 \right|^2\\
	&=&\int_Q \rho_2^2 \left| \frac{1}{\lambda}(F_2'(c)\lambda \bar y_2) - F_2'(y_2)\bar y_2 \right|^2\\
	&=&\int_Q \rho_2^2 \left| F_2'(c) \bar y_2 - F_2'(y_2)\bar y_2 \right|^2
\end{eqnarray*}
Since $F_2$ is in the class $C^2$, we get that $J_2$ converges to zero as $\lambda \to 0$ because $c\rightarrow y_2$. \\
\end{comment}

On the other hand, analogously, for $\mathcal{A}^1_i$, $i=1,2$, we get
\begin{equation}
	\begin{split}
		&\left\| \frac{1}{\lambda}\left[ \mathcal{A}^1_i (( y_1, y_2, p_i^j, h)+\lambda(\bar y_1,\bar y_2,\bar p_i^j,\bar h)) - \mathcal{A}^1_i( y_1, y_2, p_i^j, h) \right] - D\mathcal{A}^1_i(\bar y_1,\bar y_2,\bar p_i^j,\bar h) \right\|^{2}_{L^2(\rho_2^2,Q)} \\
        &= \left\| \frac{1}{\lambda} \left[ D_1 F_1(y_1+\lambda \bar y_1,y_2+\lambda \bar y_2) (p^i_1 +\lambda \bar p^i_1) + D_1 F_2(y_1+\lambda \bar y_1,y_2+\lambda \bar y_2) (p^i_2 +\lambda \bar p^i_2) \right.\right. \\
        &\left.\left.\qquad - D_1 F_1(y_1,y_2) p^i_1 - D_1 F_2(y_1,y_2) p^i_2 \right] - D_{11}^2 F_1(y_1,y_2)\bar y_1 p^i_1 - D_{12}^2 F_1(y_1,y_2)\bar y_2 p^i_1 \right. \\
        &\left.\qquad - D_{11}^2 F_2(y_1,y_2)\bar y_1 p^i_2 - D_{12}^2 F_2(y_1,y_2)\bar y_2 p^i_2 
        - D_1 F_1(y_1,y_2) \bar p^i_1 - D_1 F_2(y_1,y_2) \bar p^i_2
        %- F''_1(y_1)\bar y_1 p^i_1 - F_1'(y_1)\bar p^i_1
        \right\|^{2}_{L^2(\rho_2^2,Q)} := J_3 \\
		&= \left\| \frac{1}{\lambda} \left( F_1'(y_1+\lambda \bar y_1) (p^i_1 +\lambda \bar p^i_1) - F_1'(y_1) p^i_1  \right) - F''_1(y_1)\bar y_1 p^i_1 - F_1'(y_1)\bar p^i_1 \right\|^{2}_{L^2(\rho_2^2,Q)} := J_3
	\end{split}
\end{equation}
Then, 
\begin{equation*}
	\begin{split}
		J_3  
        %\int_Q \rho_2^2 \left| \frac{1}{\lambda} \left( F_1'(y_1+\lambda \bar y_1) (p^i_1 +\lambda \bar p^i_1) - F_1'(y_1) p^i_1  \right) - F''_1(y_1)\bar y_1 p^i_1 - F_1'(y_1)\bar p^i_1 \right|^2 \\
		&\leq \int_Q \rho_2^2 \left| \frac{1}{\lambda}\left[D_1 F_1(y_1+\lambda \bar y_1,y_2+\lambda \bar y_2) p^i_1 - D_1 F_1(y_1,y_2) p^i_1 \right] - D_{11}^2 F_1(y_1,y_2)\bar y_1 p^i_1 - D_{12}^2 F_1(y_1,y_2)\bar y_2 p^i_1\right|^2  \\
        &+ \int_Q \rho_2^2 \left| \frac{1}{\lambda}\left[ D_1 F_2(y_1+\lambda \bar y_1,y_2+\lambda \bar y_2) p^i_2  - D_1 F_2(y_1,y_2) p^i_2 \right] - D_{11}^2 F_2(y_1,y_2)\bar y_1 p^i_2 - D_{12}^2 F_2(y_1,y_2)\bar y_2 p^i_2 \right|^2 \\
        &+\int_Q \rho_2^2 \left| D_1 F_1(y_1+\lambda \bar y_1,y_2+\lambda \bar y_2) \bar p^i_1 + D_1 F_2(y_1+\lambda \bar y_1,y_2+\lambda \bar y_2) \bar p^i_2  - D_1 F_1(y_1,y_2) \bar p^i_1 - D_1 F_2(y_1,y_2) \bar p^i_2 \right|^2 \\
        &= J_{31} + J_{32} + J_{33}.
    \end{split}
\end{equation*}
Using the Mean Value Theorem, for $\theta \in [0,1]$,
\begin{equation*}
    \begin{split}
        J_{31} &\leq \int_Q \rho_2^2 \left| \frac{1}{\lambda}
        \nabla (D_1 F_1) (y_1+ \theta\lambda \bar y_1,y_2+\theta\lambda \bar y_2)\cdot(\lambda \bar y_1,\lambda \bar y_2) p^i_1 - \nabla (D_1 F_1)(y_1,y_2)\cdot(\bar y_1,\bar y_2)p^i_1
        \right|^2,
    \end{split}
\end{equation*}
and the fact that $D_1 F_1$ is in the class $C^1$, we get that $J_{31}$ converges to zero as $\lambda \to 0$. By Analogous estimates we get that $J_{32}$ and $F_{33}$ converges to zero as $\lambda \to 0$

\begin{comment}
Then, using the Mean Value Theorem,
\begin{equation*}
	\begin{split}
		J_3 %&= 
        %\int_Q \rho_2^2 \left| \frac{1}{\lambda} \left( F_1'(y_1+\lambda \bar y_1) (p^i_1 +\lambda \bar p^i_1) - F_1'(y_1) p^i_1  \right) - F''_1(y_1)\bar y_1 p^i_1 - F_1'(y_1)\bar p^i_1 \right|^2 \\
		&\leq C \int_Q \rho_2^2 \left| \frac{1}{\lambda} \left( F_1'(y_1+\lambda \bar y_1) - F_1'(y_1)\right) p^i_1 - F_1''(y_1)\bar y_1 p^i_1 \right|^2 + C \int_Q \rho_2^2 \left| F_1'(y_1+\lambda \bar y_1) \bar p^i_1 - F_1'(y_1)\bar p^i_1  \right|^2\\
		&\leq C \int_Q \rho_2^2 \left| \frac{1}{\lambda} \left( F_1''(c)\lambda \bar y_1 \right) p^i_1 - F_1''(y_1)\bar y_1 p^i_1 \right|^2 + C \int_Q \rho_2^2 \left| F_1'(y_1+\lambda \bar y_1) \bar p^i_1 - F_1'(y_1)\bar p^i_1  \right|^2\\
		&\leq C \int_Q \rho_2^2 \left|   F_1''(c) \bar y_1  p^i_1 - F_1''(y_1)\bar y_1 p^i_1 \right|^2 + C \int_Q \rho_2^2 \left| F_1'(y_1+\lambda \bar y_1) \bar p^i_1 - F_1'(y_1)\bar p^i_1  \right|^2
	\end{split}
\end{equation*}
and from the same estimates used for $J_2$ and the fact that $F_1'$ is in the class $C^1$, we get that $J_3$ converges to zero as $\lambda \to 0$. %because $c\rightarrow y_1$.
\end{comment}

For $\mathcal{A}^2_i$, $i=1,2$, a similar estimation gives
\begin{equation*}
	\begin{split}
        J_4 &:= \left\| \frac{1}{\lambda} \left[ D_2 F_1(y_1+\lambda \bar y_1,y_2+\lambda \bar y_2) (p^i_1 +\lambda \bar p^i_1) + D_2 F_2(y_1+\lambda \bar y_1,y_2+\lambda \bar y_2) (p^i_2 +\lambda \bar p^i_2) \right.\right. \\
        &\left.\left.\qquad - D_2 F_1(y_1,y_2) p^i_1 - D_2 F_2(y_1,y_2) p^i_2 \right] - D_{21}^2 F_1(y_1,y_2)\bar y_1 p^i_1 - D_{22}^2 F_1(y_1,y_2)\bar y_2 p^i_1 \right. \\
        &\left.\qquad - D_{21}^2 F_2(y_1,y_2)\bar y_1 p^i_2 - D_{22}^2 F_2(y_1,y_2)\bar y_2 p^i_2 
        - D_2 F_1(y_1,y_2) \bar p^i_1 - D_2 F_2(y_1,y_2) \bar p^i_2
        %- F''_1(y_1)\bar y_1 p^i_1 - F_1'(y_1)\bar p^i_1
        \right\|^{2}_{L^2(\rho_2^2,Q)}\\        
        %&\left\| \frac{1}{\lambda} \left( F_2'(y_2+\lambda \bar y_2) (p^i_2 +\lambda \bar p^i_2) - F_2'(y_2) p^i_2  \right) - F''_2(y_2)\bar y_2 p^i_2 - F_2'(y_2)\bar p^i_2 \right\|^{2}_{L^2(\rho_2^2,Q)}\\
        %&\leq C \int_Q \rho_2^2 \left|   F_2''(c) \bar y_2  p^i_2 - F_2''(y_2)\bar y_2 p^i_2 \right|^2 + C \int_Q \rho_2^2 \left| F_2'(y_2+\lambda \bar y_2) \bar p^i_2 - F_2'(y_2)\bar p^i_2  \right|^2,
        &\leq \int_Q \rho_2^2 \left| \frac{1}{\lambda}\left[D_2 F_1(y_1+\lambda \bar y_1,y_2+\lambda \bar y_2) p^i_1 - D_2 F_1(y_1,y_2) p^i_1 \right] - D_{21}^2 F_1(y_1,y_2)\bar y_1 p^i_1 - D_{22}^2 F_1(y_1,y_2)\bar y_2 p^i_1\right|^2  \\
        &+ \int_Q \rho_2^2 \left| \frac{1}{\lambda}\left[ D_2 F_2(y_1+\lambda \bar y_1,y_2+\lambda \bar y_2) p^i_2  - D_2 F_2(y_1,y_2) p^i_2 \right] - D_{21}^2 F_2(y_1,y_2)\bar y_1 p^i_2 - D_{22}^2 F_2(y_1,y_2)\bar y_2 p^i_2 \right|^2 \\
        &+\int_Q \rho_2^2 \left| D_2 F_1(y_1+\lambda \bar y_1,y_2+\lambda \bar y_2) \bar p^i_1 - D_2 F_1(y_1,y_2) \bar p^i_1 \right|^2 + \left| D_2 F_2(y_1+\lambda \bar y_1,y_2+\lambda \bar y_2) \bar p^i_2 - D_2 F_2(y_1,y_2) \bar p^i_2 \right|^2.
    \end{split}
\end{equation*}
By the Mean Value Theorem, and the fact that $D_2 F_i$, $i=1,2$, are in the class $C^1$, we get that $J_{4}$ converges to zero as $\lambda \to 0$.

 This finish the proof that $\mathcal{A}$ is Gateaux differentiable, with a \textit{G-derivative}, given by\\
 $\mathcal{A^{\prime}}( y_1, y_2, p_i^j, h)= D\mathcal{A}( y_1, y_2, p_i^j, h)$.

 Now, take $( y_1, y_2, p_i^j, h)\in \mathcal{Y}$ and let $\{(y_{1,n},y_{2,n},p^{j}_{i,n},h_{n})\}_{n=0}^{\infty}$ be a sequence which converges to  $( y_1, y_2, p_i^j, h)$ in $\mathcal{Y}$. For each $(\bar y_1,\bar y_2,\bar p_i^j,\bar h)\in B_{r}(0)$, we have to prove that 
\begin{equation}
	\begin{split}
		&\left\|(D\mathcal{A}^l_m(y_{1,n},y_{2,n},p^{j}_{i,n},h_{n})-D\mathcal{A}^l_m( y_1, y_2, p_i^j, h))(\bar y_1,\bar y_2,\bar p_i^j,\bar h) \right\|^{2}_{L^2(\rho_2^2,Q)}\rightarrow 0 
	\end{split}
\end{equation}
for $m=0,1,2,3. \ \ l=1,2$.

First, for $m=0$ and $l=1$, we have
\begin{eqnarray*}
	(D\mathcal{A}^1_0(y_{1,n},y_{2,n},p^{j}_{i,n},h_{n})-D\mathcal{A}^1_0( y_1, y_2, p_i^j, h))(\bar y_1,\bar y_2,\bar p_i^j,\bar h)= &(D_1 F_1(y_{1,n},y_{2,n})-D_1 F_1(y_1,y_2))\bar y_1\\
    &+ (D_2 F_1(y_{1,n},y_{2,n})-D_2 F_1(y_1,y_2))\bar y_2 = X_0^1.
\end{eqnarray*}
Then, 
\begin{eqnarray*}
	\int_{0}^{T}\int_{0}^{1}\rho^{2}_{2}|X_{1}^{0}|^{2} dxdt&\leq&  \int_{0}^{T}\int_{0}^{1}\rho_{2}^{2} |(D_1 F_1(y_{1,n},y_{2,n})-D_1 F_1(y_1,y_2))|^2 |\bar y_1|^{2}dxdt\\
    && + \int_{0}^{T}\int_{0}^{1}\rho_{2}^{2} |(D_2 F_1(y_{1,n},y_{2,n})-D_2 F_1(y_1,y_2))|^2 |\bar y_2|^{2}dxdt\\
	&\leq&C\int_{0}^{T}\int_{0}^{1}\rho_{2}^{2} M^2 ( |y_{1,n}-y_1|^2 + |y_{2,n}-y_2|^2 ) (|\bar y_1|^{2} + |\bar y_2|^{2} ) dxdt
\end{eqnarray*}
An the other hand, we have by \eqref{eq:compara_rhos3} that  $\rho_2^2\leq C \rho_1^2 \rho^2_1\leq C \hat{\rho}^2 \rho_0^2 $, then
\begin{eqnarray*}
	\int_{0}^{T}\int_{0}^{1}\rho^{2}_{2}|X_{1}^{0}|^{2} dxdt&\leq& C\int_{0}^{T}\int_{0}^{1}\hat{\rho}^2 \rho_0^2  ( |y_{1,n}-y_1|^2 + |y_{2,n}-y_2|^2 ) (|\bar y_1|^{2} + |\bar y_2|^{2} )dxdt\\
	&\leq& C\sup_{[0,T]}\left(\hat{\rho}^2 (\| \bar{y}_1\|^2_{L^2(\Omega)} + \| \bar{y}_2\|^2_{L^2(\Omega)})\right)\int_{0}^{T}\int_{0}^{1} \rho_0^2  (|y_{1,n}-y_1|^2 + |y_{2,n}-y_2|^2 ) dxdt\\
	&\leq&C\|(\bar y_1,\bar y_2,\bar p_i^j,\bar h)\|_{\mathcal{Y}} \cdot \left( \|\rho_0 (y_{1,n}-y_1)\|_{L^2(Q)} + \|\rho_0 (y_{2,n}-y_2)\|_{L^2(Q)} \right)\\
	&\leq&C\|(y_{1,n}-y_1),(y_{2,n}-y_2),(p^{j}_{i,n}-p^{j}_i),(h_{n}-h)\|_{\mathcal{Y}}\rightarrow 0 \quad\text{as}\quad n \to \infty.
\end{eqnarray*}
Now, for $m=0$, $l=2$,
\begin{eqnarray*}
	&X_0^2&=(D\mathcal{A}^2_0(y_{1,n},y_{2,n},p^{j}_{i,n},h_{n})-D\mathcal{A}^2_0( y_1, y_2, p_i^j, h))(\bar y_1,\bar y_2,\bar p_i^j,\bar h)\\
    &&= (D_1 F_2(y_{1,n},y_{2,n})-D_1 F_2(y_1,y_2))\bar y_1 + (D_2 F_2(y_{1,n},y_{2,n})-D_2 F_2(y_1,y_2))\bar y_2.
\end{eqnarray*}
We have, analogously,
\begin{eqnarray*}
	\int_{0}^{T}\int_{0}^{1}\rho^{2}_{2}|X_{0}^{2}|^{2} dxdt&\leq& 
    %C \int_{0}^{T}\int_{0}^{1}\rho_{2}^{2} |(F_2'(y_{2,n})-F_2'(y_2))|^2 |\bar y_2|^{2}dxdt\\
	%&\leq&C\int_{0}^{T}\int_{0}^{1}\rho_{2}^{2} M^2 |y_{2,n}-y_2|^2 |\bar y_2|^{2}dxdt\\
	%&\leq& 
    C\int_{0}^{T}\int_{0}^{1}\hat{\rho}^2 \rho_0^2 ( |y_{1,n}-y_1|^2+ |y_{2,n}-y_2|^2) (|\bar y_1|^{2} + |\bar y_2|^{2}) dxdt\\
	&\leq& C\sup_{[0,T]}\left(\hat{\rho}^2 (\| \bar{y}_1\|^2_{L^2(\Omega)} + \| \bar{y}_2\|^2_{L^2(\Omega)})\right)\int_{0}^{T}\int_{0}^{1} \rho_0^2  (|y_{1,n}-y_1|^2 + |y_{2,n}-y_2|^2 ) dxdt\\
	%&\leq&C\|(\bar y_1,\bar y_2,\bar p_i^j,\bar h)\|_{\mathcal{Y}} \cdot \|\rho_0 (y_{2,n}-y_2)\|_{L^2(Q)}\\
	&\leq&C\|(y_{1,n}-y_1),(y_{2,n}-y_2),(p^{j}_{i,n}-p^{j}_i),(h_{n}-h)\|_{\mathcal{Y}}\rightarrow 0, \quad\text{as}\quad n \to \infty.
\end{eqnarray*}

%---------------------------------------- HASTA AQUI ------------------

Now, consider the terms $D\mathcal{A}^1_{i}$, $i=1,2$. Then 
%\hspace*{-0.5cm}\begin{eqnarray*}
%	&(D\mathcal{A}^1_i(y_{1,n},y_{2,n},p^{j}_{i,n},h_{n})-D\mathcal{A}^1_i( y_1, y_2, p_i^j, h))(\bar y_1,\bar y_2,\bar p_i^j,\bar h) =& F_1''(y_{1,n})\bar y_1 p_{1,n}^i + F_1'(y_{1,n})\bar p^i_1\\
%	&&- F_1''(y_1)\bar y_1 p^i_1 - F_1'(y_1)\bar p^i_1 = X_i^1,
%\end{eqnarray*}
\hspace*{-0.5cm}\begin{eqnarray*}
	%&X_i^1&=
    &&(D\mathcal{A}^1_i(y_{1,n},y_{2,n},p^{j}_{i,n},h_{n})-D\mathcal{A}^1_i( y_1, y_2, p_i^j, h))(\bar y_1,\bar y_2,\bar p_i^j,\bar h) \\
    %&\quad = F_1''(y_{1,n})\bar y_1 p_{1,n}^i + F_1'(y_{1,n})\bar p^i_1\\
	%&- F_1''(y_1)\bar y_1 p^i_1 - F_1'(y_1)\bar p^i_1 = X_i^1,\\
    &&\quad=(D_{11}^2 F_1(y_{1,n,}y_{2,n}) p^i_{1,n} - D_{11}^2 F_1(y_{1,n,}y_2) p^i_1)\bar y_1 + (D_{11}^2 F_2(y_{1,n},y_{2,n})p^i_{2,n} - D_{11}^2 F_2(y_1,y_2)p^i_2) \bar y_1 \\
    &&\quad\quad+ (D_{12}^2 F_1(y_{1,n},y_{2,n})p^i_{1,n} - D_{12}^2 F_1(y_1,y_2)p^i_1)\bar y_2 + (D_{12}^2 F_2(y_{1,n},y_{2,n})p^i_{2,n} - D_{12}^2 F_2(y_1,y_2) p^i_2) \bar y_2 \\
    &&\quad\quad+ (D_1 F_1(y_{1,n},y_{2,n}) - D_1 F_1(y_1,y_2)) \bar p^i_1 + (D_1 F_2(y_{1,n},y_{2,n}) - D_1 F_2(y_1,y_2))\bar p^i_2\\
    &&\quad=X_{i,1}^1+X_{i,2}^1+X_{i,3}^1+X_{i,4}^1+X_{i,5}^1+X_{i,6}^1.
\end{eqnarray*}
First we estimate
\begin{eqnarray*}
%\hspace*{-0.6cm}	                     
    \int_{0}^{T}\int_{0}^{1}\rho^{2}_{2}|X_{i,1}^{1}|^{2} dxdt
    %&\leq& C \int_{0}^{T}\int_{0}^{1}\rho_{2}^{2} |F_1''(y_{1,n})\bar y_1 p_{1,n}^i + F_1'(y_{1,n})\bar p^i_1- F_1''(y_1)\bar y_1 p^i_1 - F_1'(y_1)\bar p^i_1|^2 dxdt\\
	&\leq&\int_{0}^{T}\int_{0}^{1}\rho_{2}^{2} |D_{11}^2 F_1(y_{1,n},y_{2,n})p_{1,n}^i - D_{11}^2 F_1(y_1,y_2) p^i_1 |^2 |\bar y_1|^2 \\
    %+ C\int_{0}^{T}\int_{0}^{1}\rho_{2}^{2} | F_1'(y_{1,n})\bar p^i_1 - F_1'(y_1)\bar p^i_1|^2 \\
%	&\leq &C\int_{0}^{T}\int_{0}^{1}\rho_{2}^{2} |F_1''(y_{1,n}) (p_{1,n}^i-p^i_1)+(F_1''(y_{1,n}) - F_1''(y_1)) |p^i_1 |^2 |\bar y_1|^2\\
%	&& +\int_{0}^{T}\int_{0}^{1}\rho_{2}^{2} | F_1'(y_{1,n}) - F_1'(y_1)|^2 |\bar p^i_1|^2\\
	&\leq &\int_{0}^{T}\int_{0}^{1}\rho_{2}^{2} |D_{11}^2 F_1(y_{1,n},y_{2,n}) (p_{1,n}^i-p^i_1) |^2 |\bar y_1|^2 \\
    &&+ \int_{0}^{T}\int_{0}^{1}\rho_{2}^{2} |(D_{11}^2 F_1(y_{1,n},y_{2,n}) - D_{11}^2 F_1(y_1,y_2))|^2 |p^i_1|^2 |\bar y_1|^2\\
	%&& +C\int_{0}^{T}\int_{0}^{1}\rho_{2}^{2} | F_1'(y_{1,n}) - F_1'(y_1)|^2 |\bar p^i_1|^2\\
	&\leq &C\int_{0}^{T}\int_{0}^{1}\rho_{2}^{2} |p_{1,n}^i-p^i_1 |^2 |\bar y_1|^2 + C\int_{0}^{T}\int_{0}^{1}\rho_{2}^{2} ( |y_{1,n} - y_1|^2 + |y_{2,n} - y_2|^2 ) | p^i_1 |^2 |\bar y_1|^2.
%	&& +C\int_{0}^{T}\int_{0}^{1}\rho_{2}^{2} |y_{1,n} - y_1|^2 |\bar p^i_1|^2.
\end{eqnarray*}

Moreover, we have by \eqref{eq:compara_rhos3} that  $\rho_2^2\leq C \rho_1^2 \rho^2_1\leq C \hat{\rho}^2 \rho_0^2\leq C \hat{\rho}^4 \rho_0^2 $. Then,
\begin{eqnarray*}
	\int_{0}^{T}\int_{0}^{1}\rho^{2}_{2}|X_{i,1}^{1}|^{2} dxdt &\leq& 
    %C\int_{0}^{T}\int_{0}^{1}\rho_{2}^{2} |p_{1,n}^i-p^i_1 |^2 |\bar y_1|^2+\int_{0}^{T}\int_{0}^{1}\rho_{2}^{2} |y_{1,n} - y_1|^2 | p^i_1 |^2 |\bar y_1|^2\\
	%&& +\int_{0}^{T}\int_{0}^{1}\rho_{2}^{2} |y_{1,n} - y_1|^2 |\bar p^i_1|^2\\
	%&\leq &
    C\int_{0}^{T}\int_{0}^{1}\hat{\rho}^2|\bar y_1|^2  \rho_0^2|p_{1,n}^i-p^i_1 |^2 \\
    &&+ C \int_{0}^{T}\int_{0}^{1}\hat{\rho}^2| p^i_1 |^2 \hat{\rho}^2|\bar y_1|^2 \rho_0^2 (|y_{1,n} - y_1|^2 + |y_{2,n} - y_2|^2) \\
	%&& + C \int_{0}^{T}\int_{0}^{1}\hat{\rho}^2|\bar p^i_1|^2  \rho_0^2 |y_{1,n} - y_1|^2 \\
	&\leq&  C\sup_{[0,T]}\left(\hat{\rho}^2\| \bar{y}_1\|^2_{L^2(\Omega)}\right)\int_{0}^{T}\int_{0}^{1} \rho_0^2  |p^i_{1,n}-p^i_1|^2\\
    &&+C\sup_{[0,T]}\left(\hat{\rho}^2\| \bar{p}^i_1\|^2_{L^2(\Omega)}\right)\sup_{[0,T]}\left(\hat{\rho}^2\| \bar{y}_1\|^2_{L^2(\Omega)}\right)\int_{0}^{T}\int_{0}^{1} \rho_0^2  (|y_{1,n}-y_1|^2 + |y_{2,n}-y_2|^2) \\
	%&& +C\sup_{[0,T]}\left(\hat{\rho}^2\| \bar{p}^i_1\|^2_{L^2(\Omega)}\right)\int_{0}^{T}\int_{0}^{1} \rho_0^2  |y_{1,n}-y_1|^2 \\
	&\leq & C \|(\bar y_1,\bar y_2,\bar p_i^j,\bar h)\|_{\mathcal{Y}} \cdot \|\rho_0 (p^i_{1,n}-p^i_1)\|_{L^2(Q)}\\
    &&+ C \|(\bar y_1,\bar y_2,\bar p_i^j,\bar h)\|_{\mathcal{Y}}\cdot \|( \bar y_1, \bar y_2, \bar p_i^j, \bar h)\|_{\mathcal{Y}}\cdot (\|\rho_0 (y_{1,n}-y_1)\|_{L^2(Q)}+\|\rho_0 (y_{2,n}-y_2)\|_{L^2(Q)})\\
	%&&+ C \|(\bar y_1,\bar y_2,\bar p_i^j,\bar h)\|_{\mathcal{Y}} \cdot \|\rho_0 (y_{{1,n}}-y_1)\|_{L^2(Q)}\\
	&\leq& C\|(y_{1,n}-y_1),(y_{2,n}-y_2),(p^{j}_{i,n}-p^{j}_i),(h_{n}-h)\|_{\mathcal{Y}}\rightarrow 0.
\end{eqnarray*}
The terms $X_{i,k}^{1}$, $k=2,3,4$ are estimated similarly. On the other hand
\begin{eqnarray*}
%\hspace*{-0.6cm}	                     
    \int_{0}^{T}\int_{0}^{1}\rho^{2}_{2}|X_{i,5}^{1}|^{2} dxdt
    %&\leq& C \int_{0}^{T}\int_{0}^{1}\rho_{2}^{2} |F_1''(y_{1,n})\bar y_1 p_{1,n}^i + F_1'(y_{1,n})\bar p^i_1- F_1''(y_1)\bar y_1 p^i_1 - F_1'(y_1)\bar p^i_1|^2 dxdt\\
	&\leq&
    %C\int_{0}^{T}\int_{0}^{1}\rho_{2}^{2} |F_1''(y_{1,n})\bar y_1 p_{1,n}^i - F_1''(y_1)\bar y_1 p^i_1 |^2 +
    C\int_{0}^{T}\int_{0}^{1}\rho_{2}^{2} | D_1 F_1(y_{1,n},y_{2,n}) - D_1 F_1(y_1,y_2)|^2 |\bar p^i_1|^2 \\
%	&\leq &C\int_{0}^{T}\int_{0}^{1}\rho_{2}^{2} |F_1''(y_{1,n}) (p_{1,n}^i-p^i_1)+(F_1''(y_{1,n}) - F_1''(y_1)) |p^i_1 |^2 |\bar y_1|^2\\
%	&& +\int_{0}^{T}\int_{0}^{1}\rho_{2}^{2} | F_1'(y_{1,n}) - F_1'(y_1)|^2 |\bar p^i_1|^2\\
	%&\leq &C\int_{0}^{T}\int_{0}^{1}\rho_{2}^{2} |F_1''(y_{1,n}) (p_{1,n}^i-p^i_1) |^2 |\bar y_1|^2 + C\int_{0}^{T}\int_{0}^{1}\rho_{2}^{2} |(F''_1(y_{1,n}) - F_1''(y_1))|^2 |p^i_1|^2 |\bar y_1|^2\\
	%&& +C\int_{0}^{T}\int_{0}^{1}\rho_{2}^{2} | F_1'(y_{1,n}) - F_1'(y_1)|^2 |\bar p^i_1|^2\\
	&\leq &
    %C\int_{0}^{T}\int_{0}^{1}\rho_{2}^{2} |p_{1,n}^i-p^i_1 |^2 |\bar y_1|^2 + C\int_{0}^{T}\int_{0}^{1}\rho_{2}^{2} |y_{1,n} - y_1|^2 | p^i_1 |^2 |\bar y_1|^2\\
	C\int_{0}^{T}\int_{0}^{1}\rho_{2}^{2} ( |y_{1,n} - y_1|^2 + |y_{2,n} - y_2|^2 ) |\bar p^i_1|^2.
\end{eqnarray*}

Once more, from \eqref{eq:compara_rhos3} we have $\rho_2^2\leq C \rho_1^2 \rho^2_1\leq C \hat{\rho}^2 \rho_0^2\leq C \hat{\rho}^4 \rho_0^2 $. Then,
\begin{eqnarray*}
	\int_{0}^{T}\int_{0}^{1}\rho^{2}_{2}|X_{i,5}^{1}|^{2} dxdt &\leq& 
    %C\int_{0}^{T}\int_{0}^{1}\rho_{2}^{2} |p_{1,n}^i-p^i_1 |^2 |\bar y_1|^2+\int_{0}^{T}\int_{0}^{1}\rho_{2}^{2} |y_{1,n} - y_1|^2 | p^i_1 |^2 |\bar y_1|^2\\
	%&& +\int_{0}^{T}\int_{0}^{1}\rho_{2}^{2} |y_{1,n} - y_1|^2 |\bar p^i_1|^2\\
	%&\leq &
    %C\int_{0}^{T}\int_{0}^{1}\hat{\rho}^2|\bar y_1|^2  \rho_0^2|p_{1,n}^i-p^i_1 |^2 + C \int_{0}^{T}\int_{0}^{1}\hat{\rho}^2| p^i_1 |^2 \hat{\rho}^2|\bar y_1|^2 \rho_0^2 |y_{1,n} - y_1|^2 \\
	%&& + 
    C \int_{0}^{T}\int_{0}^{1}\hat{\rho}^2|\bar p^i_1|^2  \rho_0^2 (|y_{1,n} - y_1|^2 - |y_{2,n} - y_2|^2) \\
	&\leq&  
    %C\sup_{[0,T]}\left(\hat{\rho}^2\| \bar{y}_1\|^2_{L^2(\Omega)}\right)\int_{0}^{T}\int_{0}^{1} \rho_0^2  |p^i_{1,n}-p^i_1|^2\\
    %&&+C\sup_{[0,T]}\left(\hat{\rho}^2\| \bar{p}^i_1\|^2_{L^2(\Omega)}\right)\sup_{[0,T]}\left(\hat{\rho}^2\| \bar{y}_1\|^2_{L^2(\Omega)}\right)\int_{0}^{T}\int_{0}^{1} \rho_0^2  |y_{1,n}-y_1|^2 \\
	%&& +
    C\sup_{[0,T]}\left(\hat{\rho}^2\| \bar{p}^i_1\|^2_{L^2(\Omega)}\right)\int_{0}^{T}\int_{0}^{1} \rho_0^2  (|y_{1,n}-y_1|^2 + |y_{2,n}-y_2|^2) \\
	&\leq & 
    %C \|(\bar y_1,\bar y_2,\bar p_i^j,\bar h)\|_{\mathcal{Y}} \cdot \|\rho_0 (p^i_{1,n}-p^i_1)\|_{L^2(Q)}\\
    %&&+ C \|(\bar y_1,\bar y_2,\bar p_i^j,\bar h)\|_{\mathcal{Y}}\cdot \|( y_1, y_2, p_i^j, h)\|_{\mathcal{Y}}\cdot \|\rho_0 (y_{1,n}-y_1)\|_{L^2(Q)}\\
	%&&+ 
    C \|(\bar y_1,\bar y_2,\bar p_i^j,\bar h)\|_{\mathcal{Y}} \cdot (\|\rho_0 (y_{{1,n}}-y_1)\|_{L^2(Q)} + \|\rho_0 (y_{{2,n}}-y_2)\|_{L^2(Q)})\\
	&\leq& C\|(y_{1,n}-y_1),(y_{2,n}-y_2),(p^{j}_{i,n}-p^{j}_i),(h_{n}-h)\|_{\mathcal{Y}}\rightarrow 0.
\end{eqnarray*}
The term $X_{i,6}^{1}$ is estimated analogously.

Now, consider the term $D\mathcal{A}^2_{i}$. The estimates are the same as for $D\mathcal{A}^1_{i}$ using
\hspace*{-0.5cm}\begin{eqnarray*}
	%&X_i^1&=
    &&(D\mathcal{A}^2_i(y_{1,n},y_{2,n},p^{j}_{i,n},h_{n})-D\mathcal{A}^2_i( y_1, y_2, p_i^j, h))(\bar y_1,\bar y_2,\bar p_i^j,\bar h) \\
    %&\quad = F_1''(y_{1,n})\bar y_1 p_{1,n}^i + F_1'(y_{1,n})\bar p^i_1\\
	%&- F_1''(y_1)\bar y_1 p^i_1 - F_1'(y_1)\bar p^i_1 = X_i^1,\\
    &&\quad=(D_{21}^2 F_1(y_{1,n,}y_{2,n}) p^i_{1,n} - D_{21}^2 F_1(y_{1,n,}y_2) p^i_1)\bar y_1 + (D_{21}^2 F_2(y_{1,n},y_{2,n})p^i_{2,n} - D_{21}^2 F_2(y_1,y_2)p^i_2) \bar y_1 \\
    &&\quad\quad+ (D_{22}^2 F_1(y_{1,n},y_{2,n})p^i_{1,n} - D_{22}^2 F_1(y_1,y_2)p^i_1)\bar y_2 + (D_{22}^2 F_2(y_{1,n},y_{2,n})p^i_{2,n} - D_{22}^2 F_2(y_1,y_2) p^i_2) \bar y_2 \\
    &&\quad\quad+ (D_2 F_1(y_{1,n},y_{2,n}) - D_2 F_1(y_1,y_2)) \bar p^i_1 + (D_1 F_2(y_{1,n},y_{2,n}) - D_1 F_2(y_1,y_2))\bar p^i_2.
    %\\
    %&&\quad=X_{i,1}^1+X_{i,2}^1+X_{i,3}^1+X_{i,4}^1+X_{i,5}^1+X_{i,6}^1.
\end{eqnarray*}

Therefore, $( y_1, y_2, p_i^j, h)\longmapsto\mathcal{A}^{\prime}( y_1, y_2, p_i^j, h)$ is continuous from $\mathcal{Y}$ into $\mathcal{L}(\mathcal{Y},\mathcal{Z})$ and consequently, in view of some classical results, we will
have that $\mathcal{A}$ is Fr\'echet-differentiable and $\mathcal{C}^{1}$.
\end{proof}

\begin{lema}\label{Mapa sobrejetivo3}
Let $\mathcal{A}$ be the mapping in \eqref{aplicação A3}. Then, $\mathcal{A}^{\prime}(0,0,0,0,0,0,0)$ is onto.
\end{lema}
\begin{proof}
Let $(F_0, \overline{F}_0,F_1, \overline{F}_1,F_2, \overline{F}_2,y^{0}_1, y^{0}_2)\in \mathcal{Z}$. From Theorem \ref{theorem case linear3} we know there exists $y_1,y_2,p^{1}_1,p^{2}_1,p^{1}_2,p^{2}_2$ satisfying \eqref{optimal1} and \eqref{optimal2}. Furthermore, we know that $y_1,y_2,p^{1}_1,p^{2}_1,p^{1}_2,p^{2}_2\in C^{0}([0,T];L^{2}(0,1))\cap L^{2}(0,T;H^{1}_{a})$. Consequently, $(y,p^{1},p^{2},h)\in \mathcal{Y}$ and $$\mathcal{A}^{\prime}(0,0,0,0,0,0,0)(y_1,y_2,p^{1}_1,p^{2}_1,p^{1}_2,p^{2}_2,h)=(F_0, \overline{F}_0,F_1, \overline{F}_1,F_2, \overline{F}_2,y^{0}_1,y^0_2).$$ 
\hfill

This ends the proof.
\end{proof}
\noindent\textbf{Proof of Theorem \ref{thm:local_null_controllability4}.}

 According to Lemmas \ref{A bem definido3}-\ref{Mapa sobrejetivo3} we can apply the Inverse Mapping Theorem (Theorem \ref{Liusternik}) and consequently there exists $\delta > 0$ and a mapping $W:B_{\delta}(0)\subset {\mathcal{Z}}\rightarrow {\mathcal{Y}}$ such that
\begin{equation*}
W(z)\in B_{r}(0)\,\,\, \text{and}\,\,\, {\mathcal{A}}(W(z))=z, \,\,\, \forall z\in B_{\delta}(0).
\end{equation*}

Taking $(0,0,0,0,0,0,y_{1}^0,y_2^0)\in B_{\delta}(0)$ and $(y_1,y_2,p_1^{1},p_1^{2},p_2^{1},p_2^{2},h)=W(0,0,0,0,0,0,y_1^0,y_2^0)\in {\mathcal{Y}}$, we have
\begin{equation*}
{\mathcal{A}}(y_1,y_2,p^{1}_1,p^{2}_1,p^{1}_2,p^{2}_2,h)=(0,0,0,0,0,0,y^{0}_1,y^0_2).
\end{equation*}

Thus, we conclude that \eqref{sistemanocilin4} is locally null controllable at time $T > 0$.\\
\qed

\noindent\textbf{Acknowledgments}

This study was financed in part by the Coordenação de Aperfeiçoamento de Pessoal de Nível Superior - Brasil (CAPES) - Finance Code 001.
A.S. and L.Y were partially supported by CAPES-Brazil.
J.L. was partially supported by CNPq-Brazil.
The authors thank André Lopes and Laurent Prouvé for useful comments of a first draft of the paper. L.Y. thanks the kind hospitality of FAU DCN-AvH Chair for Dynamics, Control and Numerics-Germany and 
Instituto de Matemática Pura e Aplicada-Brazil, where significant parts of this work were done.

\appendix

\section{Proof of Proposition \ref{prop:carleman_sistema}.}\label{appendix_a}
\begin{proof}

    We rewrite the system of equations in order to present each equation in the appropriate form for using Proposition \ref{prop_carleman_13} and inequality \eqref{ecjv343}, with the notation $I(\cdot)$. We get
	\begin{eqnarray*}
	       I(\phi_1)	&\leq &  C\left(	\int_{0}^{T}\int_{0}^{1} e^{2s\varphi} |\rho_1\cara_{_{{\mathcal{O}}_{d}}}-c_{21}\phi_2+G_0|^2+(\lambda s)^3\int_{0}^{T}\int_{\mathcal{O}}e^{2s\varphi} \sigma^3  |\phi_1|^2  \right)\\
		I(\phi_2)	&\leq &  C\left(	\int_{0}^{T}\int_{0}^{1} e^{2s\varphi} |\rho_2\cara_{_{{\mathcal{O}}_{d}}}-c_{12}\phi_1+\overline{G}_0|^2+(\lambda s)^3\int_{0}^{T}\int_{\mathcal{O}}e^{2s\varphi} \sigma^3  |\phi_2|^2  \right)\\
		I(\rho_1)	&\leq &  C\left(	\int_{0}^{T}\int_{0}^{1} e^{2s\varphi} \left| \left( \frac{\alpha_1}{\mu_1}\cara_{_{{\mathcal{O}}_{1}}} +\frac{\alpha_2}{\mu_2}\cara_{_{{\mathcal{O}}_{2}}}  \right)    \phi_1-c_{12}\rho_2+G\right|^2+(\lambda s)^3\int_{0}^{T}\int_{\mathcal{O}}e^{2s\varphi} \sigma^3  |\rho_1|^2  \right)\\
		I(\rho_2)	&\leq &  C\left(	\int_{0}^{T}\int_{0}^{1} e^{2s\varphi} |-c_{21}\rho_1+\overline{G}|^2+(\lambda s)^3\int_{0}^{T}\int_{\mathcal{O}}e^{2s\varphi} \sigma^3  |\rho_2|^2  \right).
	\end{eqnarray*}
    Thus,
	\begin{multline*}
	   I(\phi_1)+I(\phi_2)+I(\rho_1)+I(\rho_2)\leq C\int_{0}^{T}\int_{0}^{1} e^{2s\varphi}(|G_0|^2+|\overline{G}_0|^2+|G|^2+|\overline{G}|^2)\\
	   \hspace*{5cm}+C\left(	\int_{0}^{T}\int_{0}^{1} e^{2s\varphi} \left[(K+\|c_{12}\|_{L^\infty})|\phi_1|^2+\|c_{21}\|_{L^\infty} |\phi_2|^2 \right.\right.\\
	   \hspace*{5cm}\left.+(1+\|c_{21}\|_{L^\infty})|\rho_1|^2+(1+\|c_{12}\|_{L^\infty})|\rho_2|^2\right]\\
	   \left.    +(\lambda s)^3\int_{0}^{T}\int_{\mathcal{O}}e^{2s\varphi} \sigma^3  \left(|\phi_1|^2+|\phi_2|^2+|\rho_1|^2+|\rho_2|^2  \right)  \right).
	\end{multline*}
	Now, we take $s$ and $\lambda$ sufficiently large such that $I(\phi_1)$ absorbs the integral depending on $\phi_1$, $I(\phi_2)$ absorbs the integral depending on $\phi_2$, $I(\rho_1)$ absorbs the integral depending on $\rho_1$ and $I(\rho_2)$ absorbs the integral depending on $\rho_2$. This give us the following inequality:
	\begin{equation}\label{eq:first_absorb}
        \begin{split}
            I(\phi_1)+I(\phi_2)+I(\rho_1)+I(\rho_2)\leq	C\int_{0}^{T}\int_{0}^{1} e^{2s\varphi}(|G_0|^2+|\overline{G}_0|^2+|G|^2+|\overline{G}|^2)\\
	        +C\left(	\int_{0}^{T}\int_{\mathcal{O}}e^{2s\varphi} \lambda^3 s^3\sigma^3  \left(|\phi_1|^2+|\phi_2|^2+|\rho_1|^2+|\rho_2|^2  \right)  \right) .       
        \end{split}   
	\end{equation}
    
    Let us take $\chi\in C_0^\infty(\mathcal{O})$ satisfying $0\leq \chi \leq 1$	 and $\chi =1$ in $ \mathcal{O}$. Since , we can easily see that	
    $$\int_{0}^{T}\int_{\widetilde{\mathcal{O}}}e^{2s\varphi}\lambda^3s^3\sigma^3|\rho_1|^2 \leq \int_{0}^{T}\int_{{\mathcal{O}}}\chi e^{2s\varphi}\lambda^3s^3\sigma^3|\rho_1|^2.  $$
	
    Now, multiplying the first equation in \eqref{optimaladj1_simp} by $e^{2s\varphi}\lambda^3s^3\sigma^3 \chi\rho_1$ and integrating over $(0,T)\times (0,1)$, we get
	\begin{multline}\label{eq:est_int_rho1_cuadrado}
		\int_{0}^{T}\int_{0}^{1}\chi e^{2s\varphi}\lambda^3s^3\sigma^3 |\rho_1|^2\cara_{_{{\mathcal{O}}_{d}}}=-\int_{0}^{T}\int_{0}^{1}\chi e^{2s\varphi}\lambda^3s^3\sigma^3 \rho_1 \phi_{1t}-\int_{0}^{T}\int_{0}^{1}\chi e^{2s\varphi}\lambda^3s^3\sigma^3b(t)\left({a}(x)\phi_{1x}\right)_x\rho_1\\
		\hspace*{5cm}+\int_{0}^{T}\int_{0}^{1}\chi e^{2s\varphi}\lambda^3s^3\sigma^3(B\sqrt{a}\phi_1)_{x}\rho_1+\int_{0}^{T}\int_{0}^{1}\chi e^{2s\varphi}\lambda^3s^3\sigma^3 c_{11}\phi_1\rho_1\\
		\hspace*{4cm}+\int_{0}^{T}\int_{0}^{1}\chi e^{2s\varphi}\lambda^3s^3\sigma^3 c_{21}\phi_2\rho_1-\int_{0}^{T}\int_{0}^{1}\chi e^{2s\varphi}\lambda^3s^3\sigma^3 G_0 \rho_1\\
	   =:I_1+I_2+I_3+I_4+I_5+I_6.
	\end{multline}
    
	Since $c_{11}$ and $c_{21}$ are bounded, from Young's inequality, we have the following estimates for $I_4, I_5$ and $I_6$:
    \begin{equation}\label{eq:Est_I4}
        \begin{split}
            I_4&\leq C \int_{0}^{T}\int_{0}^{1}\chi e^{2s\varphi}b\lambda^3s^3\sigma^3 |\phi_1||\rho_1|\leq \varepsilon \int_{0}^{T}\int_{0}^{1}e^{2s\varphi}\lambda^2s^2\sigma^2b^2 |\rho_1|^2+C_\varepsilon\int_{0}^{T}\int_{0}^{1}e^{2s\varphi}\chi^2\lambda^4s^4\sigma^4 |\phi_1|^2 \\
            &\leq \varepsilon I(\rho_1)+C_\varepsilon\int_{0}^{T}\int_{\mathcal{O}}e^{2s\varphi}\lambda^4s^4\sigma^4 |\phi_1|^2.
        \end{split}
    \end{equation}
    Analogously,
    \begin{equation}\label{eq:Est_I5}
        \begin{split}
	       I_5 &\leq C \int_{0}^{T}\int_{0}^{1}\chi e^{2s\varphi}\lambda^3s^3\sigma^3b |\phi_2||\rho_1|\leq \varepsilon I(\rho_1)+C_\varepsilon\int_{0}^{T}\int_{\mathcal{O}}e^{2s\varphi}\chi^2\lambda^4s^4\sigma^4 |\phi_2|^2,
        \end{split}
    \end{equation}
    \begin{equation}\label{eq:Est_I6}
        \begin{split}
            I_6 &\leq C \int_{0}^{T}\int_{0}^{1}\chi e^{2s\varphi}\lambda^3s^3\sigma^3b |\rho_1||G_0|\leq
            %\varepsilon \int_{0}^{T}\int_{0}^{1}e^{2s\varphi}\lambda^2s^2\sigma^2b^2 |\rho_1|^2 + C_\varepsilon\int_{0}^{T}\int_{0}^{1}e^{2s\varphi}\chi^2\lambda^4s^4\sigma^4 |G_0|^2
            \varepsilon I(\rho_1)+C_\varepsilon\int_{0}^{T}\int_{0}^{1}e^{2s\varphi}\lambda^4s^4\sigma^4 |G_0|^2.
        \end{split}
    \end{equation}
	
	Using integration by parts, we split up $I_1$ and $I_2$ in several integrals. Indeed,
    \begin{equation*}
        \begin{split}
            I_2&=-\int_{0}^{T}\int_{0}^{1}\chi e^{2s\varphi}\lambda^3s^3\sigma^3b(t)\left({a}(x)\phi_{1x}\right)_x\rho_1=\int_{0}^{T}\int_{0}^{1}b(t)\lambda^3s^3(\chi e^{2s\varphi}\sigma^3\rho_{1})_x a(x)\phi_{1x} \\
            &=\int_{0}^{T}\int_{0}^{1}b(t)\lambda^3s^3\chi e^{2s\varphi}\sigma^3\rho_{1x} a(x)\phi_{1x}+\int_{0}^{T}\int_{0}^{1}b(t)\lambda^3s^3(\chi e^{2s\varphi}\sigma^3)_x a(x)\rho_{1}\phi_{1x} \\
            &=\int_{0}^{T}\int_{0}^{1}b(t)\lambda^3s^3\chi e^{2s\varphi}\sigma^3\rho_{1x} a(x)\phi_{1x}-\int_{0}^{T}\int_{0}^{1}b(t)\lambda^3s^3(\chi e^{2s\varphi}\sigma^3)_x a(x)\rho_{1x}\phi_{1}\\
            &\quad -\int_{0}^{T}\int_{0}^{1} b(t)\lambda^3s^3(a(\chi e^{2s\varphi}\sigma^3)_x)_x	 \rho_{1}\phi_{1}.
        \end{split}
    \end{equation*}
	and, recalling that $e^{2s\varphi}$ vanishes at $0$ and $T$ and using the third equation of \eqref{optimaladj1_simp}, we have
    $$I_1=-\int_{0}^{T}\int_{0}^{1}\chi e^{2s\varphi}\lambda^3s^3\sigma^3 \rho_1 \phi_{1t}=\int_{0}^{T}\int_{0}^{1}\lambda^3s^3 \chi (e^{2s\varphi}\sigma^3\rho_{1})_t \phi_{1} $$
    \begin{eqnarray*}
		I_1&=&\int_{0}^{T}\int_{0}^{1}\lambda^3s^3 \chi (e^{2s\varphi}\sigma^3)_t \rho_{1}\phi_{1} +\int_{0}^{T}\int_{0}^{1}\lambda^3s^3 \chi e^{2s\varphi}\sigma^3\phi_{1}  \rho_{1t}
	   =\int_{0}^{T}\int_{0}^{1}\lambda^3s^3 \chi (e^{2s\varphi}\sigma^3)_t \rho_{1}\phi_{1} \\
	   &&+\int_{0}^{T}\int_{0}^{1}\lambda^3s^3 \chi e^{2s\varphi}\sigma^3\phi_{1} \left(b(t)\left(a(x)\rho_{1x}\right)_x+B\sqrt{a}\rho_{1x}-c_{11}\rho_1 - c_{12}\rho_2 -\rho_*^{-2}  d    \phi_1+G  \right) \\
	   &=&\int_{0}^{T}\int_{0}^{1}\left(\lambda^3s^3 \chi (e^{2s\varphi}\sigma^3)_t   - \lambda^3s^3 \chi e^{2s\varphi}\sigma^3 c_{11} \right) \rho_{1}\phi_{1}-\int_{0}^{T}\int_{0}^{1}\lambda^3s^3b (\chi e^{2s\varphi}\sigma^3)_x\phi_{1} a(x)\rho_{1x}\\
       &&- \int_{0}^{T}\int_{0}^{1}\left(\lambda^3s^3 \chi e^{2s\varphi}\sigma^3 c_{12} \right) \rho_{2}\phi_{1}\\
	   &&-\int_{0}^{T}\int_{0}^{1}\lambda^3s^3b \chi e^{2s\varphi}\sigma^3\phi_{1x} a(x)\rho_{1x}+\int_{0}^{T}\int_{0}^{1}\lambda^3s^3b \chi e^{2s\varphi}\sigma^3B\phi_1\sqrt{a}\rho_{1x}\\
	   && -\int_{0}^{T}\int_{0}^{1}\lambda^3s^3b \chi e^{2s\varphi}\sigma^3 d|\phi_1|^2+\int_{0}^{T}\int_{0}^{1}\lambda^3s^3b \chi e^{2s\varphi}\sigma^3 G \phi_1.
    \end{eqnarray*}	

    Thus,
    {\small \begin{eqnarray*}
        I_1+I_2+I_3&=&\int_{0}^{T}\int_{0}^{1}\left(\lambda^3s^3 \chi (e^{2s\varphi}\sigma^3)_t   - \lambda^3s^3 \chi e^{2s\varphi}\sigma^3 c_{11} \right) \rho_{1}\phi_{1}-\int_{0}^{T}\int_{0}^{1}\lambda^3s^3b (\chi e^{2s\varphi}\sigma^3)_x\phi_{1} a(x)\rho_{1x}\\
        &&- \int_{0}^{T}\int_{0}^{1}\left(\lambda^3s^3 \chi e^{2s\varphi}\sigma^3 c_{12} \right) \rho_{2}\phi_{1}\\
	   &&-\int_{0}^{T}\int_{0}^{1}\lambda^3s^3b \chi e^{2s\varphi}\sigma^3\phi_{1x} a(x)\rho_{1x}+\int_{0}^{T}\int_{0}^{1}\lambda^3s^3b \chi e^{2s\varphi}\sigma^3B\phi_1\sqrt{a}\rho_{1x}\\
        && -\int_{0}^{T}\int_{0}^{1}\lambda^3s^3b \chi e^{2s\varphi}\sigma^3 d|\phi_1|^2+\int_{0}^{T}\int_{0}^{1}\lambda^3s^3b \chi e^{2s\varphi}\sigma^3 G \phi_1\\
        &&+\int_{0}^{T}\int_{0}^{1}b(t)\lambda^3s^3\chi e^{2s\varphi}\sigma^3\rho_{1x} a(x)\phi_{1x}-\int_{0}^{T}\int_{0}^{1}b(t)\lambda^3s^3(\chi e^{2s\varphi}\sigma^3)_x a(x)\rho_{1x}\phi_{1}\\
        &&-\int_{0}^{T}\int_{0}^{1}b(t)\lambda^3s^3(a(\chi e^{2s\varphi}\sigma^3)_x)_x	 \rho_{1}\phi_{1}+\int_{0}^{T}\int_{0}^{1}\chi e^{2s\varphi}\lambda^3s^3\sigma^3(B\sqrt{a}\phi_1)_{x}\rho_1.
    \end{eqnarray*}}
    After rewriting,
    {\small \begin{eqnarray}\label{ecrho1}
		I_1+I_2+I_3&=&\int_{0}^{T}\int_{0}^{1}\left[\lambda^3s^3 \chi (e^{2s\varphi}\sigma^3)_t   - \lambda^3s^3 \chi e^{2s\varphi}\sigma^3 c_{11} -b(t)\lambda^3s^3(a(\chi e^{2s\varphi}\sigma^3)_x)_x	\right] \rho_{1}\phi_{1}\nonumber\\
		&&-2\int_{0}^{T}\int_{0}^{1}\lambda^3s^3b (\chi e^{2s\varphi}\sigma^3)_x\phi_{1} a(x)\rho_{1x}+\int_{0}^{T}\int_{0}^{1}\chi e^{2s\varphi}\lambda^3s^3\sigma^3(B\sqrt{a}\phi_1)_{x}\rho_1\nonumber\\
		&&+\int_{0}^{T}\int_{0}^{1}\lambda^3s^3b \chi e^{2s\varphi}\sigma^3B\phi_1\sqrt{a}\rho_{1x} -\int_{0}^{T}\int_{0}^{1}\lambda^3s^3b \chi e^{2s\varphi}\sigma^3 d|\phi_1|^2\nonumber\\
		&&+\int_{0}^{T}\int_{0}^{1}\lambda^3s^3b \chi e^{2s\varphi}\sigma^3 |G||\phi_1| - \int_{0}^{T}\int_{0}^{1}\left(\lambda^3s^3 \chi e^{2s\varphi}\sigma^3 c_{12} \right) \rho_{2}\phi_{1}\\\nonumber\\
		&=&J_1+J_2+J_3+J_4+J_5+J_6+J_7.
    \end{eqnarray}}

    In order to estimate $J_1$, we analyze each term between the brackets. Firstly, we observe that all the terms are multiplied by $\chi$, which vanishes outside of $\mathcal{O}$. Clearly,
    $$\left|-\lambda^3s^3 \chi e^{2s\varphi}\sigma^3 c_{11} \right|\leq C\chi\lambda^5s^5  \sigma^5e^{2s\varphi}.$$
    
    Since $|\sigma_x|\leq C\lambda\sigma$, $|\sigma_{xx}|\leq C\lambda^2\sigma$ and $a\in C^1(\mathcal{O})$, after distributing the derivatives with respect to $x$, we see that
    $$\left|  -b(t)\lambda^3s^3(a(\chi e^{2s\varphi}\sigma^3)_x)_x \right|\leq C \chi s^5\lambda^5\sigma^5e^{2s\varphi}.$$
    
    Likewise, the relations $|\varphi_t|\leq C\sigma^2$ and $|\sigma_t|\leq C\sigma^2$ yield
    $$\left|\lambda^3s^3 \chi (e^{2s\varphi}\sigma^3)_t  \right|\leq  C \chi s^5\lambda^5\sigma^5e^{2s\varphi}.$$
    
    As a conclusion, from Young's inequality, we have
    $$J_1\leq C \int_{0}^{T}\int_{0}^{1}\chi s^5\lambda^5\sigma^5be^{2s\varphi} |\rho_1||\phi_1|\leq \varepsilon \int_{0}^{T}\int_{0}^{1}s^2\lambda^2\sigma^2b^2e^{2s\varphi}|\rho_1|^2+C_{\varepsilon} \int_{0}^{T}\int_{0}^{1}\chi^2s^8\lambda^8\sigma^8e^{2s\varphi}|\phi_1|^2.$$
    That is,
    $$J_1\leq \varepsilon I(\rho_1)+C_{\varepsilon} \int_{0}^{T}\int_{\mathcal{O}}s^8\lambda^8\sigma^8e^{2s\varphi}|\phi_1|^2.$$
    
    Estimate for $J_2$: First, we have the estimate 
    $$(\chi e^{2s\varphi}\sigma^3)_x=\chi(2s\varphi_x\sigma^3+3\sigma^2\sigma_x)e^{2s\varphi}\leq \chi(2s\lambda \psi_x \sigma^4+3\sigma^2\lambda \sigma)e^{2s\varphi}\leq C \chi s\lambda \sigma^4e^{2s\varphi}.$$
    
    Thus,
    \begin{equation*}
        \begin{split}
            J_2 &\leq Cs\lambda\int_{0}^{T}\int_{0}^{1}e^{s\varphi}b^{1/2}\sigma^{1/2}\sqrt{a}\rho_{1x} e^{s\varphi}\chi\lambda^3s^3\sigma^{7/2}b^{-1/2}\sqrt{a}\phi_{1} \\
            &\leq \varepsilon \int_{0}^{T}\int_{0}^{1}e^{2s\varphi}s\lambda b\sigma {a}|\rho_{1x}|^2+C_\varepsilon \int_{0}^{T}\int_{0}^{1}\chi^2e^{2s\varphi}\lambda^7s^7 \sigma^{7}b^{-1}{a}|\phi_{1}|^2.
        \end{split}
    \end{equation*}
    Using that $a\in C^1(\mathcal{O})$ and $b(t)\geq m$, we get
    $$J_2\leq \varepsilon I(\rho_1)+C_\varepsilon \int_{0}^{T}\int_{\mathcal{O}}\lambda^8s^8\sigma^8 e^{2s\varphi} |\phi_{1}|^2.$$
    
    Estimate for $J_3$:
    \begin{eqnarray*}
	   J_3&=& =\int_{0}^{T}\int_{0}^{1}\chi e^{2s\varphi}\lambda^3s^3\sigma^3(B\sqrt{a}\phi_1)_{x}\rho_1=-\int_{0}^{T}\int_{0}^{1}\lambda^3s^3(\chi e^{2s\varphi}\sigma^3\rho_1)_xB\sqrt{a}\phi_1 \\
       &=&-\int_{0}^{T}\int_{0}^{1}\lambda^3s^3(\chi e^{2s\varphi}\sigma^3)_xB\sqrt{a}\phi_1\rho_1-\int_{0}^{T}\int_{0}^{1}\lambda^3s^3\chi e^{2s\varphi}\sigma^3B\sqrt{a}\phi_1\rho_{1x}\\
	   &=&\mathcal{J}_1+\mathcal{J}_2.
    \end{eqnarray*}
    
    For $\mathcal{J}_1$ we have: 
    \begin{eqnarray*}
	   \mathcal{J}_1&\leq & \int_{0}^{T}\int_{0}^{1}\left(e^{s\varphi} \lambda s b \sigma|\rho_1| \right)\left( \chi e^{s\varphi}\lambda^3s^3 \sigma^{3}b^{-1}B\sqrt{a}|\phi_1|  \right)\\
	   &\leq&\varepsilon \int_{0}^{T}\int_{0}^{1}e^{2s\varphi}\lambda^2 s^2 b^2 \sigma^2 |\rho_1|^2+C_\varepsilon  \int_{0}^{T}\int_{0}^{1}\chi^2 e^{2s\varphi}\lambda^6s^6 \sigma^6 b^{-2}B^2 {a}|\phi_1|^2 \\
	   &\leq&\varepsilon I(\rho_1) +C_\varepsilon \int_{0}^{T}\int_{\mathcal{O}}\lambda^8s^8\sigma^8 e^{2s\varphi} |\phi_{1}|^2.
    \end{eqnarray*}
    Analogously,
    \begin{eqnarray*}
	   \mathcal{J}_2&=&\lambda s \int_{0}^{T}\int_{0}^{1}\left( e^{s\varphi}  b \sigma^{1/2} \sqrt{a}|\rho_{1x}| \right)\left( e^{s\varphi} \lambda^2s^2 \sigma^{5/2}b^{-1}B |\phi_1|  \right)\\
	   &\leq&\varepsilon \int_{0}^{T}\int_{0}^{1}e^{2s\varphi}\lambda s b^2 \sigma a|\rho_{1x}|^2+C_\varepsilon  \int_{0}^{T}\int_{0}^{1}\chi^2e^{2s\varphi}\lambda^5s^5 \sigma^{5}b^{-2}B^2|\phi_1|^2 \\
	   &\leq&\varepsilon I(\rho_1) +C_\varepsilon \int_{0}^{T}\int_{\mathcal{O}}\lambda^8s^8\sigma^8 e^{2s\varphi} |\phi_{1}|^2. 
    \end{eqnarray*}

    Estimate for $J_4$:
    \begin{eqnarray*}
	   J_4 &=&\lambda s\int_{0}^{T}\int_{0}^{1}\left( e^{s\varphi} \sigma^{1/2}b\sqrt{a}\rho_{1x} \right)\left( \lambda^2s^2  \chi e^{s\varphi}\sigma^{5/2}B\phi_1\right)\\
	   &\leq&\varepsilon \int_{0}^{T}\int_{0}^{1}e^{2s\varphi}\lambda s b^2 \sigma a|\rho_{1x}|^2+C_\varepsilon  \int_{0}^{T}\int_{0}^{1}\chi^2e^{2s\varphi}\lambda^5s^5 \sigma^{5}B^2|\phi_1|^2 \\
	   &\leq&\varepsilon I(\rho_1) +C_\varepsilon \int_{0}^{T}\int_{\mathcal{O}}\lambda^8s^8\sigma^8 e^{2s\varphi} |\phi_{1}|^2. 
    \end{eqnarray*}

    It is immediate that
    $$J_5\leq  C \int_{0}^{T}\int_{\mathcal{O}}\lambda^8s^8\sigma^8 e^{2s\varphi} |\phi_{1}|^2.$$

    Moreover,
    \begin{equation*}
        \begin{split}
            J_6&\leq C \int_{0}^{T}\int_{0}^{1}\chi e^{2s\varphi}\lambda^3s^3\sigma^3b |\phi_1||G|\leq \varepsilon \int_{0}^{T}\int_{0}^{1}e^{2s\varphi}\lambda^2s^2\sigma^2b^2 |\phi_1|^2+C_\varepsilon \int_{0}^{T}\int_{0}^{1}e^{2s\varphi}\chi^2\lambda^4s^4\sigma^4 |G|^2 \\
            &\leq \varepsilon I(\phi_1)+C_\varepsilon\int_{0}^{T}\int_{0}^{1}e^{2s\varphi}\lambda^4s^4\sigma^4 |G|^2,
        \end{split}
    \end{equation*}
    and
    \begin{equation*}
        \begin{split}
            J_7&\leq C \int_{0}^{T}\int_{0}^{1}\left(\lambda^3s^3 \chi e^{2s\varphi}\sigma^3 \right) \rho_{2}\phi_{1} \leq \varepsilon \int_{0}^{T}\int_{0}^{1}e^{2s\varphi}\lambda^2s^2\sigma^2 |\rho_2|^2+C_\varepsilon \int_{0}^{T}\int_{0}^{1}e^{2s\varphi}\chi^2\lambda^4s^4\sigma^4 |\phi_1|^2 \\
            &\leq \varepsilon I(\rho_2)+C_\varepsilon\int_{0}^{T}\int_{\mathcal{O}}e^{2s\varphi}\lambda^4 s^4\sigma^4 |\phi_1|^2
        \end{split}
    \end{equation*}

    This finished the estimates of $I_1 + I_2 + I_3$ in \eqref{ecrho1}. This, jointly with \eqref{eq:Est_I4}-\eqref{eq:Est_I6} finished the estimate of \eqref{eq:est_int_rho1_cuadrado}.

    Therefore, the integral where $\rho_1$ appears in \eqref{eq:first_absorb} is estimated as:
    \begin{equation} \label{ecintrho1}
        \begin{split}
            \int_{0}^{T}\int_{0}^{1}\chi e^{2s\varphi}\lambda^3s^3\sigma^3 |\rho_1|^2\cara_{_{{\mathcal{O}}_{d}}}\leq& \varepsilon I(\rho_1) + \varepsilon I(\rho_2) + \varepsilon I(\phi_1)+C_\varepsilon\int_{0}^{T}\int_{0}^{1}e^{2s\varphi}\lambda^4s^4\sigma^4 \left(|G_0|^2+|G|^2\right)\\
            &+C_\varepsilon \int_{0}^{T}\int_{\overline{\mathcal{O}}}\lambda^4s^4\sigma^4 e^{2s\varphi} \left(|\phi_{1}|^2+|\phi_{2}|^2\right)
            +C_\varepsilon \int_{0}^{T}\int_{\overline{\mathcal{O}}}\lambda^8s^8\sigma^8 e^{2s\varphi} |\phi_{1}|^2.    
        \end{split}
    \end{equation}

    To estimate the integral with $|\rho_{2}|^2$ in \eqref{eq:first_absorb} we proceed as follows:
    
    Multiplying the second equation in \eqref{optimaladj1_simp} by $ e^{2s\varphi}\lambda^3s^3\sigma^3 \chi\rho_2$ and integrating over $(0,T)\times (0,1)$, we get
    \begin{equation*}
        \begin{split}
            \int_{0}^{T}\int_{0}^{1}\chi e^{2s\varphi}\lambda^3s^3\sigma^3 |\rho_2|^2\cara_{_{{\mathcal{O}}_{d}}}&=-\int_{0}^{T}\int_{0}^{1}\chi e^{2s\varphi}\lambda^3s^3\sigma^3 \rho_2 \phi_{2t}-\int_{0}^{T}\int_{0}^{1}\chi e^{2s\varphi}\lambda^3s^3\sigma^3 b\left(a(x) \phi_{2x}\right)_x\rho_2\\
            &\qquad +\int_{0}^{T}\int_{0}^{1}\chi e^{2s\varphi}\lambda^3s^3\sigma^3(B\sqrt{a}\phi_2)_{x}\rho_2 +\int_{0}^{T}\int_{0}^{1}\chi e^{2s\varphi}\lambda^3s^3\sigma^3 c_{22}\phi_2\rho_2\\
            &\qquad + \int_{0}^{T}\int_{0}^{1}\chi e^{2s\varphi}\lambda^3s^3\sigma^3 c_{12}\phi_1\rho_2 - \int_{0}^{T}\int_{0}^{1}\chi e^{2s\varphi}\lambda^3s^3\sigma^3 \overline{G}_0\rho_2\\
            &=\mathcal{I}_1+\mathcal{I}_2+\mathcal{I}_3+\mathcal{I}_4+\mathcal{I}_5 + \mathcal{I}_6.        
        \end{split}
    \end{equation*}

    In the same way, since $c_{12}$ and $c_{22}$ are bounded, from Young’s inequality, we have the estimates
    \begin{equation*}
        \begin{split}
            \mathcal{I}_4 &\leq C \int_{0}^{T}\int_{0}^{1}\chi e^{2s\varphi}b\lambda^3s^3\sigma^3 |\phi_2||\rho_2|\leq \varepsilon \int_{0}^{T}\int_{0}^{1}e^{2s\varphi}\lambda^2s^2\sigma^2b^2 |\rho_2|^2+C_\varepsilon \int_{0}^{T}\int_{0}^{1}e^{2s\varphi}\chi^2\lambda^4s^4\sigma^4 |\phi_2|^2 \\
            &\leq \varepsilon I(\rho_2)+C_\varepsilon\int_{0}^{T}\int_{\mathcal{O}}e^{2s\varphi}\lambda^4s^4\sigma^4 |\phi_2|^2,
        \end{split}
    \end{equation*}
    Moreover,
    \begin{equation*}
        \begin{split}
            \mathcal{I}_5 &\leq C \int_{0}^{T}\int_{0}^{1}\chi e^{2s\varphi}b\lambda^3s^3\sigma^3 |\phi_1||\rho_2|\leq \varepsilon \int_{0}^{T}\int_{0}^{1}e^{2s\varphi}\lambda^2s^2\sigma^2b^2 |\rho_2|^2+C_\varepsilon \int_{0}^{T}\int_{0}^{1}e^{2s\varphi}\chi^2\lambda^4s^4\sigma^4 |\phi_1|^2 \\
            &\leq \varepsilon I(\rho_2)+C_\varepsilon\int_{0}^{T}\int_{\mathcal{O}}e^{2s\varphi}\lambda^4s^4\sigma^4 |\phi_1|^2,
        \end{split}
    \end{equation*}
    and
    \begin{equation*}
        \begin{split}
            \mathcal{I}_6&\leq C \int_{0}^{T}\int_{0}^{1}\chi e^{2s\varphi}b\lambda^3s^3\sigma^3 |\overline{G}_0||\rho_2|\leq \varepsilon \int_{0}^{T}\int_{0}^{1}e^{2s\varphi}\lambda^2s^2\sigma^2b^2 |\rho_2|^2+ C_\varepsilon \int_{0}^{T}\int_{0}^{1}e^{2s\varphi}\chi^2\lambda^4s^4\sigma^4 |\overline{G}_0|^2 \\
            &\leq \varepsilon I(\rho_2)+C_\varepsilon\int_{0}^{T}\int_{\mathcal{O}}e^{2s\varphi}\lambda^4s^4\sigma^4 |\overline{G}_0|^2.
        \end{split}
    \end{equation*}
      
    Using integration by parts, we split up $\mathcal{I}_1$ and $\mathcal{I}_2$ in several integrals. Indeed,
    \begin{equation*}
        \begin{split}
            \mathcal{I}_2&=-\int_{0}^{T}\int_{0}^{1}\chi e^{2s\varphi}\lambda^3s^3\sigma^3b(t)\left({a}(x)\phi_{2x}\right)_x\rho_2=\int_{0}^{T}\int_{0}^{1}b(t)\lambda^3s^3(\chi e^{2s\varphi}\sigma^3\rho_{2})_x a(x)\phi_{2x} \\
            &=\int_{0}^{T}\int_{0}^{1}b(t)\lambda^3s^3\chi e^{2s\varphi}\sigma^3\rho_{2x} a(x)\phi_{2x}+\int_{0}^{T}\int_{0}^{1}b(t)\lambda^3s^3(\chi e^{2s\varphi}\sigma^3)_x a(x)\rho_{2}\phi_{2x} \\
            &=\int_{0}^{T}\int_{0}^{1}b(t)\lambda^3s^3\chi e^{2s\varphi}\sigma^3\rho_{2x} a(x)\phi_{2x}-\int_{0}^{T}\int_{0}^{1}b(t)\lambda^3s^3(\chi e^{2s\varphi}\sigma^3)_x a(x)\rho_{2x}\phi_{2}\\
		  &\quad -\int_{0}^{T}\int_{0}^{1}b(t)\lambda^3s^3(a(\chi e^{2s\varphi}\sigma^3)_x)_x	 \rho_{2}\phi_{2}
        \end{split}
    \end{equation*}

%	$$\mathcal{I}_2=-\int_{0}^{T}\int_{0}^{1}\chi e^{2s\varphi}\lambda^3s^3\sigma^3b(t)\left({a}(x)\phi_{2x}\right)_x\rho_2=\int_{0}^{T}\int_{0}^{1}b(t)\lambda^3s^3(\chi e^{2s\varphi}\sigma^3\rho_{2})_x a(x)\phi_{2x}$$
		
%    $$=\int_{0}^{T}\int_{0}^{1}b(t)\lambda^3s^3\chi e^{2s\varphi}\sigma^3\rho_{2x} a(x)\phi_{2x}+\int_{0}^{T}\int_{0}^{1}b(t)\lambda^3s^3(\chi e^{2s\varphi}\sigma^3)_x a(x)\rho_{2}\phi_{2x}$$
	
%    \begin{multline*}
%		=\int_{0}^{T}\int_{0}^{1}b(t)\lambda^3s^3\chi e^{2s\varphi}\sigma^3\rho_{2x} a(x)\phi_{2x}-\int_{0}^{T}\int_{0}^{1}b(t)\lambda^3s^3(\chi e^{2s\varphi}\sigma^3)_x a(x)\rho_{2x}\phi_{2}\\
%		-\int_{0}^{T}\int_{0}^{1}b(t)\lambda^3s^3(a(\chi e^{2s\varphi}\sigma^3)_x)_x	 \rho_{2}\phi_{2}
%	\end{multline*}
	and, recalling that $e^{2s\varphi}$ vanishes at $0$ and $T$, 
		$$\mathcal{I}_1=-\int_{0}^{T}\int_{0}^{1}\chi e^{2s\varphi}\lambda^3s^3\sigma^3 \rho_2 \phi_{1t}=\int_{0}^{T}\int_{0}^{1}\lambda^3s^3 \chi (e^{2s\varphi}\sigma^3\rho_{2})_t \phi_{1}, $$
    and, using the fourth equation of \eqref{optimaladj1_simp}, we have
		\begin{eqnarray*}
			\mathcal{I}_1&=&\int_{0}^{T}\int_{0}^{1}\lambda^3s^3 \chi (e^{2s\varphi}\sigma^3)_t \rho_{2}\phi_{2} +\int_{0}^{T}\int_{0}^{1}\lambda^3s^3 \chi e^{2s\varphi}\sigma^3\phi_{2}  \rho_{2t}
			=\int_{0}^{T}\int_{0}^{1}\lambda^3s^3 \chi (e^{2s\varphi}\sigma^3)_t \rho_{2}\phi_{2} \\
			&&+\int_{0}^{T}\int_{0}^{1}\lambda^3s^3 \chi e^{2s\varphi}\sigma^3\phi_{2} \left(b(t)\left(a(x)\rho_{2x}\right)_x+B\sqrt{a}\rho_{2x}-c_{21}\rho_1-c_{22}\rho_2  +\overline{G} \,  \right) \\
			&=&\int_{0}^{T}\int_{0}^{1}\left(\lambda^3s^3 \chi (e^{2s\varphi}\sigma^3)_t   - \lambda^3s^3 \chi e^{2s\varphi}\sigma^3 c_{22} \right) \rho_{2}\phi_{2}-\int_{0}^{T}\int_{0}^{1}\lambda^3s^3b (\chi e^{2s\varphi}\sigma^3)_x\phi_{2} a(x)\rho_{2x}\\
			&&-\int_{0}^{T}\int_{0}^{1}\lambda^3s^3b \chi e^{2s\varphi}\sigma^3\phi_{2x} a(x)\rho_{2x}+\int_{0}^{T}\int_{0}^{1}\lambda^3s^3b \chi e^{2s\varphi}\sigma^3B\phi_2\sqrt{a}\rho_{2x}\\
			&& -\int_{0}^{T}\int_{0}^{1}\lambda^3s^3b \chi e^{2s\varphi}\sigma^3 c_{21} \rho_2\phi_2+\int_{0}^{T}\int_{0}^{1}\lambda^3s^3b \chi e^{2s\varphi}\sigma^3 \overline{G}\phi_2
		\end{eqnarray*}

		Thus,
		{\small \begin{eqnarray*}
				\mathcal{I}_1+\mathcal{I}_2+\mathcal{I}_3&=&\int_{0}^{T}\int_{0}^{1}\left(\lambda^3s^3 \chi (e^{2s\varphi}\sigma^3)_t   - \lambda^3s^3 \chi e^{2s\varphi}\sigma^3 c_{22} \right) \rho_{2}\phi_{2}-\int_{0}^{T}\int_{0}^{1}\lambda^3s^3b (\chi e^{2s\varphi}\sigma^3)_x\phi_{2} a(x)\rho_{2x}\\
				&&-\int_{0}^{T}\int_{0}^{1}\lambda^3s^3b \chi e^{2s\varphi}\sigma^3\phi_{2x} a(x)\rho_{2x}+\int_{0}^{T}\int_{0}^{1}\lambda^3s^3b \chi e^{2s\varphi}\sigma^3B\phi_2\sqrt{a}\rho_{2x}\\
				&& -\int_{0}^{T}\int_{0}^{1}\lambda^3s^3b \chi e^{2s\varphi}\sigma^3 c_{21} \rho_1\phi_2+\int_{0}^{T}\int_{0}^{1}\lambda^3s^3b \chi e^{2s\varphi}\sigma^3 \overline{G}\phi_2\\
				&&+\int_{0}^{T}\int_{0}^{1}b(t)\lambda^3s^3\chi e^{2s\varphi}\sigma^3\rho_{2x} a(x)\phi_{2x}-\int_{0}^{T}\int_{0}^{1}b(t)\lambda^3s^3(\chi e^{2s\varphi}\sigma^3)_x a(x)\rho_{2x}\phi_{2}\\
				&&-\int_{0}^{T}\int_{0}^{1}b(t)\lambda^3s^3(a(\chi e^{2s\varphi}\sigma^3)_x)_x	 \rho_{2}\phi_{2}+\int_{0}^{T}\int_{0}^{1}\chi e^{2s\varphi}\lambda^3s^3\sigma^3(B\sqrt{a}\phi_2)_{x}\rho_2,
		\end{eqnarray*}}
    and, after simplification, we get
		{\small \begin{eqnarray*}
				\mathcal{I}_1+\mathcal{I}_2+\mathcal{I}_3&=&\int_{0}^{T}\int_{0}^{1}\left(\lambda^3s^3 \chi (e^{2s\varphi}\sigma^3)_t   - \lambda^3s^3 \chi e^{2s\varphi}\sigma^3 c_{22} -b(t)\lambda^3s^3(a(\chi e^{2s\varphi}\sigma^3)_x)_x	\right) \rho_{2}\phi_{2}\\
				&&-2\int_{0}^{T}\int_{0}^{1}\lambda^3s^3b (\chi e^{2s\varphi}\sigma^3)_x\phi_{2} a(x)\rho_{2x}+\int_{0}^{T}\int_{0}^{1}\chi e^{2s\varphi}\lambda^3s^3\sigma^3(B\sqrt{a}\phi_2)_{x}\rho_2\\
				&&+\int_{0}^{T}\int_{0}^{1}\lambda^3s^3b \chi e^{2s\varphi}\sigma^3B\phi_2\sqrt{a}\rho_{2x} -\int_{0}^{T}\int_{0}^{1}\lambda^3s^3b \chi e^{2s\varphi}\sigma^3 c_{21} \rho_1\phi_2\\
				&&+\int_{0}^{T}\int_{0}^{1}\lambda^3s^3b \chi e^{2s\varphi}\sigma^3 \overline{G}\phi_2
		\end{eqnarray*}}
		We see that all the integrals above are of the same type as those in (\ref{ecrho1}). 
		Therefore, the integral where $\rho_2$ appears in \eqref{eq:first_absorb} is estimated similarly to (\ref{ecintrho1}), 
		\begin{multline}\label{ecintrho2}
			\int_{0}^{T}\int_{0}^{1}\chi e^{2s\varphi}\lambda^3s^3\sigma^3 |\rho_2|^2\cara_{_{{\mathcal{O}}_{d}}}\leq \varepsilon I(\rho_2)+\varepsilon I(\phi_2)+C_\varepsilon\int_{0}^{T}\int_{0}^{1}e^{2s\varphi}\lambda^4s^4\sigma^4 \left(|\overline{G}_0|^2+|\overline{G}|^2\right)\\
			C \int_{0}^{T}\int_{\overline{\mathcal{O}}}\lambda^4s^4\sigma^4 e^{2s\varphi} \left(|\phi_{1}|^2+|\phi_{2}|^2\right)
			+C\int_{0}^{T}\int_{\overline{\mathcal{O}}}\lambda^8s^8\sigma^8 e^{2s\varphi} |\phi_{2}|^2.
		\end{multline}

	%Following the same arguments developed, we have the partial result of \eqref{ecintrho1} and \eqref{ecintrho2}.
    Substituting the estimates \eqref{ecintrho1} and \eqref{ecintrho2} in \eqref{eq:first_absorb}, absorbing the terms with $\varepsilon$, leads to the following intermediate estimate:        
	\begin{multline}\label{carlemanintermedio}
        I(\phi_1)+I(\phi_2)+I(\rho_1)+I(\rho_2)\leq C\int_{0}^{T}\int_{0}^{1}e^{2s\varphi}\lambda^4s^4\sigma^4 \left(|{G}_0|^2+|{G}|^2|+\overline{G}_0|^2+|\overline{G}|^2\right)\\
        + C \int_{0}^{T}\int_{\overline{\mathcal{O}}}\lambda^8s^8\sigma^8 e^{2s\varphi} \left(|\phi_{1}|^2+|\phi_{2}|^2 \right).
	\end{multline}
    
	Now, we want to eliminate the local integral term of $\phi_2$ on the right hand side of \eqref{carlemanintermedio}.
	We take a nonempty open set $\overline{\mathcal{O}}$ such that $\widetilde{\mathcal{O}} \subset\overline{\mathcal{O}} \subset \mathcal{O}_d \cap\mathcal{O}$.\\
		
	Let us take $\chi\in C_0^\infty(\mathcal{O})$ satisfying $0\leq \chi \leq 1$	 and $\chi =1$ in $ \widetilde{\mathcal{O}}$. We easily
	see that	
    $$\int_{0}^{T}\int_{\widetilde{\mathcal{O}}}e^{2s\varphi}\lambda^8s^8\sigma^8|\phi_2|^2 \leq \int_{0}^{T}\int_{{\mathcal{O}}}\chi e^{2s\varphi}\lambda^8s^8\sigma^8|\phi_2|^2.$$
    
	Now, multiplying the first equation in \eqref{optimaladj1_simp} by $ e^{2s\varphi}\lambda^8s^8\sigma^8 \chi\phi_2$ and integrating over $(0,T)\times (0,1)$, we get	
			
	\begin{eqnarray}
		\int_{0}^{T}\int_{0}^{1}\chi e^{2s\varphi}\lambda^8s^8\sigma^8 c_{21} |\phi_2|^2 &=& \int_{0}^{T}\int_{0}^{1}\chi e^{2s\varphi}\lambda^8s^8\sigma^8 \phi_2 \phi_{1t}+\int_{0}^{T}\int_{0}^{1}\chi e^{2s\varphi}\lambda^8s^8\sigma^8b(t)\left({a}(x)\phi_{1x}\right)_x\phi_2 \nonumber\\
		&&-\int_{0}^{T}\int_{0}^{1}\chi e^{2s\varphi}\lambda^8s^8\sigma^8(B\sqrt{a}\phi_1)_{x}\phi_2-\int_{0}^{T}\int_{0}^{1}\chi e^{2s\varphi}\lambda^8s^8\sigma^8 c_{11}\phi_1\phi_2 \nonumber \\
		&&+\int_{0}^{T}\int_{0}^{1}\chi e^{2s\varphi}\lambda^8s^8\sigma^8 \phi_2\rho_1\cara_{\mathcal{O}_{d}}+\int_{0}^{T}\int_{0}^{1}\chi e^{2s\varphi}\lambda^8s^8\sigma^8 \phi_2G_0 \nonumber \\
        &=&\varUpsilon_1+\varUpsilon_2+\varUpsilon_3+\varUpsilon_4+\varUpsilon_5+\varUpsilon_6. \label{eq:Y1-Y6}
	\end{eqnarray}
		
%	In the same way, since $c_{11}, c_{21}$ is bounded, from Young's inequality, we have
%	$$I_4\leq C \int_{0}^{T}\int_{0}^{1}\chi e^{2s\varphi}b\lambda^3s^3\sigma^3 |\phi_1||\rho_1|\leq \varepsilon \int_{0}^{T}\int_{0}^{1}e^{2s\varphi}\lambda^2s^2\sigma^2b^2 |\rho_1|^2+\int_{0}^{T}\int_{0}^{1}e^{2s\varphi}\chi^2\lambda^4s^4\sigma^4 |\phi_1|^2$$
%	$$I_4\leq \varepsilon I(\rho_1)+C_\varepsilon\int_{0}^{T}\int_{\mathcal{O}}e^{2s\varphi}\lambda^4s^4\sigma^4 |\phi_1|^2$$
%	$$I_5\leq C \int_{0}^{T}\int_{0}^{1}\chi e^{2s\varphi}\lambda^3s^3\sigma^3b |\phi_2||\rho_1|\leq \varepsilon \int_{0}^{T}\int_{0}^{1}e^{2s\varphi}\lambda^2s^2\sigma^2b^2 |\rho_1|^2+\int_{0}^{T}\int_{0}^{1}e^{2s\varphi}\chi^2\lambda^4s^4\sigma^4 |\phi_2|^2$$
%	$$I_5\leq \varepsilon I(\rho_1)+C_\varepsilon\int_{0}^{T}\int_{\mathcal{O}}e^{2s\varphi}\chi^2\lambda^4s^4\sigma^4 |\phi_2|^2$$

    Now we estimate each term of \eqref{eq:Y1-Y6}. The result is given in \eqref{ecphi2}.
    
	Using integration by parts, we split up $\varUpsilon_1$ and $\varUpsilon_2$ in several integrals. In fact
    \begin{equation*}
        \begin{split}
            \varUpsilon_2 &= \int_{0}^{T}\int_{0}^{1}\chi e^{2s\varphi}\lambda^8s^8\sigma^8b(t)\left({a}(x)\phi_{1x}\right)_x\phi_2=-\int_{0}^{T}\int_{0}^{1}b(t)\lambda^8s^8(\chi e^{2s\varphi}\sigma^3\phi_{2})_x a(x)\phi_{1x}\\
            &=-\int_{0}^{T}\int_{0}^{1}b(t)\lambda^8s^8\chi e^{2s\varphi}\sigma^8\phi_{2x} a(x)\phi_{1x}-\int_{0}^{T}\int_{0}^{1}b(t)\lambda^8s^8(\chi e^{2s\varphi}\sigma^8)_x a(x)\phi_{2}\phi_{1x}\\
            &=-\int_{0}^{T}\int_{0}^{1}b(t)\lambda^8s^8\chi e^{2s\varphi}\sigma^8\phi_{2x} a(x)\phi_{1x}+\int_{0}^{T}\int_{0}^{1}b(t)\lambda^8s^8(\chi e^{2s\varphi}\sigma^8)_x a(x)\phi_{2x}\phi_{1}\\
            &\quad +\int_{0}^{T}\int_{0}^{1}b(t)\lambda^8s^8(a(\chi e^{2s\varphi}\sigma^8)_x)_x	 \phi_{2}\phi_{1}.
        \end{split}
    \end{equation*}
	and, recalling that $e^{2s\varphi}$ vanishes at $0$ and $T$ and using the second equation of \eqref{optimaladj1_simp}, we have
    \begin{equation*}
        \begin{split}
            \varUpsilon_1&=\int_{0}^{T}\int_{0}^{1}\chi e^{2s\varphi}\lambda^8s^8\sigma^8 \phi_2 \phi_{1t}=-\int_{0}^{T}\int_{0}^{1}\lambda^8s^8 \chi (e^{2s\varphi}\sigma^8\phi_{2})_t \phi_{1} \\
            &=-\int_{0}^{T}\int_{0}^{1}\lambda^8s^8 \chi (e^{2s\varphi}\sigma^8)_t \phi_{2}\phi_{1} -\int_{0}^{T}\int_{0}^{1}\lambda^8s^8 \chi e^{2s\varphi}\sigma^8\phi_{1}  \phi_{2t}\\
            &=-\int_{0}^{T}\int_{0}^{1}\lambda^8s^8 \chi (e^{2s\varphi}\sigma^8)_t \phi_{2}\phi_{1} \\
            &\qquad +\int_{0}^{T}\int_{0}^{1}\lambda^8s^8 \chi e^{2s\varphi}\sigma^8\phi_{1} \left(b(t)\left(a(x)\phi_{2x}\right)_x-(B\sqrt{a}\phi_{2})_x - c_{12}\phi_1 -c_{22}\phi_2 + \rho_2\cara_{\mathcal{O}_{d}} + \overline{G}_0  \right).
        \end{split}
    \end{equation*}
    Rewriting the latter we get
	\begin{eqnarray*}
		\varUpsilon_1&=&\int_{0}^{T}\int_{0}^{1}\left(-\lambda^8s^8 \chi (e^{2s\varphi}\sigma^8)_t   - \lambda^8s^8 \chi e^{2s\varphi}\sigma^8 c_{22} \right) \phi_{2}\phi_{1}-\int_{0}^{T}\int_{0}^{1}\lambda^8s^8b (\chi e^{2s\varphi}\sigma^8)_x\phi_{1} a(x)\phi_{2x}\\
		&&-\int_{0}^{T}\int_{0}^{1}\lambda^8s^8b \chi e^{2s\varphi}\sigma^8\phi_{1x} a(x)\phi_{2x}-\int_{0}^{T}\int_{0}^{1}\lambda^8s^8 \chi e^{2s\varphi}\sigma^8\phi_1(B\sqrt{a}\phi_{2})_x\\
		&& +\int_{0}^{T}\int_{0}^{1}\lambda^8s^8 \chi e^{2s\varphi}\sigma^8\phi_1 \rho_2\cara_{\mathcal{O}_{d}} - \int_{0}^{T}\int_{0}^{1}\lambda^8s^8 \chi e^{2s\varphi}\sigma^8 c_{12} |\phi_{1}|^2 + \int_{0}^{T}\int_{0}^{1}\lambda^8s^8 \chi e^{2s\varphi}\sigma^8\phi_{1} \overline{G}_0.
	\end{eqnarray*}	
	Thus,
	{\small \begin{eqnarray*}
			\varUpsilon_1+\varUpsilon_2+\varUpsilon_3&=&\int_{0}^{T}\int_{0}^{1}\left[ -\lambda^8s^8 \chi (e^{2s\varphi}\sigma^8)_t   - \lambda^8s^8 \chi e^{2s\varphi}\sigma^8 c_{22} \right] \phi_{2}\phi_{1}-\int_{0}^{T}\int_{0}^{1}\lambda^8s^8b (\chi e^{2s\varphi}\sigma^8)_x\phi_{1} a(x)\phi_{2x}\\
			&&-\int_{0}^{T}\int_{0}^{1}\lambda^8s^8b \chi e^{2s\varphi}\sigma^8\phi_{1x} a(x)\phi_{2x}-\int_{0}^{T}\int_{0}^{1}\lambda^8s^8 \chi e^{2s\varphi}\sigma^8\phi_1(B\sqrt{a}\phi_{2})_x\\
			&&  +\int_{0}^{T}\int_{0}^{1}\lambda^8s^8 \chi e^{2s\varphi}\sigma^8\phi_1 \rho_2\cara_{\mathcal{O}_{d}} - \int_{0}^{T}\int_{0}^{1}\lambda^8s^8 \chi e^{2s\varphi}\sigma^8 c_{12} |\phi_{1}|^2 + \int_{0}^{T}\int_{0}^{1}\lambda^8s^8 \chi e^{2s\varphi}\sigma^8\phi_{1} \overline{G}_0\\
			&&-\int_{0}^{T}\int_{0}^{1}b(t)\lambda^8s^8\chi e^{2s\varphi}\sigma^8\phi_{2x} a(x)\phi_{1x}+\int_{0}^{T}\int_{0}^{1}b(t)\lambda^8s^8(\chi e^{2s\varphi}\sigma^8)_x a(x)\phi_{2x}\phi_{1}\\
			&&+\int_{0}^{T}\int_{0}^{1}b(t)\lambda^8s^8(a(\chi e^{2s\varphi}\sigma^8)_x)_x	 \phi_{2}\phi_{1}-\int_{0}^{T}\int_{0}^{1}\chi e^{2s\varphi}\lambda^8s^8\sigma^8(B\sqrt{a}\phi_1)_{x}\phi_2
	\end{eqnarray*}}
	{\small \begin{eqnarray*}
		\varUpsilon_1+\varUpsilon_2+\varUpsilon_3&=&\int_{0}^{T}\int_{0}^{1}\left[ -\lambda^8s^8 \chi (e^{2s\varphi}\sigma^8)_t   - \lambda^8s^8 \chi e^{2s\varphi}\sigma^8 c_{22}+b\lambda^8s^8(a(\chi e^{2s\varphi}\sigma^8)_x)_x	 \right] \phi_{1}\phi_{2}\\
		&&-2\int_{0}^{T}\int_{0}^{1}\lambda^8s^8b \chi e^{2s\varphi}\sigma^8\phi_{1x} a(x)\phi_{2x}-\int_{0}^{T}\int_{0}^{1}\lambda^8s^8 \chi e^{2s\varphi}\sigma^8\phi_1(B\sqrt{a}\phi_{2})_x\\
		&&-\int_{0}^{T}\int_{0}^{1}\chi e^{2s\varphi}\lambda^8s^8\sigma^8(B\sqrt{a}\phi_1)_{x}\phi_2  +\int_{0}^{T}\int_{0}^{1}\lambda^8s^8 \chi e^{2s\varphi}\sigma^8\phi_1 \rho_2\cara_{\mathcal{O}_{d}}\\
        &&- \int_{0}^{T}\int_{0}^{1}\lambda^8s^8 \chi e^{2s\varphi}\sigma^8 c_{12} |\phi_{1}|^2 + \int_{0}^{T}\int_{0}^{1}\lambda^8s^8 \chi e^{2s\varphi}\sigma^8\phi_{1} \overline{G}_0\\
		&=& \varPi_1 + \varPi_2 + \varPi_3 + \varPi_4 + \varPi_5 + \varPi_6 + \varPi_7.
    \end{eqnarray*}}	

	In order to estimate $\varPi_1$, we will analyze each term between the square brackets. First, we observe that all the terms are multiplied by $\chi$, which vanishes outside of $\widetilde{\mathcal{O}}$. Clearly,
	$$\left|-\lambda^8s^8 \chi e^{2s\varphi}\sigma^8 c_{22} \right|\leq C\chi\lambda^9s^9  \sigma^9e^{2s\varphi}.$$
		
    Since $|\sigma_x|\leq C\lambda\sigma$, $|\sigma_{xx}|\leq C\lambda^2\sigma$ and $a\in C^1(\widetilde{\mathcal{O}})$, after distributing the derivatives with respect to $x$, we see that
	$$\left|  -b(t)\lambda^8s^8(a(\chi e^{2s\varphi}\sigma^8)_x)_x \right|\leq C \chi s^9\lambda^9\sigma^9e^{2s\varphi}.$$
	
    Likewise, the relations $|\varphi_t|\leq C\sigma^2$ and $|\sigma_t|\leq C\sigma^2$ yield
	$$\left|-\lambda^8s^8 \chi (e^{2s\varphi}\sigma^8)_t  \right|\leq  C \chi s^9\lambda^9\sigma^9e^{2s\varphi}.$$
	
    As a conclusion, from Young's inequality, we have
	\begin{equation*}
        \begin{split}
            \varPi_1 & \leq C \int_{0}^{T}\int_{0}^{1}\chi s^5\lambda^9\sigma^9be^{2s\varphi} |\phi_1||\phi_2| \\
            &\leq \varepsilon \int_{0}^{T}\int_{0}^{1}s^2\lambda^2\sigma^2b^2e^{2s\varphi}|\phi_2|^2 + C_{\varepsilon} \int_{0}^{T}\int_{0}^{1}\chi^2s^{16}\lambda^{16}\sigma^{16}e^{2s\varphi}|\phi_1|^2 \\
            &\leq \varepsilon I(\phi_2)+C_{\varepsilon} \int_{0}^{T}\int_{\mathcal{O}}s^{16}\lambda^{16}\sigma^{16}e^{2s\varphi}|\phi_1|^2.
        \end{split}
	\end{equation*}
	
    Estimate for $\varPi_2:$ Analogously,  
	\begin{eqnarray*}
        \varPi_2&\leq& Cs\lambda\int_{0}^{T}\int_{0}^{1}\left( e^{s\varphi}b\sigma^{1/2}  \sqrt{a}|\phi_{2x}|\right)\left( \chi e^{s\varphi}\lambda^7s^7\sigma^{15/2}  \sqrt{a}|\phi_{1x}| \right) \\
        &\leq& \varepsilon\int_{0}^{T}\int_{0}^{1}e^{2s\varphi}s\lambda\sigma a b^2|\phi_{2x}|^2 + \frac{C_\varepsilon}{m}\int_{0}^{T}\int_{\widetilde{\mathcal{O}}}s^{15}\lambda^{15}\sigma^{15}e^{2s\varphi}b a |\phi_{1x}|^2\\
        &\leq & \varepsilon I(\phi_2)+ \Xi
	\end{eqnarray*}

    Estimate for $\Xi$:
    
    Let $\chi\in C^{\infty}(\widetilde{\mathcal{O}})$ such that satisfying $0\leq \chi \leq 1$	 and $\chi =1$ in $ \mathcal{O}'$. Then, since
    $$
        \int_{0}^{T}\frac{d}{dt}\left[ \int_{0}^{1} \chi e^{2s\varphi}\phi_1^2 dx\right]dt=0,
    $$
    using the first equation in \eqref{optimaladj1_simp}, this implies
    \begin{eqnarray*}
	   %0&=&\int_{0}^{T}\frac{d}{dt}\left[ \int_{0}^{1} \chi e^{2s\varphi}\phi_1^2 dx\right]dt=2\int_{0}^{T}\int_{0}^{1}\chi s\varphi_t e^{2s\varphi}\phi_1^2+2\int_{0}^{T}\int_{0}^{1}\chi e^{2s\varphi}\phi_1 \phi_{1t}\\
	   &&2\int_{0}^{T}\int_{0}^{1}\chi s\varphi_t e^{2s\varphi}\phi_1^2\\
	   &&+2\int_{0}^{T}\int_{0}^{1}\chi e^{2s\varphi}(-b(t)\left({a}(x)\phi_{1x}\right)_x+(B\sqrt{a}\phi_1)_{x}+c_{11}\phi_1+c_{21}\phi_2-\rho_1\cara_{_{{\mathcal{O}}_{d}}}-G_0)\phi_1 = 0.
    \end{eqnarray*}

    Integrating by parts the terms with second-order derivative and using that, from \eqref{eq:def_b_B} $(B\sqrt{a})_x$, is bounded, we get
%    \begin{eqnarray*}
%	   0&=&2\int_{0}^{T}\int_{0}^{1}\chi s\varphi_t e^{2s\varphi}\phi_1^2\\
%	   && +2\int_{0}^{T}\int_{0}^{1}(\chi e^{2s\varphi})_x a\phi_{1x} \phi_1+2\int_{0}^{T}\int_{0}^{1}\chi e^{2s\varphi} a\phi_{1x}^2+2\int_{0}^{T}\int_{0}^{1}\chi e^{2s\varphi}(B\sqrt{a})_x\phi_1^2\\
%	   &&+2\int_{0}^{T}\int_{0}^{1}\chi e^{2s\varphi}B\sqrt{a}\phi_1 \phi_{1x}+2\int_{0}^{T}\int_{0}^{1}\chi e^{2s\varphi}c_{11}\phi_1^2+2\int_{0}^{T}\int_{0}^{1}\chi e^{2s\varphi}c_{21}\phi_1\phi_2\\
%	   &&-2\int_{0}^{T}\int_{0}^{1}\chi e^{2s\varphi}\phi_1 \rho_1\cara_{_{{\mathcal{O}}_{d}}}-2\int_{0}^{T}\int_{0}^{1}\chi e^{2s\varphi}G_0 \phi_1 
%   \end{eqnarray*}
    \begin{equation*}
        \begin{split}
    	   \int_{0}^{T}\int_{0}^{1}\chi e^{2s\varphi} b a \phi_{1x}^2 &= -\int_{0}^{T}\int_{0}^{1}\chi s\varphi_t e^{2s\varphi}\phi_1^2 - \int_{0}^{T}\int_{0}^{1}(\chi e^{2s\varphi})_x b a\phi_{1x} \phi_1 %-\int_{0}^{T}\int_{0}^{1}\chi e^{2s\varphi}(B\sqrt{a})_x\phi_{1}^2 
            \\
	       &\qquad -\int_{0}^{T}\int_{0}^{1}\chi e^{2s\varphi}B\sqrt{a}\phi_1 \phi_{1x}  - \int_{0}^{T}\int_{0}^{1} \chi e^{2s\varphi} (B\sqrt{a})_x \phi_1^2
	       -\int_{0}^{T}\int_{0}^{1}\chi e^{2s\varphi}c_{11}\phi_1^2\\
           &\qquad-\int_{0}^{T}\int_{0}^{1}\chi e^{2s\varphi} c_{21} \phi_{1}\phi_{2}+\int_{0}^{T}\int_{0}^{1}\chi e^{2s\varphi}\phi_1 \rho_1\cara_{_{{\mathcal{O}}_{d}}}
	       +\int_{0}^{T}\int_{0}^{1}\chi e^{2s\varphi}G_0 \phi_1 \\
          &= -\int_{0}^{T}\int_{0}^{1} \chi \left( s\varphi_t + (B\sqrt{a})_x + c_{11} \right) e^{2s\varphi} \phi_1^2 -\int_{0}^{T}\int_{0}^{1}\chi e^{2s\varphi} c_{21} \phi_{1}\phi_{2}\\
        &\qquad -\int_{0}^{T}\int_{0}^{1} (\chi 2s\varphi_x b a + \chi_x b a + \chi B\sqrt{a})e^{2s\varphi}\phi_1 \phi_{1x}
        +\int_{0}^{T}\int_{0}^{1}\chi e^{2s\varphi}\phi_1 \rho_1\cara_{_{{\mathcal{O}}_{d}}}\\
        &\qquad +\int_{0}^{T}\int_{0}^{1}\chi e^{2s\varphi}G_0 \phi_1 
        \end{split}
    \end{equation*}

%    \begin{eqnarray*}
%        \int_{0}^{T}\int_{0}^{1}\chi e^{2s\varphi}a\phi_{1x}^2 &=&-\int_{0}^{T}\int_{0}^{1}\chi (s\varphi_t+(B\sqrt{a})_x+c_{11}) e^{2s\varphi}\phi_1^2-\int_{0}^{T}\int_{0}^{1}\chi e^{2s\varphi} c_{21} \phi_{1}\phi_{2}\\
%        &&-\int_{0}^{T}\int_{0}^{1}\chi (2s\varphi_xa+B\sqrt{a})e^{2s\varphi}\phi_1 \phi_{1x}
%        +\int_{0}^{T}\int_{0}^{1}\chi e^{2s\varphi}\phi_1 \rho_1\cara_{_{{\mathcal{O}}_{d}}}\\
%        &&+\int_{0}^{T}\int_{0}^{1}\chi e^{2s\varphi}G_0 \phi_1
%    \end{eqnarray*}

    Then, using that $c_{11}, c_{21} \in L^\infty(Q)$, that $\varphi, \chi, a(x), b(t), B(x,t)$ and their derivatives are bounded,the fact that $\phi_1 \phi_2\leq \phi_1^2+\phi_2^2$ and $2\phi_i \phi_{ix}= (\phi_i^2)_x$, and integration by parts in the term $(\phi_i^2)_x$, we deduce the estimate
    \begin{eqnarray*}
	   \Xi&\leq &C\int_{0}^{T}\int_{0}^{1}\chi s^{15}\lambda^{15} e^{2s\varphi}|\phi_1|^2+\int_{0}^{T}\int_{0}^{1}\chi s^{15}\lambda^{15} e^{2s\varphi}|\phi_1||\phi_2|+\int_{0}^{T}\int_{0}^{1}\chi s^{15}\lambda^{15} e^{2s\varphi}|\phi_1|^2 \\
	   &&+\int_{0}^{T}\int_{0}^{1}\chi s^{15}\lambda^{15} e^{2s\varphi}|\rho_1|^2\cara_{_{{\mathcal{O}}_{d}}}+\int_{0}^{T}\int_{0}^{1}s^{15}\lambda^{15}	\chi e^{2s\varphi} |G_0| |\phi_1|\\
	   &\leq &C\int_{0}^{T}\int_{0}^{1}\chi s^{15}\lambda^{15} \sigma^{15}e^{2s\varphi}|\phi_1|^2+\varepsilon I(\phi_2)+C_\varepsilon \int_{0}^{T}\int_{\widetilde{\mathcal{O}}} s^{28}\lambda^{28} \sigma^{28}e^{2s\varphi}|\phi_1|^2\\
	   &&+\int_{0}^{T}\int_{\widetilde{\mathcal{O}}}s^{15}\lambda^{15} \sigma^{15} e^{2s\varphi}  |\rho_1|^2+\int_{0}^{T}\int_{0}^{1}s^{14}\lambda^{14}\sigma^{14} e^{2s\varphi}|G_0|^2 +C\int_{0}^{T}\int_{\widetilde{\mathcal{O}}}s^{16}\lambda^{16}	\sigma^{16} e^{2s\varphi}|\phi_1|^2.
    \end{eqnarray*}			
		
%	Now, multiplying the first equation in \eqref{optimaladj1_simp} by $ \chi e^{2s\varphi}s^{16}\lambda^{16}\sigma^{16} \phi_1$ and integrating over $(0,T)\times (0,1)$, we get	
%			\begin{multline*}
%			\int_{0}^{T}\int_{0}^{1}\chi e^{2s\varphi}s^{16}\lambda^{16}\sigma^{16}b(t)\left({a}(x)\phi_{1x}\right)_x\phi_1=-\int_{0}^{T}\int_{0}^{1}\chi e^{2s\varphi}s^{16}\lambda^{16}\sigma^{16} \phi_1 \phi_{1t}\\
%			+\int_{0}^{T}\int_{0}^{1}\chi e^{2s\varphi}s^{16}\lambda^{16}\sigma^{16} c_{11}|\phi_1|^2
%			+\int_{0}^{T}\int_{0}^{1}\chi e^{2s\varphi}s^{16}\lambda^{16}\sigma^{16}(B\sqrt{a}\phi_1)_{x}\phi_1\\
%			+\int_{0}^{T}\int_{0}^{1}\chi e^{2s\varphi}s^{16}\lambda^{16}\sigma^{16} c_{21}\phi_1\phi_2
%			-\int_{0}^{T}\int_{0}^{1}\chi e^{2s\varphi}s^{16}\lambda^{16}\sigma^{16} \phi_1\rho_1\cara_{\mathcal{O}_{d}}=\varUpsilon_1+\varUpsilon_2+\varUpsilon_3+\varUpsilon_4+\varUpsilon_5
%		\end{multline*}

	Estimate for $\varPi_3$ and $\varPi_4$:
    
	We observe that
    $$(\chi e^{2s\varphi}\sigma^8)_x=\chi(2s\varphi_x\sigma^8+8\sigma^7\sigma_x)e^{2s\varphi}\leq \chi(2s\lambda \psi_x \sigma^9+8\sigma^7\lambda \sigma)e^{2s\varphi}\leq C \chi s\lambda \sigma^9e^{2s\varphi}.$$
    
%		$$\varPi_3\leq Cs\lambda\int_{0}^{T}\int_{0}^{1}e^{s\varphi}b^{1/2}\sigma^{1/2}\sqrt{a}\rho_{1x} e^{s\varphi}\chi\lambda^3s^3\sigma^{7/2}b^{-1/2}\sqrt{a}\phi_{1} $$
%		$$J_2\leq \varepsilon \int_{0}^{T}\int_{0}^{1}e^{2s\varphi}s\lambda b\sigma {a}|\rho_{1x}|^2+C_\varepsilon \int_{0}^{T}\int_{0}^{1}\chi^2e^{2s\varphi}\lambda^7s^7 \sigma^{7}b^{-1}{a}|\phi_{1}|^2   $$
%		recalling that $a\in C^1(\mathcal{O})$ and bounded and $b(t)\geq m$.
%		$$J_2\leq \varepsilon I(\rho_1)+C_\varepsilon \int_{0}^{T}\int_{\mathcal{O}}\lambda^8s^8\sigma^8 e^{2s\varphi} |\phi_{1}|^2   $$
%		Estimate for $J_3$
	
    %$$\varPi_3=-\int_{0}^{T}\int_{0}^{1}\chi e^{2s\varphi}\lambda^8s^8\sigma^8(B\sqrt{a}\phi_2)_{x}\phi_1=-\int_{0}^{T}\int_{0}^{1}\lambda^8s^8\chi e^{2s\varphi}\sigma^8\left((B\sqrt{a})_x\phi_2+B\sqrt{a}\phi_{2x} \right)\phi_1$$
    We have 
	\begin{eqnarray*}
	   \varPi_3&=& -\int_{0}^{T}\int_{0}^{1}\lambda^8s^8\chi e^{2s\varphi}\sigma^8\left((B\sqrt{a})_x\phi_2+B\sqrt{a}\phi_{2x} \right)\phi_1 \\
       %&=&-\int_{0}^{T}\int_{0}^{1}\lambda^8s^8\chi e^{2s\varphi}\sigma^8(B\sqrt{a})_x\phi_2\phi_1-\int_{0}^{T}\int_{0}^{1}\lambda^8s^8\chi e^{2s\varphi}\sigma^8B\sqrt{a}\phi_{2x}\phi_1\\
		&=&\mathcal{K}_1+\mathcal{K}_2.
	\end{eqnarray*}
    Then, using Young's inequality,
    \begin{eqnarray*}
	   \mathcal{K}_1&= &   \int_{0}^{T}\int_{0}^{1}\left(e^{s\varphi} \lambda s b \sigma|\phi_2| \right)\left( \chi e^{s\varphi}\lambda^7s^7 \sigma^{7}b^{-1}(B\sqrt{a})_x|\phi_1|  \right)\\
	   &\leq&\varepsilon \int_{0}^{T}\int_{0}^{1}e^{2s\varphi}\lambda^2 s^2 b^2 \sigma^2 |\phi_2|^2+C_\varepsilon  \int_{0}^{T}\int_{0}^{1}\chi^2 e^{2s\varphi}\lambda^{14}s^{14} \sigma^{14} b^{-2}|(B\sqrt{a})_x|^2|\phi_1|^2 \\
	   &\leq&\varepsilon I(\phi_2) +C_\varepsilon \int_{0}^{T}\int_{\widetilde{\mathcal{O}}}e^{2s\varphi}\lambda^{16}s^{16}\sigma^{16}  |\phi_{1}|^2, 
    \end{eqnarray*}
    and, analogously,
    \begin{eqnarray*}
	   \mathcal{K}_2&=&\lambda s \int_{0}^{T}\int_{0}^{1}\left( e^{s\varphi}  b \sigma^{1/2} \sqrt{a}|\phi_{2x}| \right)\left( e^{s\varphi} \lambda^7s^7 \sigma^{15/2}b^{-1}B |\phi_1|  \right)\\
	   &\leq&\varepsilon \int_{0}^{T}\int_{0}^{1}e^{2s\varphi}\lambda s b^2 \sigma a|\phi_{2x}|^2+C_\varepsilon  \int_{0}^{T}\int_{0}^{1}\chi^2e^{2s\varphi}\lambda^{15}s^{15} \sigma^{15}b^{-2}B^2|\phi_1|^2 \\
	   &\leq&\varepsilon I(\phi_2) +C_\varepsilon \int_{0}^{T}\int_{\widetilde{\mathcal{O}}}e^{2s\varphi}\lambda^{16}s^{16}\sigma^{16}  |\phi_{1}|^2.
    \end{eqnarray*}

    %--------------- HASTA AQUI ------------------

    Estimate for $\varPi_4$: Using integration by parts,	
	%$$\varPi_4=-\int_{0}^{T}\int_{0}^{1}\chi e^{2s\varphi}\lambda^8s^8\sigma^8(B\sqrt{a}\phi_1)_{x}\phi_2 = \int_{0}^{T}\int_{0}^{1}\lambda^8s^8\left(\chi e^{2s\varphi}\sigma^8\phi_2\right)_xB\sqrt{a}\phi_1$$
	\begin{eqnarray*}
	   \varPi_4&=& \int_{0}^{T}\int_{0}^{1}\lambda^8s^8\left(\chi e^{2s\varphi}\sigma^8\phi_2\right)_xB\sqrt{a}\phi_1 \\
       &=&\int_{0}^{T}\int_{0}^{1}\lambda^8s^8\left(\chi e^{2s\varphi}\sigma^8\right)_x\phi_2B\sqrt{a}\phi_1+\int_{0}^{T}\int_{0}^{1}\lambda^8s^8\chi e^{2s\varphi}\sigma^8B\sqrt{a}\phi_{2x}\phi_1\\
	   &=&\mathcal{K}_3+\mathcal{K}_2.
	\end{eqnarray*}
    
    The term $\mathcal{K}_2$ has already been estimated. For $\mathcal{K}_3$ we have
    \begin{eqnarray*}
%	   \mathcal{K}_3&\leq&C\int_{0}^{T}\int_{0}^{1}\lambda^9s^9\chi e^{2s\varphi}\sigma^9\phi_2B\sqrt{a}\phi_1\\
	   \mathcal{K}_3&\leq &   \int_{0}^{T}\int_{0}^{1}\left(e^{s\varphi} \lambda s b \sigma|\phi_2| \right)\left( \chi e^{s\varphi}\lambda^8s^8 \sigma^{8}b^{-1}(B\sqrt{a})|\phi_1|  \right)\\
	   &\leq&\varepsilon \int_{0}^{T}\int_{0}^{1}e^{2s\varphi}\lambda^2 s^2 b^2 \sigma^2 |\phi_2|^2+C_\varepsilon  \int_{0}^{T}\int_{0}^{1}\chi^2 e^{2s\varphi}\lambda^{16}s^{16} \sigma^{16} b^{-2}|B\sqrt{a}|^2|\phi_1|^2 \\
	   &\leq&\varepsilon I(\phi_2) + C_\varepsilon \int_{0}^{T}\int_{\widetilde{\mathcal{O}}}e^{2s\varphi}\lambda^{16}s^{16}\sigma^{16}  |\phi_{1}|^2.
    \end{eqnarray*}		
	
    Estimate for $\varPi_5$:
	\begin{eqnarray*}
	   \varPi_5&=&\int_{0}^{T}\int_{0}^{1}\left(e^{s\varphi}\lambda s \sigma b |\phi_1|\right)\left( e^{s\varphi} \lambda^7s^7 \chi \sigma^7 b^{-1} |\rho_2|\cara_{\mathcal{O}_{d}}  \right)\\
	   &\leq& \varepsilon \int_{0}^{T}\int_{0}^{1}e^{2s\varphi}\lambda^2 s^2 \sigma^2 b^2 |\phi_1|^2+C_\varepsilon\int_{0}^{T}\int_{0}^{1}\chi^2 e^{2s\varphi}\lambda^{14}s^{14} \sigma^{14} b^{-2} |\rho_2|^2\cara_{\mathcal{O}_{d}}\\
	   &\leq&\varepsilon I(\phi_1) +C_\varepsilon \int_{0}^{T}\int_{\widetilde{\mathcal{O}}}e^{2s\varphi}\lambda^{16}s^{16}\sigma^{16}  |\rho_2|^2\cara_{\mathcal{O}_{d}}.
	\end{eqnarray*}

    Clearly, 
    $$
        \varPi_6 \leq C \int_{0}^{T}\int_{\widetilde{\mathcal{O}}}e^{2s\varphi}\lambda^{16}s^{16}\sigma^{16}  |\phi_{1}|^2,
    $$
	and, as above,
    \begin{equation*}
        \begin{split}
            \varPi_7 & = \int_{0}^{T}\int_{0}^{1}\lambda^8s^8 \chi e^{2s\varphi}\sigma^8 |\phi_{1}| |\overline{G}_0| \\
            & \leq \varepsilon I(\varphi_1) + C_\varepsilon \int_{0}^{T}\int_{\widetilde{\mathcal{O}}} \lambda^{14} s^{14} e^{2s\varphi}\sigma^{14} \phi_{1} |\overline{G}_0|^2.
        \end{split}
    \end{equation*}
    This finishes the estimate of $\varUpsilon_1+\varUpsilon_2+\varUpsilon_3$.
    
    \noindent Estimate for $\varUpsilon_4$:
	\begin{eqnarray*}
		\varUpsilon_4&\leq&C\int_{0}^{T}\int_{0}^{1}(e^{s\varphi}\lambda s \sigma b|\phi_2|)(\chi e^{s\varphi}\lambda^7s^7\sigma^7b^{-1}|\phi_1|)\\
		&\leq& \varepsilon I(\phi_2)+C_\varepsilon \int_{0}^{T}\int_{\widetilde{\mathcal{O}}}e^{2s\varphi}\lambda^{16}s^{16}\sigma^{16} |\phi_1|^2.
	\end{eqnarray*}

	\noindent Estimate for $\varUpsilon_5$:
	\begin{eqnarray*}
		\varUpsilon_5&\leq&\int_{0}^{T}\int_{0}^{1}(e^{s\varphi}\lambda s \sigma b|\phi_2|)(\chi e^{s\varphi}\lambda^7s^7\sigma^7b^{-1}|\rho_1|
        %\chi e^{2s\varphi}\lambda^8s^8\sigma^8 \phi_2\rho_1
        \cara_{\mathcal{O}_{d}})\\
		&\leq& \varepsilon I(\phi_2)+C_\varepsilon \int_{0}^{T}\int_{\widetilde{\mathcal{O}}}e^{2s\varphi}\lambda^{16}s^{16}\sigma^{16} |\rho_1|^2.
	\end{eqnarray*}
    
	\noindent Estimate for $\varUpsilon_6$:
	\begin{eqnarray*}
		\varUpsilon_6&\leq&\int_{0}^{T}\int_{0}^{1}(e^{s\varphi}\lambda s \sigma b|\phi_2|)(\chi e^{s\varphi}\lambda^7s^7\sigma^7b^{-1}|G_0|)\\
		&\leq& \varepsilon I(\phi_2)+C_\varepsilon \int_{0}^{T}\int_{\widetilde{\mathcal{O}}}e^{2s\varphi}\lambda^{14}s^{14}\sigma^{14} |G_0|^2.
	\end{eqnarray*}
    
    This Finishes the proof of the estimate of $\varUpsilon_1+\varUpsilon_2+\varUpsilon_3+\varUpsilon_4+\varUpsilon_5+\varUpsilon_6$. 
    This allows us to estimate \eqref{eq:Y1-Y6}. Using that $c_{21}$ is bounded, we get
	\begin{multline}\label{ecphi2}
				\int_{0}^{T}\int_{0}^{1}\chi e^{2s\varphi}\lambda^8s^8\sigma^8 |\phi_2|^2\leq \varepsilon I(\phi_2)+\int_{0}^{T}\int_{0}^{1}e^{2s\varphi}\lambda^{14}s^{14}\sigma^{14} |G_0|^2
				\\ \hspace*{5cm}+C\int_{0}^{T}\int_{\widetilde{\mathcal{O}}}e^{2s\varphi}\lambda^{16}s^{16}\sigma^{16} |\phi_1|^2+C_\varepsilon \int_{0}^{T}\int_{\widetilde{\mathcal{O}}} s^{28}\lambda^{28} \sigma^{28}e^{2s\varphi}|\phi_1|^2\\
				+C_\varepsilon \int_{0}^{T}\int_{\widetilde{\mathcal{O}}}e^{2s\varphi}\lambda^{16}s^{16}\sigma^{16} \left(|\rho_1|^2+|\rho_2|^2\right).
	\end{multline}
    
    To estimate the last integral in \eqref{ecphi2}, we recall that $\phi$ and $\sigma$ are bounded in $\widetilde{\mathcal{O}}$, then
	$$\int_{0}^{T}\int_{\widetilde{\mathcal{O}}}e^{2s\varphi}\lambda^{16}s^{16}\sigma^{16} \left(|\rho_1|^2+|\rho_2|^2\right)\leq C\lambda^{16}s^{16}\int_{0}^{T}\int_{\widetilde{\mathcal{O}}}\left(|\rho_1|^2+|\rho_2|^2\right).$$
		
    Standard energy estimates for the third and fourth equations of the system \eqref{optimaladj1_simp} give
	\begin{equation}\label{ecenergia}
        \int_{0}^{T}\int_{\widetilde{\mathcal{O}}}\left(|\rho_1|^2+|\rho_2|^2\right)\leq C \left(\frac{\alpha_{1}^2}{\mu_1^2}+\frac{\alpha_{2}^2}{\mu_2^2}\right)\int_{0}^{T}\int_{\widetilde{\mathcal{O}}}|\rho_*^{-2}\phi_1|^2 + C \int_0^T \int_0^1 |G|^2 + |\overline{G}|^2.
	\end{equation}
	
    Using the definition of $\rho_*(t)$, we get $\rho_*^{-2}\leq e^{2\varphi_*}\leq e^{s\varphi}$\\
	Then, the inequality \eqref{ecenergia} becomes
	\begin{equation}\label{ecenergia2}
        \int_{0}^{T}\int_{\widetilde{\mathcal{O}}}\left(|\rho_1|^2+|\rho_2|^2\right)\leq C \left(\frac{\alpha_{1}^2}{\mu_1^2}+\frac{\alpha_{2}^2}{\mu_2^2}\right)\int_{0}^{T}\int_{\widetilde{\mathcal{O}}}e^{2s\varphi}|\phi_1|^2 + C \int_0^T \int_0^1 |G|^2 + |\overline{G}|^2.
	\end{equation}

    Combining \eqref{ecphi2} with \eqref{ecenergia2}, we obtain
	\begin{multline}\label{ecphi22}
	   \int_{0}^{T}\int_{0}^{1}\chi e^{2s\varphi}\lambda^8s^8\sigma^8 |\phi_2|^2\leq C \int_{0}^{T}\int_{0}^{1}e^{2s\varphi}\lambda^{14}s^{14}\sigma^{14}|G_0|^2  + C\int_{0}^{T}\int_{\widetilde{\mathcal{O}}}e^{2s\varphi}\lambda^{16}s^{16}\sigma^{16} |\phi_1|^2 \\
       +C_\varepsilon \int_{0}^{T}\int_{\widetilde{\mathcal{O}}} s^{28}\lambda^{28} \sigma^{28}e^{2s\varphi}|\phi_1|^2 +  C \int_0^T \int_0^1 |G|^2 + |\overline{G}|^2.
	\end{multline}

    Thus, substituting \eqref{ecphi22} in the intermediate Carleman estimate \eqref{carlemanintermedio}, we get
	\begin{multline}\label{carlemanfin}
	   I(\phi_1)+I(\phi_2)+I(\rho_1)+I(\rho_2)\leq C\int_{0}^{T}\int_{0}^{1}e^{2s\varphi}\lambda^{14}s^{14}\sigma^{14} \left(|{G}_0|^2+|{G}|^2+|\overline{G}_0|^2+|\overline{G}|^2\right)\\
	   + C \int_{0}^{T}\int_{{\mathcal{O}}}s^{28}\lambda^{28} \sigma^{28} e^{2s\varphi} |\phi_{1}|^2.
    \end{multline}
	This ends the proof.
\end{proof}

\section{Proof of Proposition \ref{propocarleman23}.}\label{appendix_b}
\begin{proof}
    The proof is very similar to the case of one equation as was done in Proposition 4 of \cite{GYL-equation-2025}.
    %, then we give only the main steps. 
    We decompose the integral $\Gamma_0(\phi_1,\phi_2,\rho_1,\rho_2)$ as
    $$\Gamma_0(\phi_1,\phi_2,\rho_1,\rho_2)=\Gamma^1_0(\phi_1,\phi_2,\rho_1,\rho_2)+\Gamma^2_0(\phi_1,\phi_2,\rho_1,\rho_2),$$
	where
	\begin{multline*}
	   \Gamma^1_0(\phi_1,\phi_2,\rho_1,\rho_2) = \int_{0}^{T/2}\int_{0}^{1} e^{2s A}(s\lambda) \zeta b^2 a (|\phi_{1x}|^2 + |\phi_{2x}|^2+|\rho_{1x}|^2 + |\rho_{2x}|^2)dxdt\\
	   +\int_{0}^{T/2}\int_{0}^{1} e^{2s A}(s\lambda)^2 {\zeta}^2 b^2 (|\phi_1|^2 + |\phi_2|^2+|\rho_1|^2 + |\rho_2|^2)dxdt
	\end{multline*}
    and
	\begin{multline*}
	   \Gamma^2_0(\phi_1,\phi_2,\rho_1,\rho_2) = \int_{T/2}^{T}\int_{0}^{1} e^{2s A}(s\lambda) \zeta b^2 a (|\phi_{1x}|^2 + |\phi_{2x}|^2+|\rho_{1x}|^2 + |\rho_{2x}|^2)dxdt\\
	   +\int_{T/2}^{T}\int_{0}^{1} e^{2s A}(s\lambda)^2 {\zeta}^2 b^2 (|\phi_1|^2 + |\phi_2|^2+|\rho_1|^2 + |\rho_2|^2)dxdt.
    \end{multline*}
            
%	We will prove that, for $i=1,2$,
%	\begin{multline*}
%	   \Gamma^i_0(\phi_1,\phi_2,\rho_1,\rho_2) \leq C \left( \int_{0}^{T}\int_{0}^{1}e^{2sA}\lambda^{14}s^{14}\zeta^{14} \left(|{G}_0|^2+|{G}|^2+|\overline{G}_0|^2+|\overline{G}|^2\right)\right. \\
%       \left.+ C \int_{0}^{T}\int_{{\mathcal{O}}}e^{2sA} s^{28}\lambda^{28} \zeta^{28} |\phi_{1}|^2\right).
%%		\Gamma^2_0(\phi_1,\phi_2,\rho_1,\rho_2) \leq C \left( \int_{0}^{T}\int_{0}^{1}e^{2sA}\lambda^{14}s^{14}\zeta^{14} \left(|{G}_0|^2+|{G}|^2+|\overline{G}_0|^2+|\overline{G}|^2\right)\right. \\
%%		\left.+ C \int_{0}^{T}\int_{{\mathcal{O}}}e^{2sA} s^{28}\lambda^{28} \zeta^{28} |\phi_{1}|^2\right)
%    \end{multline*}
    
	In order to estimate $	\Gamma^2_0(\phi_1,\phi_2,\rho_1,\rho_2)$, let us observe that $e^{2s\varphi}\sigma^n\leq Ce^{2sA}\zeta^n$ for all $(x,t)\in [0,1]\times[0,T]$ and $n\geq 0$. Since $\tau=\theta$ and $A=\varphi$ in $[T/2, T]$, the Carleman inequality \eqref{carlemansistema} directly implies that
	\begin{multline}\label{eq:est_mitad2}
        \Gamma^2_0(\phi_1,\phi_2,\rho_1,\rho_2) \leq C \left( \int_{0}^{T}\int_{0}^{1}e^{2sA}\lambda^{14}s^{14}\zeta^{14} \left(|{G}_0|^2+|{G}|^2+|\overline{G}_0|^2+|\overline{G}|^2\right)\right. \\
        \left.+ C \int_{0}^{T}\int_{{\mathcal{O}}}e^{2sA} s^{28}\lambda^{28} \zeta^{28} |\phi_{1}|^2\right).
    \end{multline}
    
	Now, we will prove an analogous estimate for $\Gamma^1_0(\phi_1,\phi_2,\rho_1,\rho_2)$ arguing as in \cite{DemarqueLimacoViana2020}. Multiplying the first, second, third and fourth equations  of \eqref{optimaladj1_simp} by $\phi_1, \phi_2, \rho_1$ and $\rho_2$, respectively, and integrating over $[0, 1]$, we obtain	
	{\small $$\hspace*{-0.6cm}		-          \int_{0}^{1}\phi_{1t}\phi_1 - \int_{0}^{1}b(t)(a(x)\phi_{1x})_x\phi_1  - \int_{0}^{1}d_1\sqrt{a}\phi_{1x}\phi_1 + \int_{0}^{1}b_{11}\phi_1^2 + \int_{0}^{1}b_{12}\phi_1\phi_2 -\int_{0}^{1}\phi_1\rho_1\cara_{\mathcal{O}_{d}}=\int_{0}^{1}G_0\phi_1, $$
	}
	{\small $$\hspace*{-0.6cm}		-  \int_{0}^{1}\phi_{2t}\phi_2 - \int_{0}^{1}b(t)(a(x)\phi_{2x})_x\phi_2  - \int_{0}^{1}d_2\sqrt{a}\phi_{2x}\phi_2 + \int_{0}^{1}b_{22}\phi_2^2  -\int_{0}^{1}\phi_2\rho_2\cara_{\mathcal{O}_{d}}=\int_{0}^{1} \overline{G}_0\phi_1, $$
	}
	{\small $$\hspace*{-0.6cm}		-\int_{0}^{1}\rho_{1t}\rho_1 - \int_{0}^{1}b(t)(a(x)\rho_{1x})_x\rho_1  - \int_{0}^{1}d_3\sqrt{a}\rho_{1x}\rho_1 + \int_{0}^{1}b_{11}\rho_1^2  +\int_{0}^{1}\rho_*^{-2}d\phi_1\rho_1=\int_{0}^{1} {G}\rho_1, $$
	}
	{\small $$\hspace*{-0.6cm}		-\int_{0}^{1}\rho_{2t}\rho_2 - \int_{0}^{1}b(t)(a(x)\rho_{2x})_x\rho_2  - \int_{0}^{1}d_4\sqrt{a}\rho_{2x}\rho_2+ \int_{0}^{1}b_{12}\rho_1\rho_2 + \int_{0}^{1}b_{22}\rho_2^2  =\int_{0}^{1} \overline{G}\rho_2. $$
	}
    
	We use the notations $F_1=G_0$, $F_2=\overline{G}_0$, $F_3=G$ and $F_4=\overline{G}$. Using energy estimates as in Appendix A. of \cite{GYL-equation-2025},
    %\ref{apenA1} of the first chapter, 
    we have, for some $C>0$,
	\begin{multline}\label{ecsiste310}
		-\frac{1}{2} \sum_{i=1}^2 \frac{d}{dt}\left( \|\phi_i\|^2_{L^2(0,1)}+\|\rho_i\|^2_{L^2(0,1)}\right)-C \sum_{i=1}^2 \left(\|\phi_i\|^2_{L^2(0,1)}+\|\rho_i\|^2_{L^2(0,1)}\right)\\
		+ \sum_{i=1}^2 \left(\|\sqrt{a} \phi_i\|^2_{L^2(0,1)}+\|\sqrt{a} \rho_i\|^2_{L^2(0,1)}\right)\leq \sum_{i=1}^4 \|F_i\|^2_{L^2(0,1)},
	\end{multline}
	which implies
	$$- \sum_{i=1}^2 \frac{d}{dt}\left[e^{Ct}\left(\|\phi_i\|^2_{L^2(0,1)}+\|\rho_i\|^2_{L^2(0,1)}\right)\right]\leq e^{Ct}\left( \sum_{i=1}^4 \|F_i\|^2_{L^2(0,1)} \right)$$
	
    Integrating from a $t\in [0,T/2]$ to $t+T/4$, we get
%	$$-\int_{t}^{t+T/4}\frac{d}{dt}\left[e^{Ct}\left(\|y\|^2_{L^2(0,1)}+\|z\|^2_{L^2(0,1)}\right)\right]dt\leq\int_{t}^{t+T/4} e^{Ct}\left(\|F_1\|^2_{L^2(0,1)}+\|F_2\|^2_{L^2(0,1)}\right)dt$$
	\begin{multline*}
		e^{Ct} \sum_{i=1}^2 \left(\|\phi_i\|^2_{L^2(0,1)}+\|\rho_i\|^2_{L^2(0,1)}\right)-e^{C(t+T/4)} \sum_{i=1}^2 \left(\|\phi_i(t+T/4)\|^2_{L^2(0,1)}+\|\rho_i(t+T/4)\|^2_{L^2(0,1)}\right)\\ \leq\int_{t}^{t+T/4} e^{Ct} \sum_{i=1}^2 \|F_i\|^2_{L^2(0,1)}dt.
	\end{multline*}
	Thus, 
    \begin{multline*}
        \sum_{i=1}^2 \left( \|\phi_i\|^2_{L^2(0,1)}+\|\rho_i\|^2_{L^2(0,1)} \right) \leq e^{CT}\int_{0}^{3T/4} \sum_{i=1}^4 \|F_i\|^2_{L^2(0,1)}dt\\
		+e^{3CT/4} \sum_{i=1}^2 \left(\|\phi_i(t+T/4)\|^2_{L^2(0,1)}+\|\rho_i(t+T/4)\|^2_{L^2(0,1)}\right).
	\end{multline*}
    
	Integrating again from 0 to $T/2$, we have	
%	\begin{multline*}
%				\int_{0}^{T/2}\left(\|y\|^2_{L^2(0,1)}+\|z\|^2_{L^2(0,1)}\right) \leq \int_{0}^{T/2}\left[e^{CT}\int_{0}^{3T/4} \left(\|F_1(t)\|^2_{L^2(0,1)}+\|F_2(t)\|^2_{L^2(0,1)}\right)dt\right]dt\\
%				+\int_{0}^{T/2}e^{3CT/4}\left(\|y(t+T/4)\|^2_{L^2(0,1)}+\|z(t+T/4)\|^2_{L^2(0,1)}\right)
%	\end{multline*}
%	\vspace*{-1.0cm}
	\begin{multline}\label{ecsis311}
		\int_{0}^{T/2} \sum_{i=1}^2 \left(\|\phi_i\|^2_{L^2(0,1)}+\|\rho_i\|^2_{L^2(0,1)}\right) \leq \frac{T}{2}e^{CT}\int_{0}^{3T/4} \sum_{i=1}^4 \|F_i\|^2_{L^2(0,1)}dt\\
		+e^{3CT/4}\int_{T/4}^{3T/4} \sum_{i=1}^2 \left(\|\phi_i(t)\|^2_{L^2(0,1)}+\|\rho_i(t)\|^2_{L^2(0,1)}\right).
    \end{multline}
    
	Now, integrating inequality \eqref{ecsiste310} over $[0,t]$, for $t\in [0, T]$,
	\begin{multline}\label{ecsiste312}
		\int_{0}^{t} \sum_{i=1}^2 \left(\|\sqrt{a}\phi_i\|^2_{L^2(0,1)}+\|\sqrt{a}\rho_i\|^2_{L^2(0,1)}\right) \leq \frac{1}{2} \sum_{i=1}^2 \left(\|\phi_i\|^2_{L^2(0,1)}+\|\rho_i\|^2_{L^2(0,1)}\right)\\
		+C\left[ \int_{0}^{t} \sum_{i=1}^2 \left(\|\phi_i\|^2_{L^2(0,1)}+\|\rho_i\|^2_{L^2(0,1)}\right)+ \int_{0}^{t} \sum_{i=1}^4 \|F_i\|^2_{L^2(0,1)}\right].
	\end{multline}
	
    Let us consider the double integral 	$$I=\int_{T/2}^{3T/4}\int_{0}^{t} \sum_{i=1}^2\left(\|\sqrt{a}\phi_i(\tau)\|^2_{L^2(0,1)}+\|\sqrt{a}\rho_i(\tau)\|^2_{L^2(0,1)}\right)d\tau dt$$
	and define $$J(t)=\int_{0}^{t} \sum_{i=1}^2 \left(\|\sqrt{a}\phi_i(\tau)\|^2_{L^2(0,1)}+\|\sqrt{a}\rho_i(\tau)\|^2_{L^2(0,1)}\right)d\tau.$$
	
    Since  $J(t)$ is non-decreasing, we have that 
    %. For all $t\in[T/2,3T/4] $, we have $J(t)\geq J(T/2)$. Similarly, for $t\in[T/2,3T/4] $, we have  
    $J(T/2) \leq J(t)\leq  J(3T/4)$, for $t \in [T/2,3T/4]$.
	Then we can bound the integral $I$ as follows:\\
%	For $t\in[T/2,3T/4] $, we have $J(t)\geq J(T/2)$. Thus,
	$$I=\int_{T/2}^{3T/4}J(t)dt\geq \int_{T/2}^{3T/4}J(T/2)dt=\frac{T}{4}J(T/2).$$
    
	That is,
    \begin{equation}
        \label{eq:est_I_1}
        I\geq \frac{T}{4}\int_{0}^{T/2} \sum_{i=1}^2\left(\|\sqrt{a}\phi_i(\tau)\|^2_{L^2(0,1)}+\|\sqrt{a}\rho_i(\tau)\|^2_{L^2(0,1)}\right)d\tau.  
    \end{equation}
    
	Similarly, using that $J(t)\leq  J(3T/4)$, we also have
	$$I\leq \frac{T}{4}\int_{0}^{3T/4} \sum_{i=1}^2\left(\|\sqrt{a}\phi_i(\tau)\|^2_{L^2(0,1)}+\|\sqrt{a}\rho_i(\tau)\|^2_{L^2(0,1)}\right)d\tau.$$
    
	%Combining both estimates, we obtain:
	%{\small $$
    %\hspace*{-0.5cm}
    %\int_{0}^{T/2} \sum_{i=1}^2 \left(\|\sqrt{a}\phi_i(\tau)\|^2_{L^2(0,1)}+\|\sqrt{a}\rho_i(\tau)\|^2_{L^2(0,1)}\right)d\tau \leq \frac{4}{T}I\leq \int_{0}^{3T/4} \sum_{i=1}^2\left(\|\sqrt{a}\phi_i(\tau)\|^2_{L^2(0,1)}+\|\sqrt{a}\rho_i(\tau)\|^2_{L^2(0,1)}\right)d\tau.$$}
	
    On the other hand, integrating inequality \eqref{ecsiste312} from $T/2$ to $3T/4$, we get 
	\begin{multline}\label{eq:estimate_T2_3T4}
		\int_{T/2}^{3T/4}\int_{0}^{t} \sum_{i=1}^2 \left(\|\sqrt{a}\phi_i(\tau)\|^2_{L^2(0,1)}+\|\sqrt{a}\rho_i(\tau)\|^2_{L^2(0,1)}\right)d\tau dt \leq \frac{1}{2}\int_{T/2}^{3T/4} \sum_{i=1}^2 \left(\|\phi_i\|^2_{L^2(0,1)}+\|\rho_i\|^2_{L^2(0,1)}\right) dt\\
		+ C\left[\int_{T/2}^{3T/4} \int_{0}^{t} \sum_{i=1}^2 \left(\|\phi_i(\tau)\|^2_{L^2(0,1)}+\|\rho_i(\tau)\|^2_{L^2(0,1)}\right)d\tau dt + \int_{T/2}^{3T/4}\int_{0}^{t} \sum_{i=1}^4 \|F_i\|^2_{L^2(0,1)} d\tau dt \right].
	\end{multline}
    
	Using \eqref{ecsis311} we estimate the following term of the right-hand side of \eqref{eq:estimate_T2_3T4}, 
    \begin{equation*}
        \begin{split}
            &\int_{T/2}^{3T/4} \int_{0}^{t} \sum_{i=1}^2 \left(\|\phi_i(\tau)\|^2_{L^2(0,1)}+\|\rho_i(\tau)\|^2_{L^2(0,1)}\right)d\tau dt \\
            &\qquad \leq \frac{T}{4} \left( \int_{0}^{T/2} + \int_{T/2}^{3T/4} \right) \sum_{i=1}^2 \left(\|\phi_i(\tau)\|^2_{L^2(0,1)}+\|\rho_i(\tau)\|^2_{L^2(0,1)}\right)d\tau \\
            &\qquad \leq C \left[ \int_{0}^{3T/4} \sum_{i=1}^4 \|F_i\|^2_{L^2(0,1)}dt
		      + \int_{T/4}^{3T/4} \sum_{i=1}^2 \left(\|\phi_i(t)\|^2_{L^2(0,1)}+\|\rho_i(t)\|^2_{L^2(0,1)}\right) \right] \\
            &\qquad \quad + \frac{T}{4} \int_{T/2}^{3T/4} \sum_{i=1}^2 \left(\|\phi_i(\tau)\|^2_{L^2(0,1)}+\|\rho_i(\tau)\|^2_{L^2(0,1)}\right)d\tau.
        \end{split}
    \end{equation*}

    Substituting the latter in \eqref{eq:estimate_T2_3T4}, we get, for some constant $C>0$,
    \begin{multline}\label{eq:est_T2_3T4_1}
		\int_{T/2}^{3T/4}\int_{0}^{t} \sum_{i=1}^2\left(\|\sqrt{a}\phi_i(\tau)\|^2_{L^2(0,1)}+\|\sqrt{a}\rho_i(\tau)\|^2_{L^2(0,1)}\right)d\tau dt\leq\\
		C\left[\int_{T/2}^{3T/4} \sum_{i=1}^2 \left(\|\phi_i(t)\|^2_{L^2(0,1)}+\|\rho_i(t)\|^2_{L^2(0,1)}\right)dt+ \int_{T/2}^{3T/4}\int_{0}^{t} \sum_{i=1}^4 \|F_i(\tau)\|^2_{L^2(0,1)} d\tau dt\right].
	\end{multline}
	
    Using the bound \eqref{eq:est_I_1}, the inequality (\ref{eq:est_T2_3T4_1}) becomes, for another constant $C>0$,
	\begin{multline}\label{ecsis314}
        \int_{0}^{T/2} \sum_{i=1}^2 \left(\|\sqrt{a}\phi_i(\tau)\|^2_{L^2(0,1)}+\|\sqrt{a}\rho_i(\tau)\|^2_{L^2(0,1)}\right)d\tau dt\leq\\
		C\left[\int_{T/4}^{3T/4} \sum_{i=1}^2\left(\|\phi_i(t)\|^2_{L^2(0,1)}+\|\rho_i(t)\|^2_{L^2(0,1)}\right)dt + \int_{0}^{3T/4} \sum_{i=1}^4 \|F_i(t)\|^2_{L^2(0,1)} dt\right].
	\end{multline}
	
    Finally, we observe that $e^{2sA}(s\lambda \zeta)^n$ and $e^{2sA}(s\lambda \sigma)^n$ are bounded in $[0, T/2]$ and $[T/4, 3T/4]$ respectively, for all $n\in \mathbb{Z}$. Hence, estimates (\ref{ecsis311}) and (\ref{ecsis314}) 
    %and Carleman's Inequality \eqref{carlemansistema} 
    imply
	\begin{eqnarray*}
	   \Gamma^1_0(\phi_1,\phi_2,\rho_1,\rho_2) &=&\int_{0}^{T/2}\int_{0}^{1} e^{2s A} \sum_{i=1}^2 \left((s\lambda) \zeta b^2 a (|\phi_{ix}|^2 + |\rho_{ix}|^2) + (s\lambda)^2 {\zeta}^2 b^2 (|\phi_i|^2 + |\rho_i|^2)\right)dxdt\\
        &\leq&C\int_{0}^{T/2}\int_{0}^{1} \sum_{i=1}^2 \left(  a (|\phi_{ix}|^2 + |\rho_{ix}|^2) +  (|\phi_i|^2 + |\rho_i|^2) \right)dxdt\\
		&\leq&C\left[\int_{T/4}^{3T/4}\int_{0}^{1} \sum_{i=1}^2\left(|\phi_i|^2 + |\rho_i|^2\right)dxdt+ \int_{0}^{3T/4}\int_{0}^{1} \sum_{i=1}^4 |F_i|^2 dxdt\right]\\
        &\leq& C\left[ \sum_{i=1}^2 \int_{T/4}^{3T/4}\int_{0}^{1}e^{2s \varphi} (s\lambda)^2 {\sigma}^2 b^2\left(|\phi_i|^2 + |\rho_i|^2\right)dxdt\right.\\
		&&\qquad + \left. \sum_{i=1}^4 \int_{0}^{3T/4}\int_{0}^{1}e^{2s \varphi} (s\lambda)^{14} {\zeta}^{14} |F_i|^2 dxdt\right].
	\end{eqnarray*}
%	\vspace*{-0.6cm}
%	\begin{multline*}
%				\leq C\left[\int_{T/4}^{3T/4}\int_{0}^{1}e^{2s \varphi} (s\lambda)^2 {\sigma}^2 b^2\left(|y|^2+|z|^2\right)dxdt\right.+\\
%				\left.\int_{0}^{3T/4}\int_{0}^{1}e^{2s \varphi} (s\lambda)^4 {\sigma}^4 \left(|F_1|^2+|F_2|^2\right) dxdt\right]
%	\end{multline*}

    Finally, by Carleman's inequality (\ref{carlemansistema}), the last inequality gives 
    $$\Gamma^1_0(\phi_1,\phi_2,\rho_1,\rho_2) \leq C\left[ 
	\int_{0}^{T}\int_{\mathcal{O}}e^{2s A} (s\lambda)^{28} {\zeta}^{28} |\phi_1|^2  dxdt + \sum_{i=1}^4 \int_{0}^{T}\int_{0}^{1}e^{2s A}(s\lambda)^{14} {\zeta}^{14} |F_i|^2 dxdt \right],$$
    which, together with \eqref{eq:est_mitad2}, implies the result.    
\end{proof}

\bibliographystyle{abbrv}
\bibliography{referencias1}

\end{document}